\documentclass[reqno]{amsart}

\usepackage{enumerate, amsmath, amsfonts, amssymb, amsthm, wasysym, graphics, graphicx, xcolor, url, hyperref, hypcap, a4wide, pdflscape, multido, xargs, colortbl, multicol, multirow, calc, shuffle, hvfloat}
\hypersetup{colorlinks=true, citecolor=PineGreen, linkcolor=PineGreen}
\usepackage{tikz}\usetikzlibrary{trees,snakes,shapes,arrows,matrix,calc}
\usepackage{comment}
\usepackage{etex}
\usepackage{ulem}
\usepackage[noabbrev,capitalise]{cleveref}
\graphicspath{{figures/}}
\makeatletter
\def\input@path{{figures/}}
\makeatother

\title{Lattices of acyclic pipe dreams}

\author[N.~Bergeron]{Nantel Bergeron} 
\address[N.~Bergeron]{Department of Mathematics and Statistics, York University, Toronto}
\email{bergeron@yorku.ca}
\urladdr{http://bergeron.mathstats.yorku.ca}

\author[N.~Cartier]{No\'emie Cartier} 
\address[N.~Cartier]{LISN, Université Paris Saclay}
\email{noemie.cartier@lri.fr}
\urladdr{https://www.lisn.upsaclay.fr/~cartier/}

\author[C.~Ceballos]{Cesar Ceballos}
\address[C.~Ceballos]{Institute of Geometry, Technische Universit\"at Graz}
\email{cesar.ceballos@tugraz.at}
\urladdr{http://www.geometrie.tugraz.at/ceballos/}

\author[V.~Pilaud]{Vincent Pilaud}
\address[V.~Pilaud]{Universitat de Barcelona \& Centre de Recerca Matemàtica, Barcelona, Spain}
\email{vincent.pilaud@ub.edu}
\urladdr{https://www.ub.edu/comb/vincentpilaud/}

\thanks{
NB was supported by NSERC and York Research Chair in Applied Algebra.
CC was supported by the Austrian Science Fund FWF (Project P 33278).
CC, NC \& VP were supported by the Austrian\,--\,French project PAGCAP (ANR~21\,CE48\,0020 \& FWF I 5788).
NC \& VP were also supported by the French project CHARMS (ANR~19\,CE40\,0017).
VP was also supported by the Spanish project PID2022-137283NB-C21 of MCIN/AEI/10.13039/501100011033 / FEDER, UE, by the Spanish--German project COMPOTE (AEI PCI2024-155081-2 \& DFG 541393733), by the Severo Ochoa and María de Maeztu Program for Centers and Units of Excellence in R\&D (CEX2020-001084-M), by the Departament de Recerca i Universitats de la Generalitat de Catalunya (2021 SGR 00697).
VP was a CNRS researcher at \'Ecole Polytechnique when this work was done.
\\
\hspace*{.5cm} 
Part of the material of this paper was announced in an extended abstract of the conference FPSAC'23~\cite{Cartier-FPSAC}.
}

\newtheorem{theorem}{Theorem}[section]
\newtheorem{theoremA}{Theorem}

\newtheorem{corollary}[theorem]{Corollary}
\newtheorem{proposition}[theorem]{Proposition}
\newtheorem{lemma}[theorem]{Lemma}
\newtheorem{conjecture}[theorem]{Conjecture}
\crefname{conjecture}{Conjecture}{Conjectures}
\newtheorem{conjectureA}{Conjecture}

\crefname{conjectureA}{Conjecture}{Conjectures}

\theoremstyle{definition}
\newtheorem{definition}[theorem]{Definition}
\newtheorem{example}[theorem]{Example}
\newtheorem{remark}[theorem]{Remark}

\newtheorem{notation}[theorem]{Notation}

\newcommand{\R}{\mathbb{R}} 
\renewcommand{\b}[1]{\boldsymbol{#1}} 
\newcommand{\cal}[1]{\mathcal{#1}} 

\newcommand{\set}[2]{\left\{ #1 \;\middle|\; #2 \right\}} 
\newcommand{\bigset}[2]{\big\{ #1 \;|\; #2 \big\}} 
\newcommand{\ssm}{\smallsetminus} 
\newcommand{\eqdef}{\mbox{\,\raisebox{0.2ex}{\scriptsize\ensuremath{\mathrm:}}\ensuremath{=}\,}} 

\DeclareMathOperator{\cone}{cone} 
\DeclareMathOperator{\Inv}{Inv} 
\DeclareMathOperator{\Ninv}{Ninv} 
\DeclareMathOperator{\DemazureProduct}{Dem} 

\newcommand{\ie}{\textit{i.e.}~} 
\definecolor{PineGreen}{RGB}{2,120,120} 
\definecolor{darkgreen}{RGB}{57,181,74} 
\newcommand{\blue}[1]{{\color{blue} #1}} 
\newcommand{\red}[1]{{\color{red} #1}} 
\newcommand{\defn}[1]{\textbf{\textsf{\color{PineGreen} #1}}} 
\usepackage{todonotes}

\newcommand{\fS}{\mathfrak{S}} 
\newcommand{\fR}{\mathfrak{R}} 

\newcommand{\boxsize}{.35}
\newlength{\verticalOffset}
\newlength{\verticalShift}
\newcounter{length}
\newcommand{\length}[1]{%
	\setcounter{length}{0}%
	\foreach \x in {#1} {%
		\stepcounter{length}%
	}%
}

\newcommand{\pipeDreamMonoColor}[3]{
	\length{#3}%
	\begin{tikzpicture}[baseline = \value{length}*\verticalShift+\verticalOffset, scale=1]
		\coordinate (origin) at (0,0);
		\newcount{\y} \y=0
		\newcount{\x}
		\foreach \line in {#3} {
			\x=0
			\foreach \t in \line {
				\coordinate (W) at ($ (origin) + ( \boxsize * \x , -\boxsize * \y ) + ( 0      , \boxsize / 2 ) $);
				\coordinate (E) at ($ (origin) + ( \boxsize * \x , -\boxsize * \y ) + ( \boxsize     , \boxsize / 2 ) $);
				\coordinate (N) at ($ (origin) + ( \boxsize * \x , -\boxsize * \y ) + ( \boxsize / 2 , \boxsize     ) $);
				\coordinate (S) at ($ (origin) + ( \boxsize * \x , -\boxsize * \y ) + ( \boxsize / 2 , 0 ) $);
				\coordinate (C) at ($ (origin) + ( \boxsize * \x , -\boxsize * \y ) + ( \boxsize / 2 , \boxsize / 2 ) $);
				\ifthenelse{\equal{\t}{e}}{
					\draw[rounded corners=\boxsize * 8, color=#1, thick] (W) -- (C) -- (N);
					\draw[rounded corners=\boxsize * 8, color=#1, thick] (S) -- (C) -- (E);			
				}{
        				\ifthenelse{\equal{\t}{c}}{
        					\draw[color=#1, thick] (W) -- (E);
        					\draw[color=#1, thick] (S) -- (N);
        				}{
        				\ifthenelse{\equal{\t}{t}}{
        					\draw[rounded corners=\boxsize * 8, color=#2] (W) -- (C) -- (N);
        					\draw[rounded corners=\boxsize * 8, color=#1, thick] (S) -- (C) -- (E);			
        				}{
        				\ifthenelse{\equal{\t}{b}}{
        					\draw[rounded corners=\boxsize * 8, color=#1, thick] (W) -- (C) -- (N);
        					\draw[rounded corners=\boxsize * 8, color=#2] (S) -- (C) -- (E);			
        				}{
        				\ifthenelse{\equal{\t}{tb}}{
        					\draw[rounded corners=\boxsize * 8, color=#2] (W) -- (C) -- (N);
        					\draw[rounded corners=\boxsize * 8, color=#2] (S) -- (C) -- (E);			
        				}{
        				\ifthenelse{\equal{\t}{n}}{}{\node at (C) {$\small \t$};}}}}}}
        				\global\advance\x by 1
			}
			\global\advance\y by 1
		}
	\end{tikzpicture}%
}

\newcommand{\pipeDreamBiColor}[4]{
	\length{#4}%
	\begin{tikzpicture}[baseline = \value{length}*\verticalShift+\verticalOffset, scale=1]
		\coordinate (origin) at (0,0);
		\newcount{\y} \y=0
		\newcount{\x}
		\foreach \line in {#4} {
			\x=0
			\foreach \t/\colorW/\colorS in \line {
				\coordinate (W) at ($ (origin) + ( \boxsize * \x , -\boxsize * \y ) + ( 0      , \boxsize / 2 ) $);
				\coordinate (E) at ($ (origin) + ( \boxsize * \x , -\boxsize * \y ) + ( \boxsize     , \boxsize / 2 ) $);
				\coordinate (N) at ($ (origin) + ( \boxsize * \x , -\boxsize * \y ) + ( \boxsize / 2 , \boxsize     ) $);
				\coordinate (S) at ($ (origin) + ( \boxsize * \x , -\boxsize * \y ) + ( \boxsize / 2 , 0 ) $);
				\coordinate (C) at ($ (origin) + ( \boxsize * \x , -\boxsize * \y ) + ( \boxsize / 2 , \boxsize / 2 ) $);
				\ifthenelse{\equal{\t}{e}}{
					\ifthenelse{\equal{\colorW}{l}}{\draw[rounded corners=\boxsize * 8, color=#1, thick] (W) -- (C) -- (N);}{}
					\ifthenelse{\equal{\colorW}{r}}{\draw[rounded corners=\boxsize * 8, color=#2, thick] (W) -- (C) -- (N);}{}
					\ifthenelse{\equal{\colorW}{b}}{\draw[rounded corners=\boxsize * 8, color=#3] (W) -- (C) -- (N);}{}
					\ifthenelse{\equal{\colorS}{l}}{\draw[rounded corners=\boxsize * 8, color=#1, thick] (S) -- (C) -- (E);}{}
					\ifthenelse{\equal{\colorS}{r}}{\draw[rounded corners=\boxsize * 8, color=#2, thick] (S) -- (C) -- (E);}{}
					\ifthenelse{\equal{\colorS}{b}}{\draw[rounded corners=\boxsize * 8, color=#3] (S) -- (C) -- (E);}{}
				}{
				\ifthenelse{\equal{\t}{c}}{
					\ifthenelse{\equal{\colorW}{l}}{\draw[color=#1, thick] (W) -- (E);}{}
					\ifthenelse{\equal{\colorW}{r}}{\draw[color=#2, thick] (W) -- (E);}{}
					\ifthenelse{\equal{\colorS}{l}}{\draw[color=#1, thick] (S) -- (N);}{}
					\ifthenelse{\equal{\colorS}{r}}{\draw[color=#2, thick] (S) -- (N);}{}
				}{
				\ifthenelse{\equal{\t}{n}}{}{\node at (C) {$\small \t$};}}}
				\global\advance\x by 1
			}
			\global\advance\y by 1
		}
	\end{tikzpicture}%
}

\newcommand{\pipeDreamTriColor}[5]{
	\length{#5}%
	\begin{tikzpicture}[baseline = \value{length}*\verticalShift+\verticalOffset, scale=1]
		\coordinate (origin) at (0,0);
		\newcount{\y} \y=0
		\newcount{\x}
		\foreach \line in {#5} {
			\x=0
			\foreach \t/\colorW/\colorS in \line {
				\coordinate (W) at ($ (origin) + ( \boxsize * \x , -\boxsize * \y ) + ( 0      , \boxsize / 2 ) $);
				\coordinate (E) at ($ (origin) + ( \boxsize * \x , -\boxsize * \y ) + ( \boxsize     , \boxsize / 2 ) $);
				\coordinate (N) at ($ (origin) + ( \boxsize * \x , -\boxsize * \y ) + ( \boxsize / 2 , \boxsize     ) $);
				\coordinate (S) at ($ (origin) + ( \boxsize * \x , -\boxsize * \y ) + ( \boxsize / 2 , 0 ) $);
				\coordinate (C) at ($ (origin) + ( \boxsize * \x , -\boxsize * \y ) + ( \boxsize / 2 , \boxsize / 2 ) $);
				\ifthenelse{\equal{\t}{e}}{
					\ifthenelse{\equal{\colorW}{l}}{\draw[rounded corners=\boxsize * 8, color=#1, thick] (W) -- (C) -- (N);}{}
					\ifthenelse{\equal{\colorW}{m}}{\draw[rounded corners=\boxsize * 8, color=#2, thick] (W) -- (C) -- (N);}{}
					\ifthenelse{\equal{\colorW}{r}}{\draw[rounded corners=\boxsize * 8, color=#3, thick] (W) -- (C) -- (N);}{}
					\ifthenelse{\equal{\colorW}{b}}{\draw[rounded corners=\boxsize * 8, color=#4] (W) -- (C) -- (N);}{}
					\ifthenelse{\equal{\colorS}{l}}{\draw[rounded corners=\boxsize * 8, color=#1, thick] (S) -- (C) -- (E);}{}
					\ifthenelse{\equal{\colorS}{m}}{\draw[rounded corners=\boxsize * 8, color=#2, thick] (S) -- (C) -- (E);}{}
					\ifthenelse{\equal{\colorS}{r}}{\draw[rounded corners=\boxsize * 8, color=#3, thick] (S) -- (C) -- (E);}{}
					\ifthenelse{\equal{\colorS}{b}}{\draw[rounded corners=\boxsize * 8, color=#4] (S) -- (C) -- (E);}{}
				}{
				\ifthenelse{\equal{\t}{c}}{
					\ifthenelse{\equal{\colorW}{l}}{\draw[color=#1, thick] (W) -- (E);}{}
					\ifthenelse{\equal{\colorW}{m}}{\draw[color=#2, thick] (W) -- (E);}{}
					\ifthenelse{\equal{\colorW}{r}}{\draw[color=#3, thick] (W) -- (E);}{}
					\ifthenelse{\equal{\colorS}{l}}{\draw[color=#1, thick] (S) -- (N);}{}
					\ifthenelse{\equal{\colorS}{m}}{\draw[color=#2, thick] (S) -- (N);}{}
					\ifthenelse{\equal{\colorS}{r}}{\draw[color=#3, thick] (S) -- (N);}{}
				}{\ifthenelse{\equal{\t}{n}}{}{\node at (C) {$\small \t$};}}}
				\global\advance\x by 1
			}
			\global\advance\y by 1
		}
	\end{tikzpicture}%
}

\newcommand{\cross}[1][black]{\raisebox{-.15cm}{\includegraphics[scale=.9]{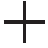}}}

\newcommand{\NScross}[1][black]{\raisebox{-.15cm}{\includegraphics[scale=.9]{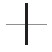}}}
\newcommand{\WEcross}[1][black]{\raisebox{-.15cm}{\includegraphics[scale=.9]{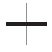}}}
\newcommand{\elbow}[1][black]{\raisebox{-.15cm}{\includegraphics[scale=.9]{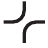}}}

\newcommand{\SEelbow}[1][black]{\raisebox{-.15cm}{\includegraphics[scale=.9]{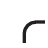}}}
\newcommand{\WNelbow}[1][black]{\raisebox{-.15cm}{\includegraphics[scale=.9]{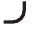}}}

\newcommand{\pipeDreams}{\Pi} 
\newcommand{\reversingPipeDreams}{\Omega} 
\newcommand{\contact}{^\#} 
\newcommand{\duality}{^\star} 
\newcommand{\acyclicPipeDreams}{\Sigma} 
\newcommand{\linearExtensions}{\mathcal{L}} 
\newcommand{\strongLinearExtensions}{\mathcal{L}^\star} 
\newcommand{\noninversions}[2]{\mathsf{ninv}(#1,#2)} 
\newcommand{\acyclicOrientations}{\Omega} 
\newcommand{\insertion}[2]{\mathsf{pd}(#1,#2)} 
\newcommand{\recoils}[2]{\mathsf{rec}(#1,#2)} 
\newcommand{\canopy}[1]{\mathsf{can}(#1)} 
\newcommand{\greedyPipeDream}{P^{\overleftarrow{gr}}} 
\newcommand{\antiGreedyPipeDream}{P^{\overrightarrow{gr}}} 

\newcommand{\wo}{\omega_\circ} 
\newcommand{\subwordComplex}{\mathcal{SC}} 
\newcommand{\Roots}{\mathrm{R}} 
\newcommand{\rootFunction}[2]{\mathrm{r}_{#1}(#2)} 
\newcommand{\subwordFacets}{\mathcal{F}} 
\newcommand{\subwordAcyclicFacets}{\mathcal{F}^\bullet} 
\newcommand{\subwordStronglyAcyclicFacets}{\mathcal{F}^\star} 
\newcommand{\greedyFacet}{I^{\overleftarrow{gr}}} 
\newcommand{\antiGreedyFacet}{I^{\overrightarrow{gr}}} 
\newcommand{\sweepingAlgorithm}{\mathsf{sweep}} 
\newcommand{\brickPolyhedron}{\mathsf{Brick}} 

\newcommand{\meet}{\wedge} 
\newcommand{\join}{\vee} 
\newcommand{\less}{\vartriangleleft} 
\newcommand{\more}{\vartriangleright} 
\newcommand{\contactLess}[1]{\less_{#1}} 
\newcommand{\contactMore}[1]{\more_{#1}} 
\newcommand{\projDown}{\pi_\downarrow} 
\newcommand{\projUp}{\pi^\uparrow} 

\begin{document}

\begin{abstract}
We show that for any permutation~$\omega$, the increasing flip graph on acyclic pipe dreams with exiting permutation~$\omega$ is a lattice quotient of the interval~$[e,\omega]$ of the weak order.
We then discuss conjectural generalizations of this result to acyclic facets of subword complexes on arbitrary finite Coxeter groups.
\end{abstract}


\maketitle

\tableofcontents



\pagebreak
\section{Introduction}
\label{sec:introduction}

The weak order is the lattice on permutations of~$[n]$ whose cover relations correspond to switching pairs of consecutive values in permutations.
The Tamari lattice is the lattice on binary trees with~$n$ internal nodes whose cover relations correspond to right rotations in binary trees.
The Tamari lattice is known to be the lattice quotient of the weak order by the sylvester congruence, defined as the equivalence relation on permutations of~$[n]$ whose equivalence classes are the sets of linear extensions of binary trees (labeled in inorder and oriented towards their leaves).

This paper develops a similar framework for acyclic pipe dreams.
Pipe dreams were introduced by N.~Bergeron and S.~Billey in~\cite{BergeronBilley} to compute Schubert polynomials and later revisited in the context of Gr\"obner geometry by A.~Knutson and E.~Miller~\cite{KnutsonMiller-GroebnerGeometry}, who coined the name \emph{pipe dreams}.
They have important connections and applications to various areas related to Schubert calculus and Schubert varieties~\cite{LascouxSchutzenberger-PolynomesSchubert, LascouxSchutzenberger-SchubertLittlewoudRichardson}. 
A pipe dream is an arrangement of pipes in the triangular shape, each entering along the vertical side and exiting along the horizontal side (see \cref{fig:pipeDreams}).
They are grouped according to their exiting permutation, given by the order in which the pipes appear along the horizontal axis.
The linear extensions of a pipe dream are the permutations of its pipes such that for each contact, the northwest pipe appears before the southeast pipe in the permutation.
The pipe dreams with at least one linear extension are called acyclic and naturally appear in the study of brick polytopes~\cite{PilaudSantos-brickPolytope}.
A flip in a pipe dream exchanges a contact with a crossing between two pipes (see \cref{fig:pipeDreams}), and the flip is increasing when the contact is southwest of the crossing involved in the flip.
A brief recollection on pipe dreams is given in \cref{sec:preliminaries}.

\begin{figure}[b]
	\centerline{
		\includegraphics[scale=.9]{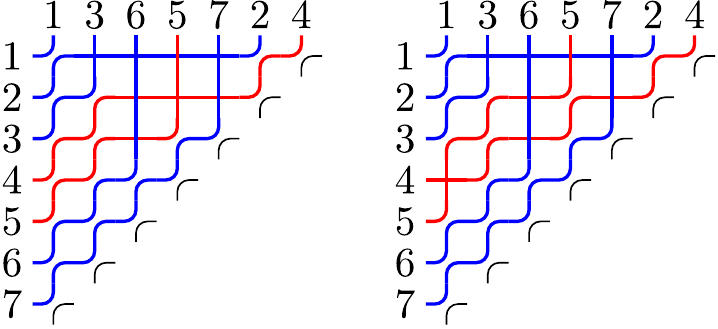}
	}
	\caption{Two pipe dreams of~$\pipeDreams(1365724)$ connected by an increasing flip (exchanging a contact with the crossing on the two red pipes~$4$ and~$5$).}
	\label{fig:pipeDreams}
\end{figure}

In the core \cref{sec:latticeAcyclicPipeDreams} of this paper, we show that for any permutation~$\omega$,
\begin{itemize}
\item the sets of linear extensions of the acyclic pipe dreams with exiting permutation~$\omega$ form a partition of the interval~$[e,\omega]$ of the weak order (\cref{subsec:linearExtensions}),
\item the equivalence relation defined by this partition is a lattice congruence of~$[e, \omega]$, that we call the pipe dream congruence (\cref{subsec:pipeDreamCongruence}),
\item the Hasse diagram of the corresponding lattice quotient is isomorphic to the increasing flip graph on acyclic pipe dreams with exiting permutation~$\omega$ (\cref{subsec:pipeDreamQuotient}).
\end{itemize}
In summary, we obtain the following statement, illustrated in \cref{fig:latticeAcyclicPipeDreams}.

\begin{theoremA}
\label{thm:A}
For any permutation~$\omega$, the Hasse diagram of the lattice quotient of the interval~$[e,\omega]$ of the weak order by the pipe dream congruence of~$\omega$ is isomorphic to the increasing flip graph on acyclic pipe dreams with exiting permutation~$\omega$.
\end{theoremA}

We recover the connection between the weak order and the Tamari lattice by the sylvester congruence for a well-chosen exiting permutation~$\omega$.
Note that \cref{thm:A} is the correct generalization of the Tamari lattice, as neither the increasing flip poset on all pipe dreams, nor its subposet induced by acyclic pipe dreams, are lattices in general (see \cref{rem:increasingFlipPosetNotLattice,rem:acyclicIncreasingFlipPosetNotLattice}).
Other conjectural lattice structures on pipe dreams are considered in~\cite[Conj.~2.8]{Rubey}.

We then explore in \cref{sec:furtherTopics} some natural further topics on pipe dreams.
We first present two algorithms to compute the unique acyclic pipe dream whose linear extensions contain a given permutation generalizing the binary tree insertion map on permutations (\cref{subsec:sweepingAlgorithm,subsec:insertionAlgorithm}).
We then describe the pipe dream congruence as the transitive closure of a local rewriting rule generalizing that of the sylvester congruence (\cref{subsec:rewritingRule}).
We then present a natural commutative diagram of lattice morphisms generalizing the connection between the recoil map and the binary tree insertion map on permutations and the canopy on binary trees (\cref{subsec:canopy}).
Finally, we show that all pipe dreams with exiting permutation $0\omega$ are acyclic whenever $\omega$ is dominant. Therefore, \cref{thm:A} implies that the $\nu$-Tamari lattice is a quotient of the interval~$[e,\omega]$ for a well-chosen dominant permutation $\omega$ (\cref{subsec:nuTamari}).

Finally, we discuss in \cref{sec:subwordComplexes} (partly conjectural) extensions of our results to subword complexes in finite Coxeter groups~\cite{KnutsonMiller-subwordComplex}.
Given a finite Coxeter group~$W$ with simple reflections~$S$, a word~$Q$ on~$S$ and an element~$\omega$ of~$W$, the subword complex~$\subwordComplex(Q,\omega)$ is a simplicial complex whose facets are the complements of the reduced expressions of~$\omega$ inside the word~$Q$.
Pipe dreams can be seen as facets of subword complexes for special words~$Q$ on the simple transpositions of the symmetric groups.
In general, there is again a natural increasing flip graph on the facets of a subword complex, which was studied in particular in~\cite{PilaudStump-ELlabelings}.
An important tool to understand this flip is the root function introduced in~\cite{CeballosLabbeStump}, which associates a root~$\rootFunction{I}{i}$ to each position~$i$ and each facet~$I$, and the root configuration~$\Roots(I)$ of a facet~$I$, which collects all roots~$\rootFunction{I}{i}$ at positions~$i$ in~$I$.
A brief recollection on finite Coxeter systems and subword complexes is given in \cref{subsec:finiteCoxeterGroups,subsec:subwordComplexes}.

We consider the set~$\linearExtensions(I)$ of linear extensions of a facet~$I$ of~$\subwordComplex(Q,\omega)$, that is the set of elements~$\pi$ of~$W$ such that~${\Roots(I) \subseteq \pi(\Phi^+)}$.
We prove the following statement in \cref{subsec:twoTheorems}.

\begin{theoremA}
\label{thm:B}
For any non-empty subword complex~$\subwordComplex(Q, \omega)$, the sets~$\linearExtensions(I)$ for all facets~$I$ of~$\subwordComplex(Q, \omega)$ are order convex and form a partition of a lower set of the weak order that contains the interval~$[e, \omega]$.
\end{theoremA}

In contrast to the case of pipe dreams, there are some subword complexes and some facets for which the interval~$[e, \omega]$ does not contain (even sometimes does not intersect) the set of linear extensions~$\linearExtensions(I)$.
However, there is a large family of subword complexes for which this cannot happen.
We say that~$Q$ is sorting if it contains a reduced expression for~$\wo$.

\begin{theoremA}
\label{thm:C}
If~$Q$ is sorting, then the sets~$\linearExtensions(I)$ for all facets~$I$ of~$\subwordComplex(Q, \omega)$ form a partition of the interval~$[e, \omega]$.
\end{theoremA}

In the case when~$[e, \omega]$ does not contain all sets of linear extensions, it is natural to consider the restriction of this partition to~$[e, \omega]$ by the sets~$[e, \omega] \cap \linearExtensions(\omega)$.
This defines an equivalence relation~$\equiv_{Q, \omega}$ on~$[e, \omega]$ that we call the subword complex equivalence.
As the sets of linear extensions are not always intervals of the weak order, this equivalence relation is not always a lattice congruence.
This seems to be fixed by an additional assumption on~$Q$.
We say that~$Q$ is alternating when all non-commuting pairs $s, t\in S$ alternate within $Q$ (this notion was already considered in \cite{PilaudSantos-brickPolytope, CeballosLabbeStump}).

\begin{conjectureA}
\label{conj:A}
If~$Q$ is alternating, then the subword complex equivalence~$\equiv_{Q, \omega}$ is a lattice congruence of the interval~$[e, \omega]$ of the weak order.
\end{conjectureA}

When~$Q$ is alternating, we can thus consider the quotient~$[e, \omega]/{\equiv_{Q, \omega}}$.
In contrast to the case of pipe dreams, the Hasse diagram of the quotient~$[e, \omega]/{\equiv_{Q, \omega}}$ is not always isomorphic to the increasing flip graph on acyclic facets of~$\subwordComplex(Q, \omega)$ for two reasons:
\begin{itemize}
\item First, not all acyclic facets of~$\subwordComplex(Q, \omega)$ appear as elements of~$[e, \omega]/{\equiv_{Q, \omega}}$ (see \cref{rem:linearExtensionsPartitionSubwordComplexA}). We say that a facet is strongly acyclic if~$[e, \omega] \cap \linearExtensions(\omega) \ne \varnothing$.
\item Second, not all flips between two strongly acyclic facets of~$\subwordComplex(Q, \omega)$ define a cover relation of~$[e, \omega]/{\equiv_{Q, \omega}}$ (see \cref{exm:acyclicFlipvsWeakOrder}). We say that the flip of a position~$i$ in a facet~$I$ is external if the root~$\rootFunction{I}{i}$ directing the flip is a ray of the root configuration~$\Roots(I)$.
\end{itemize}
This leads us to the following conjecture.

\begin{conjectureA}
\label{conj:B}
If~$Q$ is alternating, then the Hasse diagram of the quotient~$[e, \omega]/\equiv_{Q, \omega}$ is isomorphic to the graph of extremal increasing flips between strongly acyclic facets of~$\subwordComplex(Q, \omega)$.
\end{conjectureA}

Combining the sorting and alternating conditions of \cref{thm:B,conj:A,conj:B}, we thus obtain the following conjecture, which can be seen as the natural extension of~\cref{thm:A}.

\begin{conjectureA}
\label{conj:C}
If~$Q$ is sorting and alternating, then the Hasse diagram of the quotient~$[e, \omega]/\equiv_{Q, \omega}$ is isomorphic to the graph of extremal increasing flips between acyclic facets of~$\subwordComplex(Q, \omega)$.
\end{conjectureA}

This last conjecture has a strong connection to the geometry of subword complexes given by brick polytopes~\cite{PilaudSantos-brickPolytope, PilaudStump-brickPolytope} and brick polyhedra~\cite{JahnStump}.
Namely, the extremal increasing flips between acyclic facets of~$\subwordComplex(Q, \omega)$ are precisely the bounded edges (meaning forgetting the unbounded rays) of the brick polyhedron~$\brickPolyhedron(Q, \omega)$ oriented by a natural linear functional.
\cref{conj:C} thus translates geometrically to the following.

\begin{conjectureA}
\label{conj:D}
If~$Q$ is sorting and alternating, then the bounded oriented graph of the brick polyhedron~$\brickPolyhedron(Q, \omega)$ is isomorphic to the Hasse diagram of the lattice quotient~$[e, \omega]/\equiv_{Q, \omega}$.
\end{conjectureA}

In particular, when~$\omega = \wo$, we obtain the following conjecture, extending results from~\cite{Pilaud-brickAlgebra}.

\begin{conjectureA}
\label{conj:E}
If~$Q$ is sorting and alternating, then the oriented graph of the brick polytope $\brickPolyhedron(Q, \wo)$ is isomorphic to the Hasse diagram of the lattice quotient of the weak order by~$\equiv_{Q, \omega}$.
\end{conjectureA}


\section{Preliminaries on pipe dreams}
\label{sec:preliminaries}


\subsection{Pipe dreams}
\label{subsec:pipeDreams}

A \defn{pipe dream}~$P$ is a filling of a triangular shape with crossings~\cross{} and contacts~\elbow{} so that all pipes entering on the left side exit on the top side.
We only consider \defn{reduced} pipe dreams, where two pipes have at most one crossing.
We label the pipes with~$1, 2, \dots, n$ in the order of their entry points from top to bottom.
We denote by~$\omega_P \in \fS_n$ the order of the exit points of the pipes of~$P$ from left to right.
In other words, the pipe entering at row~$i$ exits at column~$\omega^{-1}_P(i)$.
For a fixed permutation~$\omega \in \fS_n$, we denote by~$\pipeDreams(\omega)$ the set of reduced pipe dreams~$P$ such that~$\omega_P = \omega$.

A contact~$c$ is \defn{flippable} if the two pipes passing through contact~$c$ have a crossing~$x$.
The \defn{flip} exchanges the contact~$c$ with the crossing~$x$.
The flip is \defn{increasing} if the contact~$c$ is weakly southwest of the crossing~$x$.
For example, \cref{fig:pipeDreams} illustrates an increasing flip from the left pipe dream to the right pipe dream.
The \defn{increasing flip graph} is the graph of increasing flips on~$\pipeDreams(\omega)$.
It is clearly a directed acyclic graph, and it has a unique source and a unique sink \cite{PilaudPocchiola}, called the \defn{greedy} and \defn{antigreedy} pipe dreams, and denoted~$\greedyPipeDream$ and~$\antiGreedyPipeDream$.
The \defn{increasing flip poset} is the reflexive and transitive closure of the increasing flip graph on~$\pipeDreams(\omega)$.

The \defn{contact graph} of a pipe dream~$P$ is the directed graph~$P\contact$ with one node for each pipe of~$P$ and one arc for each contact of~$P$ connecting the northwest pipe to the southeast pipe of the contact\footnote{We have reversed the usual orientation conventions of \cite{PilaudSantos-brickPolytope, PilaudPocchiola, Pilaud-brickAlgebra} to suit better our purposes, in particular in \cref{subsec:insertionAlgorithm}.}.
We see equivalently the contact graph~$P\contact$ as a (multi)graph on the pipes of~$P$ or on the integers~$[n]$.
We say that a pipe dream~$P$ is \defn{acyclic} if its contact graph~$P\contact$ has no oriented cycle.
For an acyclic pipe dream~$P$, we denote by~$\contactLess{P}$ the transitive closure of the contact graph of~$P$.
For~$\omega \in \fS_n$, we denote by~$\acyclicPipeDreams(\omega)$ the set of acyclic pipe dreams of~$\pipeDreams(\omega)$.

\begin{example}
\label{exm:Tamari1}
We say that a pipe dream is \defn{reversing} if it fixes the first and last pipes and reverses the order of the remaining pipes.
In this case, it is natural to label the pipes from~$0$ to~$n+1$, so that the permutation of the pipes is~$\rho_n \eqdef [0, n, n-1, \dots, 2, 1, n+1]$.
As observed in~\cite{Woo, PilaudPocchiola, Pilaud-these, Stump}, the family of reversing pipe dreams belong to the Catalan families, meaning that it is counted by the famous Catalan numbers.
\cref{fig:bijection} illustrates explicit bijections between reversing pipe dreams of~$\pipeDreams(\rho_n)$, binary trees with $n$ internal nodes, and the triangulations of a convex $(n+2)$-gon.
More precisely, the map which sends a contact in row~$i$ and column~$j$ of the triangular shape (indexed from~$0$ to~$n+1$) to the diagonal~$[i,n+1-j]$ of the~$(n+2)$-gon provides the following correspondence: \\[.3cm]
\centerline{
\begin{tabular}{r@{$\quad\longleftrightarrow\quad$}l}
pipe dream~$P \in \reversingPipeDreams_n$ & triangulation~$P\duality$ of the $(n+2)$-gon, \\
pipe $i$ of~$P$ & triangle $i\duality$ of~$P\duality$ (with central vertex~$i$), \\
contact between pipes~$i$ and~$j$ of~$P$ & common side of triangles~$i\duality$ and~$j\duality$ of~$P\duality$, \\
crossing between pipes~$i$ and~$j$ of~$P$ & common bisector of triangles~$i\duality$ and~$j\duality$ of~$P\duality$, \\
contact graph of~$P$ & dual binary tree of~$P\duality$, \\
contact flips in~$P$ & diagonal flips in~$P\duality$. \\
\end{tabular}} \\[.3cm]
\begin{figure}[h]
	\centerline{\includegraphics[scale=1.23]{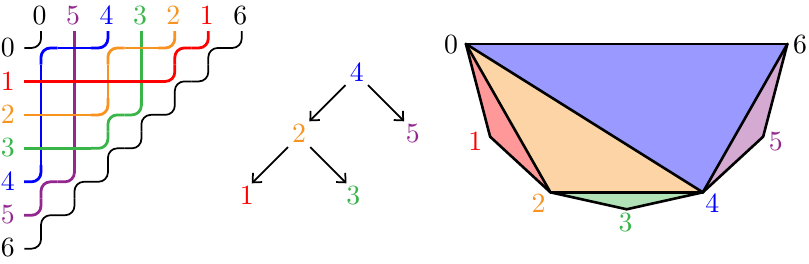}}
	\caption{The bijection between reversing pipe dreams (left), binary trees (middle) and triangulations (right).}
	\label{fig:bijection}
\end{figure}

Hence, this map sends the increasing flip poset on reversing pipe dreams to the \defn{Tamari lattice} on binary trees.
This lattice is defined as the transitive closure of right rotations on binary trees, and is obtained as the quotient of the weak order by the \defn{sylvester congruence}~\cite{Tonks,HivertNovelliThibon-algebraBinarySearchTrees}.
An important point here is that all reversing pipe dreams are acyclic, since their contact graphs are (oriented)~binary~trees.
\end{example}

\begin{remark}
\label{rem:increasingFlipPosetNotLattice}
In contrast to \cref{exm:Tamari1}, the increasing flip poset on all pipe dreams of~$\pipeDreams(\omega)$ is not always a lattice.
The first counter-example happens for the exiting permutation~${\omega = 12543}$.
\end{remark}


\subsection{Crossing and contact properties}
\label{subsec:crossingsContacts}

We now gather some elementary properties of crossings and contacts in pipe dreams that will be needed later to construct pipe dreams from permutations.
For a pipe dream~$P \in \pipeDreams(\omega)$, we call pipe~$j$ the pipe which enters at row~$j$ and exits at column~$\omega^{-1}(j)$.

\begin{lemma}
\label{lem:horiVertCrossings}
For any pipe dream~$P \in \pipeDreams(\omega)$, the pipe~$j$ of~$P$ crosses precisely
\begin{itemize}
\item vertically the pipes~$i$ such that~$i < j$ while~$\omega^{-1}(i) > \omega^{-1}(j)$,
\item horizontally the pipes~$k$ such that~$j < k$ while~$\omega^{-1}(j) > \omega^{-1}(k)$.
\end{itemize}
\end{lemma}

\begin{proof}
For~$i < j$, if~$\omega^{-1}(i) > \omega^{-1}(j)$, then the pipes~$i$ and~$j$ have to cross exactly once (and $j$ must be the vertical pipe at that crossing), while if~$\omega^{-1}(i) < \omega^{-1}(j)$ the pipes~$i$ and~$j$ cannot cross.
The same argument applies for~${k > j}$.
\end{proof}

\begin{lemma}
\label{lem:countingElbows}
For any pipe dream~$P \in \pipeDreams(\omega)$, the pipe~$j$ has precisely
\begin{itemize}
\item $\noninversions{\omega}{j}$ many southeast elbows\!\SEelbow{}
\item $1+\noninversions{\omega}{j}$ many northwest elbows~\WNelbow{}
\item $j-1-\noninversions{\omega}{j}$ vertical crossings~\NScross
\item $\omega^{-1}(j)-1-\noninversions{\omega}{j}$ horizontal crossings~\WEcross
\end{itemize}
where~$\noninversions{\omega}{j} \eqdef \#\set{i \in [n]}{i < j \text{ and } \omega^{-1}(i) < \omega^{-1}(j)}$ is the number of \defn{non-inversions} of~$j$ in~$\omega$.
\end{lemma}

\begin{proof}
Pipe~$j$ enters at row~$j$ and exits at column~$\omega^{-1}(j)$, so that it passes through~$j+\omega^{-1}(j)-1$ grid points.
By \cref{lem:horiVertCrossings}, it has~$\#\set{i \in [n]}{i < j \text{ and } \omega^{-1}(i) > \omega^{-1}(j)} = j-1-\noninversions{\omega}{j}$ vertical crossings and~$\#\set{k \in [n]}{j < k \text{ and } \omega^{-1}(j) > \omega^{-1}(k)} = \omega^{-1}(j)-1-\noninversions{\omega}{j}$ horizontal crossings.
The~$1+2\,\noninversions{\omega}{j}$ remaining grid points along the pipe~$j$ are thus alternating northwest elbows and southeast elbows.
\end{proof}

\begin{lemma}
\label{lem:characterizationPipeDreams}
A collection~$P$ of~$n$ pipes pairwise disjoint except at crossing and contacts and such that for each~$j \in [n]$, the pipe~$j$ enters at row~$j$, exits at column~$\omega^{-1}(j)$, and has~$\noninversions{\omega}{j}$ southeast contacts is a pipe dream of~$\pipeDreams(\omega)$.
\end{lemma}

\begin{proof}
The argument is similar to the previous lemma.
Observe first that the pipe~$j$ must cross the paths~$i$ such that~$i < j$ and~$\omega^{-1}(i) > \omega^{-1}(j)$ and the paths~$k$ such that~$j < k$ and~$\omega^{-1}(j) > \omega^{-1}(k)$.
Moreover, it has~$\noninversions{\omega}{j}$ southeast elbows and thus~$1+\noninversions{\omega}{j}$ northwest elbows.
This already exhausts all~$j+\omega^{-1}(j)-1$ grid points of~$j$.
Therefore, the pipe~$j$ can only cross at most once any other pipe.
\end{proof}


\subsection{Contact graph properties}
\label{subsec:contactGraph}

We now state a simple observation about the poset~$\contactLess{P}$, and two of its consequences, that will play essential roles in the proofs in \cref{sec:latticeAcyclicPipeDreams}.

\begin{lemma}
\label{lem:rectangle}
Let~$P$ be a pipe dream and~$i,j$ be pipes of~$P$.
If there is an elbow of pipe~$i$ weakly northwest of an elbow of pipe~$j$, then~$i \contactLess{P} j$.
\end{lemma}

\begin{proof}
Let~$x$ (resp.~$y$) be the location of an elbow of pipe~$i$ (resp.~$j$) such that~$x$ is weakly northwest of~$y$.
We proceed by induction on the grid distance from~$x$ to~$y$.
If they coincide, then pipes~$i$ and~$j$ share a contact, and $i$ is northwest of~$j$ at this contact by assumption, so that there is an edge from~$i$ to~$j$ in~$P\contact$.
Otherwise, let~$k$ be the pipe of~$P$ with a southeast elbow at~$x$ ($k$ is either the pipe~$i$ itself, or there is an edge from~$i$ to~$k$ in~$P\contact$) and~$\ell$ be the pipe of~$P$ with a northwest elbow at~$y$ ($\ell$ is either the pipe~$j$ itself, or there is an edge from~$\ell$ to~$j$ in~$P\contact$).
Let~$R$ be the axis-parallel rectangle with corners~$x$ and~$y$.
Since pipes~$k$ and~$\ell$ cross at most once, at least one of them has an additional elbow along the sides of~$R$.
Assume for instance that~$k$ has an elbow at~$x'$.
Then~$x'$ is still weakly northwest of~$y$ and $x',y$ are strictly closer than~$x,y$.
By induction, there is a directed path from~$k$ to~$\ell$ in~$P\contact$, and thus a path directed from~$i$ to~$j$.
\end{proof}

Note that the reciprocal assertion of \cref{lem:rectangle} is false.
We conclude this section by two consequences of \cref{lem:rectangle}.

\begin{lemma}
\label{lem:consequenceRectangle1}
If~$i < j$ and~$\omega^{-1}(i) < \omega^{-1}(j)$, then $i \contactLess{P} j$ for any~$P \in \pipeDreams(\omega)$.
\end{lemma}

\begin{proof}
If~$i < j$ and~$\omega^{-1}(i) < \omega^{-1}(j)$, then the pipes~$i$ and~$j$ do not cross.
Consider an elbow~$e$ of the pipe~$i$.
Since the pipe~$j$ passes southeast of~$e$, it has an elbow southeast of~$e$.
We conclude that~$i \contactLess{P} j$ by \cref{lem:rectangle}.
\end{proof}

\begin{lemma}
\label{lem:consequenceRectangle2}
Let~$P$ be a pipe dream and let~$i,j,k$ be three pipes of~$P$ such that~$i < j < k$ and~$\omega^{-1}(i) > \omega^{-1}(j) > \omega^{-1}(k)$.
If~$i \to k$ in~$P\contact$, then either~$i \contactLess{P} j \contactMore{P} k$ or~$i \contactMore{P} j \contactLess{P} k$.
\end{lemma}

\begin{proof}
Let~$c$ denote the contact of pipes~$i$ and~$k$ in~$P$.
Decompose the triangular shape into three regions: the region~$A$ of all points located southwest of~$c$, the region~$B$ of all points located northwest or southeast of~$c$, and the region~$C$ of all points located northeast of~$c$.
Since~$i < j$ and~$\omega^{-1}(j) > \omega^{-1}(k)$, the pipe~$j$ starts in region~$A$ and ends in region~$C$.
Hence, the pipe~$j$ has an elbow~$e$ in region~$B$.
We thus obtain that~$i \contactLess{P} j \contactMore{P} k$ if this $e$ is southeast of~$c$, and that~$i \contactMore{P} j \contactLess{P} k$ if $e$ is northwest of~$c$.
\end{proof}


\section{Lattice of acyclic pipe dreams}
\label{sec:latticeAcyclicPipeDreams}

As already mentioned, the increasing flip poset on reversing pipe dreams is isomorphic to the Tamari lattice, which is a lattice quotient of the weak order.
In this section, we extend this result to any permutation~$\omega$ by showing that the sets of linear extensions of pipe dreams of~$\acyclicPipeDreams(\omega)$ partition the interval~$[e,w]$ (\cref{subsec:linearExtensions}), that this partition actually defines a lattice congruence of the weak order on~$[e,w]$ (\cref{subsec:pipeDreamCongruence}), and that the Hasse diagram of the quotient by this congruence is the increasing flip graph on acyclic pipe dreams of~$\acyclicPipeDreams(\omega)$ (\cref{subsec:pipeDreamQuotient}).
This section goes straight to the proof of this property, and leaves alternative perspectives on this quotient to \cref{sec:furtherTopics}.


\subsection{Linear extensions of pipe dreams}
\label{subsec:linearExtensions}

The main characters in this section are the following sets of permutations.

\begin{definition}
\label{def:linearExtensions}
We say that a permutation~$\pi$ is a \defn{linear extension} of a pipe dream~$P \in \pipeDreams(\omega)$ if~$\pi^{-1}(i) < \pi^{-1}(j)$ for every arc~$i \to j$ in~$P\contact$ (we should say linear extension of~$\contactLess{P}$, but prefer to simplify notation).
We denote by~$\linearExtensions(P)$ the set of linear extensions of~$P$.
\end{definition}

In this section, we prove the following structural property of~$\linearExtensions(P)$, illustrated in \cref{fig:latticeAcyclicPipeDreams}.

\begin{theorem}
\label{thm:partitionPipeDreams}
The set~$\set{\linearExtensions(P)}{P \in \acyclicPipeDreams(\omega)}$ partitions the weak order interval~$[e,\omega]$.
\end{theorem}

\begin{example}
\label{exm:Tamari2}
Following \cref{exm:Tamari1}, observe that the permutations of~$\{0, 1, \dots, n, n+1\}$ below~$\rho_n$ in weak order are precisely the permutations of the form~$[0, \pi, n+1]$ for some~$\pi \in \fS_n$.
It is well-known that any permutation~$\pi \in \fS_n$ is a linear extension of a unique binary tree~\cite{Tonks}.
This binary tree can be obtained by inserting~$\pi$ from right to left in a binary search tree~\cite{HivertNovelliThibon-algebraBinarySearchTrees}.
Hence any permutation of~$\{0, 1, \dots, n, n+1\}$ below~$\rho_n$ is a linear extension of a unique reversing pipe dream on~$n+2$ pipes.
For instance, the pipe dream of \cref{fig:bijection} has linear extensions~$0421356$, $0423156$, $0421536$, $0423516$, $0425136$, $0425316$, $0452136$, $0452316$.
\end{example}

We will see in \cref{sec:furtherTopics} insertion algorithms to compute the pipe dream~$P \in \acyclicPipeDreams(\omega)$ such that~$\pi \in \linearExtensions(P)$ for given permutations~$\pi \le \omega$.
These algorithms are however not needed for the proof of \cref{thm:partitionPipeDreams}, which we break into the following three lemmas.

\begin{lemma}
\label{lem:lowerSetPipeDreams}
If~$\pi \eqdef UjiV$ covers $\pi' \eqdef UijV$ in weak order, and ${\pi \in \linearExtensions(P)}$ for some~${P \in \acyclicPipeDreams(\omega)}$,~then
\begin{itemize}
\item if~$P\contact$ has no arc~$j \to i$, then~$\pi' \in \linearExtensions(P)$,
\item otherwise, $\pi' \in \linearExtensions(P')$ where~$P'$ denotes the pipe dream obtained from~$P$ by flipping the furthest northeast contact between pipes~$i$ and~$j$ in~$P$.
\end{itemize}
\end{lemma}

\begin{proof}
The first point is obvious.
For the second point, observe that the flip of the furthest contact just reverses all arcs~$j \to i$ and exchanges~$i$ and~$j$ at some extremities of the arcs of the contact graph.
\end{proof}

\begin{lemma}
\label{lem:partition}
If~$\pi \le \omega$ in weak order, then~$\pi$ is a linear extension of a unique pipe dream~$P \in \acyclicPipeDreams(\omega)$.
\end{lemma}

\begin{proof}
Consider the greedy and antigreedy pipe dreams~$\greedyPipeDream$ and~$\antiGreedyPipeDream$ of \cite{PilaudPocchiola}.
For any contact between the pipes~$i$ and~$j$ in~$\greedyPipeDream$, with~$i < j$, we have
\begin{itemize}
\item if the pipes~$i$ and~$j$ never cross, then~$\omega^{-1}(i) < \omega^{-1}(j)$,
\item if the pipes~$i$ and~$j$ cross, then~$\omega^{-1}(i) > \omega^{-1}(j)$ and the contact in~$\greedyPipeDream$ must be from~$i$ to~$j$ (since all flips in~$\greedyPipeDream$ are increasing by definition).
\end{itemize}
We conclude that~$e \in \linearExtensions(\greedyPipeDream)$.
Conversely, if~$e \in \linearExtensions(P)$, all arcs of~$P\contact$ are increasing, so that all flips in~$P$ are increasing.
We conclude that~$\greedyPipeDream$ is the unique pipe dream with~$e \in \linearExtensions(\greedyPipeDream)$.
Similar arguments show that~$\antiGreedyPipeDream$ is the unique pipe dream with~$\omega \in \linearExtensions(\antiGreedyPipeDream)$.
The result thus follows from~\cref{lem:lowerSetPipeDreams}, since it shows that the existence (resp.~uniqueness) of a pipe dream~$P$ such that~$\pi \in \linearExtensions(P)$ is preserved when going down (resp.~up) in weak order.
\end{proof}

\begin{lemma}
\label{lem:interval}
If~$\pi$ is a linear extension of a pipe dream~$P \in \acyclicPipeDreams(\omega)$, then~$\pi \le \omega$ in weak order.
\end{lemma}

\begin{proof}
For any~$i < j$ with~$\omega^{-1}(i) < \omega^{-1}(j)$, we have~$i \contactLess{P} j$ by \cref{lem:consequenceRectangle1}, thus ${\pi^{-1}(i) < \pi^{-1}(j)}$ since~$\pi \in \linearExtensions(P)$.
In other words, any non-inversion of~$\omega$ is a non-inversion of~$\pi$, so that~${\pi \le \omega}$.
\end{proof}

\begin{proof}[Proof of \cref{thm:partitionPipeDreams}]
This is a direct consequence of \cref{lem:partition,lem:interval}.
\end{proof}

\begin{notation}
For~$\pi \in [e, \omega]$, we denote by~$\insertion{\pi}{\omega}$ the pipe dream of~$\acyclicPipeDreams(\omega)$ such that~${\pi \in \linearExtensions(\insertion{\pi}{\omega})}$.
\end{notation}

\begin{figure}
	\centerline{\includegraphics[scale=1]{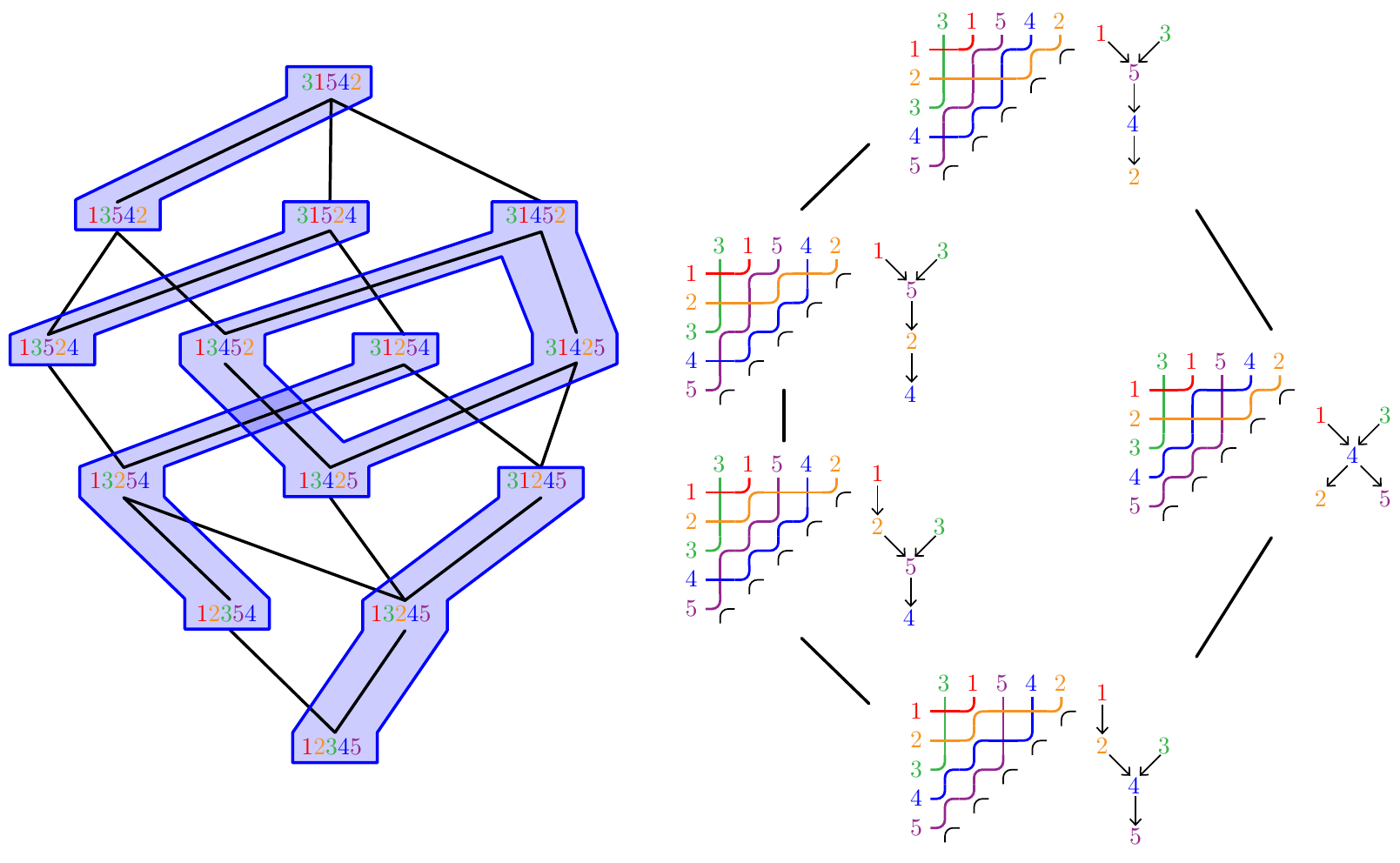}}
	\caption{The pipe dream congruence on the weak order interval~$[12345, 31542]$ (left) and the increasing flip graph on acyclic pipe dreams (right). The blue bubbles represent the classes of the pipe dream congruence. The acyclic pipe dreams are represented with their contact graphs.}
	\label{fig:latticeAcyclicPipeDreams}
\end{figure}


\subsection{Pipe dream congruence}
\label{subsec:pipeDreamCongruence}

A \defn{congruence} of a lattice~$(L, \le, \meet, \join)$ is an equivalence relation~$\equiv$ on~$L$ which respects meets and joins: $x \equiv x'$ and~$y \equiv y'$ implies $x \meet y \equiv x' \meet y'$ and~$x \join y \equiv x' \join y'$.
We will use the following classical characterization of lattice congruences, see~\cite{Reading-PosetRegionsChapter}.

\begin{proposition}
\label{prop:characterizationCongruences}
An equivalence relation~$\equiv$ on a lattice~$L$ is a congruence if and only if
\begin{enumerate}[(i)]
\item every equivalence class of~$\equiv$ is an interval of~$L$,
\item the projections~$\projDown : L \to L$ and~$\projUp : L \to L$, which maps an element of~$L$ to the minimal and maximal elements of its equivalence class respectively, are order preserving.
\end{enumerate}
\end{proposition}

We now focus on the following congruence, illustrated in \cref{fig:latticeAcyclicPipeDreams}.

\begin{definition}
\label{def:pipeDreamCongruence}
The \defn{pipe dream congruence} is the equivalence relation~$\equiv_\omega$ on the weak order interval~$[e,\omega]$ whose equivalence classes are the sets~$\linearExtensions(P)$ of linear extensions of the pipe dreams~$P$ of~$\acyclicPipeDreams(\omega)$.
In other words, $\pi \equiv_\omega \pi'$ if and only if~$\insertion{\pi}{\omega} = \insertion{\pi'}{\omega}$.
\end{definition}

\pagebreak
Note that the pipe dream congruence indeed defines an equivalence relation by \cref{thm:partitionPipeDreams}.
In this section, we prove that it is a lattice congruence.

\begin{theorem}
\label{thm:pipeDreamCongruence}
The pipe dream congruence~$\equiv_\omega$ is a congruence of the weak order interval~$[e,\omega]$.
\end{theorem}

\begin{example}
\label{exm:Tamari3}
Following \cref{exm:Tamari1,exm:Tamari2}, observe that the congruence~$\equiv_{\rho_n}$ on the permutations of~$\{0, 1, \dots, n, n+1\}$ below~$\rho_n$ corresponds to the sylvester congruence on~$\fS_n$ defined in~\cite{HivertNovelliThibon-algebraBinarySearchTrees}.
The classes of this congruence are the sets of linear extensions of the binary trees (considered as posets, labeled in inorder, and oriented towards their roots).
It can also be defined by the classical rewriting rule~$UjVikW \equiv UjVkiW$ where~$i < j < k$ are elements of~$[n]$ while~$U, V, W$ are (possibly empty) words on~$[n]$.
\end{example}

We will discuss in \cref{subsec:rewritingRule} rewriting rules for the pipe dream congruence~$\equiv_\omega$, for any permutation~$\omega$.
These rewriting rules are however not needed for the proof of \cref{thm:pipeDreamCongruence}.
We prove it by checking both conditions of \cref{prop:characterizationCongruences}.
For the first condition, we need the following classical characterization of weak order intervals, see \cite{BjornerWachs} or~\cite{ChatelPilaudPons}.

\begin{proposition}[{\cite[Thm.~6.8]{BjornerWachs}}]
\label{prop:WOIP}
The set~$\linearExtensions(\less)$ of linear extensions of a poset~$\less$ on~$[n]$ forms an interval~$I \eqdef [\min(I), \max(I)]$ of the weak order if and only if for every~$i < j < k$,
\[
i \less k \implies i \less j \text{ or } j \less k
\qquad\text{and}\qquad
i \more k \implies i \more j \text{ or } j \more k.
\]
Moreover, the inversions of~$\min(I)$ are the pairs~$i,j \in [n]$ with $i < j$ and $i \more j$, and the non-inversions of~$\max(I)$ are the pairs~$i,j \in [n]$ with $i < j$ and $i \less j$.
\end{proposition}

\begin{proposition}
\label{prop:intervals}
For any pipe dream~$P \in \acyclicPipeDreams(\omega)$, the set~$\linearExtensions(P)$ is an interval of the weak order.
\end{proposition}

\begin{proof}
We just need to show that the poset~$\contactLess{P}$ satisfies the conditions of \cref{prop:WOIP}.
Consider~$i < j < k$ such that~$i \contactLess{P} k$.
If~$\omega^{-1}(i) < \omega^{-1}(j)$, then~$i \contactLess{P} j$ by \cref{lem:consequenceRectangle1}.
Similarly, if~$\omega^{-1}(j) < \omega^{-1}(k)$, then~$j \contactLess{P} k$ by \cref{lem:consequenceRectangle1}.
We can thus assume that~${\omega^{-1}(i) > \omega^{-1}(j) > \omega^{-1}(k)}$.
Decompose the triangular shape into three regions: the region~$A$ of all points located northeast of the last elbow of the pipe~$j$ of~$P$, the region~$B$ of all points located northwest or southeast of an elbow of the pipe~$j$ of~$P$, and the region~$C$ of all points located southwest of the first elbow of the pipe~$j$ of~$P$.
Since~$i \contactLess{P} k$, there is a path~$\pi$ from the exiting point of the pipe~$i$ of~$P$ to the entering point of the pipe~$j$ of~$P$ which travels along the pipes of~$P$, possibly jumping from the northwest pipe to the southeast pipe of a contact it encounters.
Since~$i < j$ and~$\omega^{-1}(j) > \omega^{-1}(k)$, the path~$\pi$ starts in region~$A$ and ends in region~$C$, so that it necessarily passes from region~$A$ to region~$C$.
Since the southwest corner of~$A$ is located northeast of the northeast corner of~$C$, this forces an elbow~$e$ of~$\pi$ to lie in region~$B$.
\cref{lem:rectangle} then ensures that either~$i \contactLess{P} j$ (if~$e$ is north of pipe~$j$), or~$j \contactLess{P} k$ (if~$e$ is south of pipe~$j$).
The proof is similar if~$i \contactMore{P} k$.
\end{proof}

\begin{proposition}
\label{prop:orderPreserving}
Let $\sigma, \sigma'$ be two permutations of~$[e, \omega]$ and~$C, C'$ denote their $\equiv_\omega$-congruence classes.
Then $\sigma \le \sigma'$ implies~$\min(C) \le \min(C')$ and $\max(C) \le \max(C')$ in weak order.
\end{proposition}

\begin{proof}
We prove the statement for the maximums, the proof for the minimums is symmetrical.
Observe first that we can assume that~$\sigma$ is covered by~$\sigma'$ in weak order, so that we write~${\sigma' = \sigma s_p}$ for some simple transposition~$s_p \eqdef (p \; p+1)$.
The proof now works by induction on the weak order distance between~$\sigma$ and~$\max(C)$.
If~$\sigma = \max(C)$, the result is immediate as~${\max(C) = \sigma < \sigma' \le \max(C')}$.
Otherwise, $\sigma$ is covered by a permutation~$\tau$ in the class~$C$, and we write~$\tau = \sigma s_q$ for some simple transposition~$s_q \eqdef (q \; q+1)$.
Let~$P,P' \in \acyclicPipeDreams(\omega)$ be such that~$C = \linearExtensions(P)$ and~$C' = \linearExtensions(P')$.
We now distinguish five cases, according to the relative positions of~$p$ and~$q$:
\begin{enumerate}[(1)]
\item If~$p > q+1$, then~$\sigma = UijVk\ell W$, $\sigma' = UijV\ell kW$ and~$\tau = UjiVk\ell W$ for some~$i < j$~and~${k < \ell}$. Define~$\tau' \eqdef \sigma s_p s_q = \sigma s_q s_p = UjiV\ell kW$. By \cref{lem:lowerSetPipeDreams}, there is no arc~$i \to j$ in~$P\contact$ (since~$\sigma$ and~$\tau$ both belong to~$C$), and~$P\contact$ and~$P'{}\contact$ can only differ by arcs incident~to~$k$~or~$\ell$. Hence, there is no arc~$i \to j$ in~$P'{}\contact$. We thus obtain again by \cref{lem:lowerSetPipeDreams} that~$\tau' \in \linearExtensions(P') = C'$.
\item If~$p = q+1$, then~$\sigma = UijkV$, $\sigma' = UikjV$ and~$\tau = UjikV$ for some~$i < j < k$. Define~$\tau' \eqdef \sigma s_p s_q s_p = \sigma s_q s_p s_q = UkjiV$. Since~$\sigma \in \linearExtensions(P)$, we have~$i \not\contactMore{P} j$ and~$j \not\contactMore{P} k$, so that there is no arc~$i \to k$ in~$P\contact$ by \cref{lem:consequenceRectangle2}. By \cref{lem:lowerSetPipeDreams}, there is no arc~$i \to j$ in~$P\contact$, and~$P\contact$ and~$P'{}\contact$ can only differ by arcs incident to~$j$ or~$k$. We thus obtain that there is no arc~$i \to j$ nor~$i \to k$ in~$P'{}\contact$. Consequently, again by \cref{lem:lowerSetPipeDreams}, both~$\sigma' s_q$ and~$\tau' = \sigma' s_q s_p$ belong to~$\linearExtensions(P') = C'$.
\item If~$p = q$, then~$\sigma' = \tau$ is in~$C$, so that~$C = C'$ and there is nothing to prove.
\item If~$p = q-1$, we proceed similarly as in Situation~(2).
\item If~$p < q-1$, we proceed similarly as in Situation~(1).
\end{enumerate}
In all cases, we found~$\tau' > \tau$ with~$\tau' \in C'$.
Since~$\tau < \tau'$ with~$\tau \in C$ and~$\tau' \in C'$, and since~$\tau$ is closer to~$\max(C)$ than~$\sigma$, we obtain that ${\max(C) < \max(C')}$ by induction hypothesis.
\end{proof}

\begin{proof}[Proof of \cref{thm:pipeDreamCongruence}]
Follows from \cref{prop:characterizationCongruences}, whose conditions are guaranteed by \cref{prop:intervals,prop:orderPreserving}.
\end{proof}


\subsection{Pipe dream quotient}
\label{subsec:pipeDreamQuotient}

For a congruence~$\equiv$ of a lattice~$L$, the \defn{lattice quotient}~$L/{\equiv}$ is the lattice on the classes of~$\equiv$ where for any two congruence classes~$X$ and~$Y$, 
\begin{itemize}
\item $X \le Y$ in~$L/{\equiv}$ if and only if there exist representatives~$x \in X$ and~$y \in Y$ such that~$x \le y$ in~$L$, or equivalently $\min(X) \le \min(Y)$, or equivalently $\max(X) \le \max(Y)$,
\item $X \meet Y$ (resp.~$X \join Y$) is the congruence class of~$x \meet y$ (resp.~of~$x \join y$) for arbitrary representatives~$x \in X$ and~$y \in Y$.
\end{itemize}
In this section, we aim at the following statement, illustrated in \cref{fig:latticeAcyclicPipeDreams}.

\begin{theorem}
\label{thm:pipeDreamQuotient}
The Hasse diagram of the lattice quotient~$[e,\omega]/{\equiv_\omega}$ is isomorphic to the increasing flip graph on~$\acyclicPipeDreams(\omega)$.
Hence, the transitive closure of the increasing flip graph on~$\acyclicPipeDreams(\omega)$ is a lattice.
\end{theorem}

\begin{example}
\label{exm:Tamari4}
Following \cref{exm:Tamari1,exm:Tamari2,exm:Tamari3}, observe that the increasing flip poset on reversing pipe dreams is isomorphic to the Tamari lattice, which is the quotient of the weak order by the sylvester congruence.
\end{example}

\begin{remark}
\label{rem:acyclicIncreasingFlipPosetNotLattice}
The increasing flip poset on~$\pipeDreams(\omega)$ is the transitive closure of the increasing flip graph on~$\pipeDreams(\omega)$.
Observe that the two natural ways to restrict to the acyclic pipe dreams of~$\acyclicPipeDreams(\omega)$ (restrict either the flip graph or the flip poset) do not coincide in general.
Namely, the transitive closure of the subgraph of the increasing flip graph induced by~$\acyclicPipeDreams(\omega)$ may have strictly less relations than the subposet of the increasing flip poset induced by~$\acyclicPipeDreams(\omega)$.
\cref{fig:counterExampleRestrictionIncreasingFlipPoset1} illustrates two acyclic pipe dreams of~$\acyclicPipeDreams(126543)$ connected by a sequence of increasing flips in~$\pipeDreams(126543)$ but by no sequence of increasing flips in~$\acyclicPipeDreams(126543)$.
In fact, the subposet of the increasing flip poset induced by~$\acyclicPipeDreams(126543)$ is not even a lattice, as illustrated in \cref{fig:counterExampleRestrictionIncreasingFlipPoset2}.
This example is minimal.
%
%
\begin{figure}[ht]
	\centerline{\includegraphics[scale=1.1]{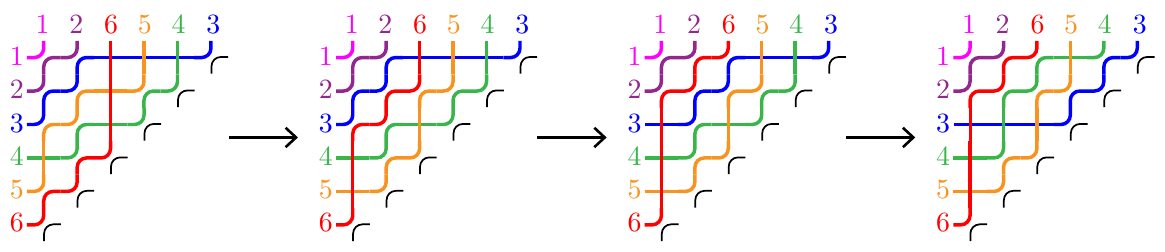}}
	\caption{Two acyclic pipe dreams in~$\acyclicPipeDreams(126543)$ connected by a sequence of increasing flips in~$\pipeDreams(126543)$ but by no sequence of increasing flips in~$\acyclicPipeDreams(126543)$.}
	\label{fig:counterExampleRestrictionIncreasingFlipPoset1}
\end{figure}
\begin{figure}[ht]
	\centerline{\includegraphics[scale=1.1]{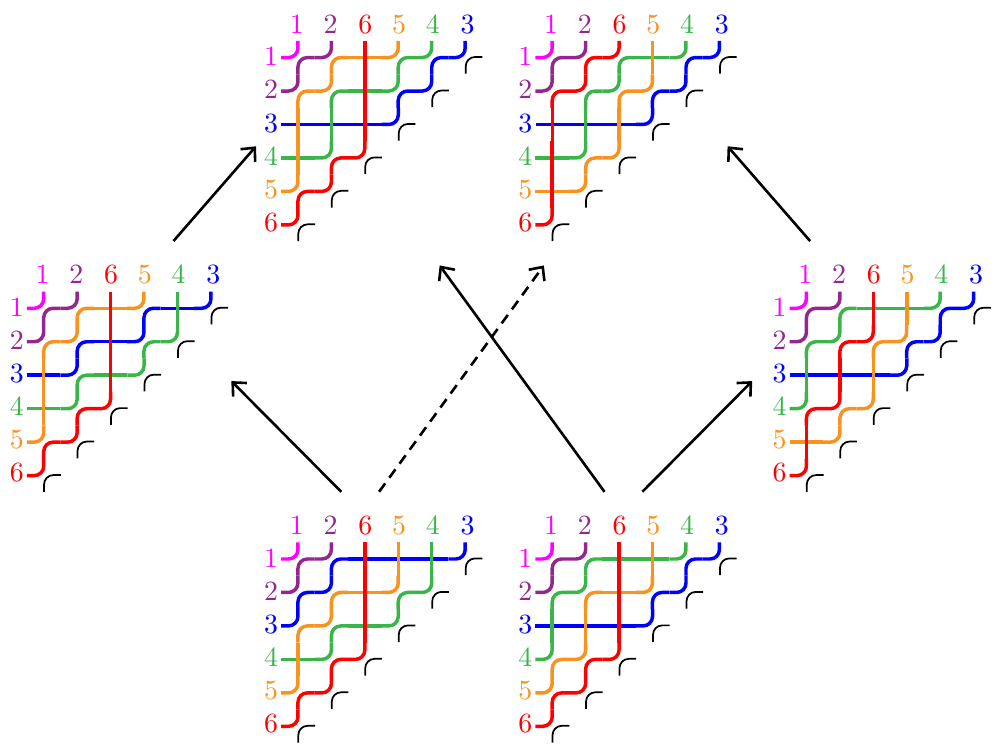}}
	\caption{Some acyclic pipe dreams of~$\acyclicPipeDreams(126543)$, connected by some strong arrows representing increasing flips, and by a dotted arrow representing the relation of \cref{fig:counterExampleRestrictionIncreasingFlipPoset1}. The subposet of the increasing flip poset on~$\pipeDreams(126543)$ induced by~$\acyclicPipeDreams(126543)$ contains the dotted arrow and is thus not a lattice (the two bottom elements of the picture have no join while the two top elements of the picture have no meet). In contrast, the transitive closure of the increasing flip graph on~$\acyclicPipeDreams(126543)$ does not contain the dotted arrow and is a lattice (\cref{thm:pipeDreamQuotient}).}
	\label{fig:counterExampleRestrictionIncreasingFlipPoset2}
\end{figure}
\end{remark}

To prove \cref{thm:pipeDreamQuotient}, we need the following auxiliary statement.

\begin{lemma}
\label{lem:flippable}
Consider two acyclic pipe dreams~$P,P' \in \acyclicPipeDreams(\omega)$ connected by the flip of a contact between their pipes~$i$ and~$j$.
Then any directed path in~$P\contact$ or~$P'{}\contact$ between~$i$ and~$j$ is an arc.
\end{lemma}

\begin{proof}
Say that~$i < j$ while~$\omega^{-1}(i) > \omega^{-1}(j)$ and that~$i \to j$ is an arc of~$P\contact$ while~$j \to i$ is an arc of~$P'{}\contact$.
Since $P\contact$ is acyclic, there is no path from~$j$ to~$i$ in~$P\contact$.
Assume by means of contradiction that there is a path~$i \to k_1 \to \dots \to k_p \to j$ in~$P\contact$ with~$p \ge 1$.
Since the arcs of~$P'{}\contact$ are the arcs of~$P\contact$ where only extremities~$i$ and~$j$ can be changed, $P'{}\contact$ contains the path~$k_1 \to \dots \to k_p$ and at least one of the arcs~$i \to k_1$ or~$j \to k_1$, and at least one of the arcs~$k_p \to j$ or~$k_p \to i$.
Consequently, since~$P'{}\contact$ contains the arc~$j \to i$ and is acyclic, it must contain the path~$j \to k_1 \to \dots \to k_p \to i$.
We thus obtained that~$i \contactLess{P} k_1 \contactLess{P} j$ while~$i \contactMore{P'} k_1 \contactMore{P'} j$, and~$k_1$ has a contact with~$i$ in~$P$ that becomes a contact with~$j$ in~$P'$.

Consider now the contact~$c$ of~$P$ which is a crossing in~$P'$ and the contact~$c'$ of~$P'$ which is a crossing of~$P$.
Let~$R$ the rectangle with corners~$c$ and~$c'$.
Since~$k_1$ has a contact with~$i$ in~$P$ and with~$j$ in~$P'$, it must pass inside~$R$.
Since~$i \contactLess{P} k_1 \contactLess{P} j$ and~$i \contactMore{P'} k_1 \contactMore{P'} j$, the pipe~$k$ has no elbow located northwest or southeast of~$c$ or~$c'$, hence no elbow located north, south, west or east of~$R$.
We thus obtain that~$k$ must be straight before it reaches~$R$, and after it leaves~$R$.
Hence $k_1 < i < j$ and~$\omega^{-1}(k_1) < \omega^{-1}(j) < \omega^{-1}(i)$.
By \cref{lem:consequenceRectangle1}, this contradicts~$i \contactLess{P} k_1$~and~$k_1 \contactMore{P'} j$.
\end{proof}

\begin{proof}[Proof of \cref{thm:pipeDreamQuotient}]
We need to prove that the following conditions are equivalent for two distinct pipe dreams~$P,P' \in \acyclicPipeDreams(\omega)$:
\begin{enumerate}[(i)]
\item there is an increasing flip from $P$ to~$P'$,
\item there exist linear extensions~$\pi$ of~$P$ and~$\pi'$ of~$P'$ such that~$\pi'$ covers~$\pi$ in weak order.
\end{enumerate}
\cref{lem:lowerSetPipeDreams,thm:partitionPipeDreams} directly imply that (ii) $\Rightarrow$ (i).
For (i) $\Rightarrow$ (ii), let~$i < j$ be the two pipes involved in the flip between~$P$ and~$P'$.
Hence, $i \to j$ is an arc of~$P\contact$ while~$j \to i$ is an arc of~$P'{}\contact$.
By \cref{lem:flippable}, there is no directed path from~$i$ to~$j$ in~$P\contact$ besides the arcs~$i \to j$ (there might be more than one such arc).
Hence, there exists a linear extension~$\pi$ of~$P$ where~$i$ and~$j$ are consecutive.
Write~$\pi \eqdef UijV$ and define~$\pi' \eqdef UjiV$.
Since~$i \to j$ is an arc of~$P\contact$, $\pi'$ is not a linear extension of~$P$.
Hence, by \cref{lem:lowerSetPipeDreams}, $\pi'$ is a linear extension of~$P'$.
\end{proof}

Let us conclude by providing more equivalent characterizations of the increasing flip lattice on~$\acyclicPipeDreams(\omega)$.

\begin{proposition}
For any pipe dreams~$P, P' \in \acyclicPipeDreams(\omega)$, the following assertions are equivalent:
\begin{enumerate}[(i)]
\item there is a path from $P$ to~$P'$ in the increasing flip graph on~$\acyclicPipeDreams(\omega)$,
\item there exist linear extensions~$\pi$ of~$P$ and~$\pi'$ of~$P'$ such that~$\pi < \pi'$ in weak order,
\item the minimal (resp.~max) linear extensions~$\pi$ of~$P$ and~$\pi'$ of~$P'$ satisfy~$\pi < \pi'$ in weak order.
\item there is no~$i < j$ such that~$i \contactMore{P} j$ and~$i \contactLess{P'} j$,
\item for all~$i < j$, if~$i \contactMore{P} j$, then~$i \contactMore{P'} j$,
\item for all~$i < j$, if~$i \contactLess{P'} j$, then~$i \contactLess{P} j$.
\end{enumerate}
\end{proposition}

\begin{proof}
We already proved the equivalence (i) $\Leftrightarrow$ (ii) in \cref{thm:pipeDreamQuotient}.
The equivalence \mbox{(ii) $\Leftrightarrow$ (iii)} is valid for any lattice quotient.
The equivalences (iii) $\Leftrightarrow$ (iv) $\Leftrightarrow$ (v) $\Leftrightarrow$ (vi) follow from the descriptions of \cref{prop:WOIP} of the inversions of the minimum and the non-inversions of the maximum of a weak order interval.
\end{proof}


\section{Further topics on pipe dreams}
\label{sec:furtherTopics}

In this section, we discuss five further topics on the pipe dream congruence.
We first present two algorithms to construct the pipe dream~$\insertion{\pi}{\omega}$ of~$\acyclicPipeDreams(\omega)$ of which a given permutation~$\pi$ is a linear extension (\cref{subsec:sweepingAlgorithm,subsec:insertionAlgorithm}).
We then describe the pipe dream congruence~$\equiv_\omega$ as the transitive closure of a rewriting rule on permutations of~$[e, \omega]$ (\cref{subsec:rewritingRule}).
We then present the natural coarsening of the pipe dream congruence~$\equiv_\omega$ by the recoil congruence~$\cong_\omega$ (\cref{subsec:canopy}).
Finally, we discuss a specific family of pipe dreams in connection to the $\nu$-Tamari lattices (\cref{subsec:nuTamari}).


\subsection{Sweeping algorithm}
\label{subsec:sweepingAlgorithm}

Our first algorithm to construct~$\insertion{\pi}{\omega}$ is a \defn{sweeping algorithm}, inspired by the algorithm to compute greedy pipe dreams~\cite{PilaudPocchiola, PilaudStump-ELlabelings}.
An extension of this algorithm to subword complexes will be discussed in \cref{subsec:sweepingAlgorithmSubwordComplexes}, and a related algorithm appeared independently in~\cite{JahnStump}.
We say that an order on the boxes of the triangular shape is \defn{northeast compatible} if each box~$b$ is before all boxes which are located weakly northeast of~$b$.

\begin{proposition}
\label{prop:sweepingAlgorithm}
For any permutations~$\pi,\omega \in \fS_n$ such that~$\pi \le \omega$ in weak order, the unique pipe dream~$P \in \acyclicPipeDreams(\omega)$ such that~$\pi \in \linearExtensions(P)$ can be constructed by sweeping the triangular shape in any northeast compatible order and placing a crossing when sweeping a vertex~$v$ of the grid where pipe~$i$ arrives horizontally and pipe~$j$ arrives vertically if and only if
\begin{itemize}
\item $i < j$ and~$\omega^{-1}(i) > \omega^{-1}(j)$, and 
\item $\pi^{-1}(i) > \pi^{-1}(j)$ or vertex~$v$ lies in column~$\omega^{-1}(j)$.
\end{itemize}
See \cref{fig:sweepingAlgorithm1,fig:sweepingAlgorithm2}.
\end{proposition}

\begin{example}
\label{exm:sweepingAlgorithm}
\cref{fig:sweepingAlgorithm1,fig:sweepingAlgorithm2} illustrate the sweeping algorithm for the exiting permutation ${\omega = 561324}$ and the input permutations~${\pi = 513264}$ and~${\pi = 512364}$ respectively.
The algorithm has 15 steps, but we have grouped together steps $2$ to $9$ (second arrow) and steps $12$ to $13$ (fifth arrow) as they have the same justification.
For~${\pi = 513264}$ in \cref{fig:sweepingAlgorithm1}, we place 
\begin{itemize}
\item contacts at steps $1$, $10$, $12$, $13$ and~$15$ since~$\omega^{-1}(i) < \omega^{-1}(j)$, and at step~$14$ since~$i > j$,
\item crossings at steps $2$ to $9$ since~$i < j$, $\omega^{-1}(i) > \omega^{-1}(j)$ and we are in the column~$\omega^{-1}(j)$, and at step $11$ since~$i < j$, $\omega^{-1}(i) > \omega^{-1}(j)$ and~$\pi^{-1}(i) > \pi^{-1}(j)$.
\end{itemize}
For~${\pi = 512364}$ in \cref{fig:sweepingAlgorithm2}, we make the same choices, except that we place 
\begin{itemize}
\item a contact at step~$11$ since~$\pi^{-1}(i) < \pi^{-1}(j)$ and we are not in column~$\omega^{-1}(j)$, 
\item a crossing at step~$14$ since~$i < j$, $\omega^{-1}(i) > \omega^{-1}(j)$ and we are in the column~$\omega^{-1}(j)$.
\end{itemize}
\begin{figure}[t]
	\centerline{\includegraphics[scale=.9]{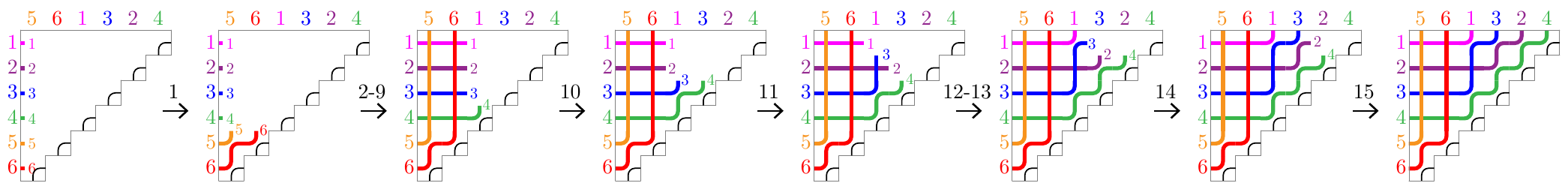}}
	\caption{Sweeping algorithm for the permutations~$\pi = 513264$ and~$\omega = 561324$.}
	\label{fig:sweepingAlgorithm1}
\end{figure}
\begin{figure}[t]
	\centerline{\includegraphics[scale=.9]{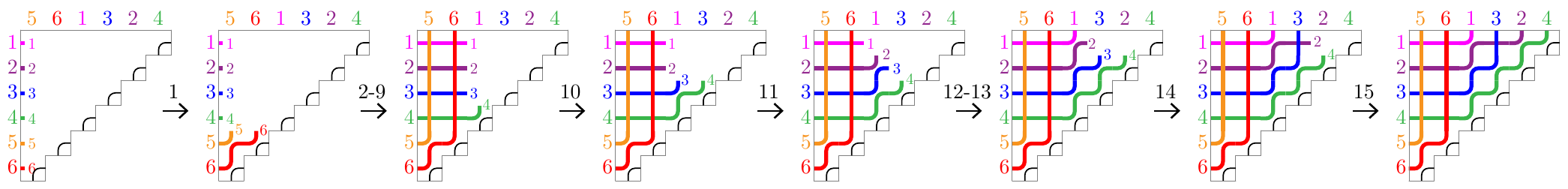}}
	\caption{Sweeping algorithm for the permutations~$\pi = 512364$ and~$\omega = 561324$.}
	\label{fig:sweepingAlgorithm2}
\end{figure}
\end{example}

\begin{proof}[Proof of \cref{prop:sweepingAlgorithm}]
When we sweep vertex~$v$ where pipe~$i$ arrives horizontally and pipe~$j$ arrives vertically,
\begin{itemize}
\item we have no choice but imposing a contact at~$v$ if~$i > j$ (since pipes $i$ and~$j$ already crossed before) or~$\omega^{-1}(i) < \omega^{-1}(j)$ (since pipes~$i$ and~$j$ do not cross at all),
\item if~$i < j$ and~$\omega^{-1}(i) > \omega^{-1}(j)$, then
\begin{itemize}
\item if~$\pi^{-1}(i) > \pi^{-1}(j)$ then pipes~$i$ and~$j$ cannot touch (otherwise $\pi$ would not be a linear extension of~$P$), so that we have no choice but imposing a crossing~at~$v$,
\item if~$\pi^{-1}(p) < \pi^{-1}(q)$ then
	\begin{itemize}
	\item if $v$ lies in column~$\omega^{-1}(j)$, then pipe $j$ needs to go straight north, and we have no choice but imposing a crossing at~$v$,
	\item otherwise, we have no choice but imposing a contact at~$v$ (otherwise \cref{lem:rectangle} ensures that~$j \contactLess{P} i$, so that~$\pi$ would not be a linear extension of~$P$).
	\qedhere
	\end{itemize}
\end{itemize}
\end{itemize}
\end{proof}


\subsection{Insertion algorithm}
\label{subsec:insertionAlgorithm}

Our second algorithm to construct~$\insertion{\pi}{\omega}$ is an \defn{insertion algorithm} inspired from~\cite{Pilaud-brickAlgebra} and similar to the insertion in binary search trees.

We call \defn{staircase} of length~$k$ a sequence~$e_1, \dots, e_k$ of southeast elbows such that~$e_i$ is located strictly southwest of~$e_{i+1}$ for each~$i \in [k-1]$. In other words, $r_1 > \dots > r_k$ and~$c_1 < \dots < c_k$ where~$e_i$ is located in row~$r_i$ and column~$c_i$.
For~$j \in [n]$ such that~$j > r_1$ and~$c_k < \omega^{-1}(j)$, we can uniquely define a pipe which enters at row~$j$, exits at column~$\omega^{-1}(j)$, and whose northeast elbows are precisely covering the southeast elbows~$e_1, \dots, e_k$.
Namely, it has a northeast elbow at row~$r_i$ and column~$c_i$ for each~$i \in [k]$, and a southeast elbow at row~$r_{i-1}$ and column~$c_i$ for each~$i \in [k+1]$, where by convention~$r_0 \eqdef j$ and~$c_{k+1} \eqdef \omega^{-1}(j)$.

\begin{proposition}
\label{prop:insertionAlgorithm}
For any permutations~$\pi,\omega \in \fS_n$ such that~$\pi \le \omega$ in weak order, the unique pipe dream~$P \in \acyclicPipeDreams(\omega)$ such that~$\pi \in \linearExtensions(P)$ can be constructed starting from the empty triangular shape and inserting the pipes~$\pi(1), \dots, \pi(n)$ one by one in the order of the permutation~$\pi$ as northwest as possible.
More precisely, at step~$t$, we insert a pipe starting at row~$\pi(t)$, ending at column~$\omega^{-1}(\pi(t))$, and whose northwest elbows are precisely covering the staircase of currently free southeast elbows in the rectangle~$[\pi(t)] \times [\omega^{-1}(\pi(t))]$.
See \cref{fig:insertionAlgorithm1,fig:insertionAlgorithm2}.
\end{proposition}

\begin{example}
\label{exm:insertionAlgorithm}
\cref{fig:insertionAlgorithm1,fig:insertionAlgorithm2} illustrate the insertion algorithm for the exiting permutation ${\omega = 561324}$ and the input permutations~${\pi = 513264}$ and~${\pi = 512364}$ respectively.
For instance, when inserting the (green) pipe~$4$ in the last step, the currently free southeast elbows are the first southeast elbow of the (blue) pipe~$3$ and the two southeast elbows of the (purple) pipe~$2$.
\begin{figure}[t]
	\centerline{\includegraphics[scale=.9]{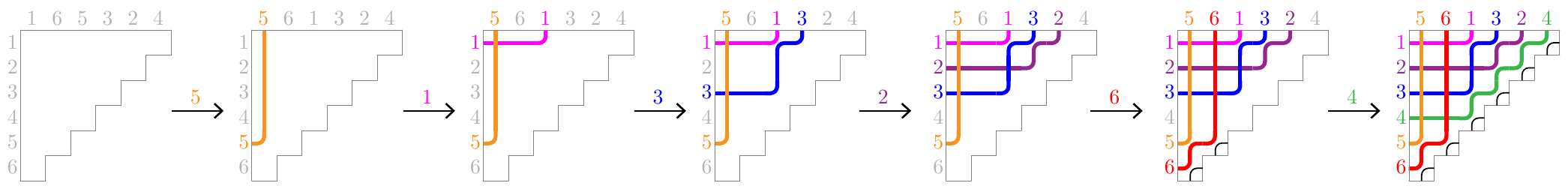}}
	\caption{Insertion algorithm for the permutations~$\pi = 513264$ and~$\omega = 561324$.}
	\label{fig:insertionAlgorithm1}
\end{figure}
\begin{figure}[t]
	\centerline{\includegraphics[scale=.9]{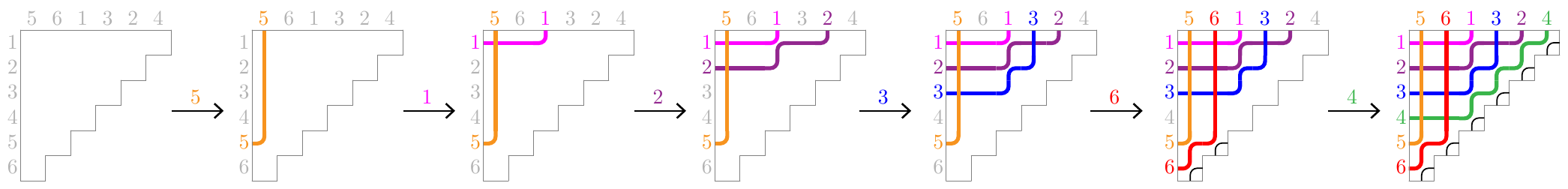}}
	\caption{Insertion algorithm for the permutations~$\pi = 512364$ and~$\omega = 561324$.}
	\label{fig:insertionAlgorithm2}
\end{figure}
\end{example}

We need to argue that this algorithm indeed creates a pipe dream of~$\pipeDreams(\omega)$.
To see it, we observe that the following two invariants are maintained throughout the algorithm.
We call \defn{hook} of a southeast elbow~$e$ the union of the horizontal segment west of~$e$ and the vertical segment north of~$e$.

\begin{lemma}
\label{lem:disjointPipesInsertionAlgorithm}
At any time during the insertion algorithm, the support of the pipes already inserted is precisely the union of the hooks of the currently free southeast elbows.
\end{lemma}

\begin{proof}
Immediate by induction, as the support of a pipe is precisely the union of the hooks of its southeast elbows minus the union of the pipes of its northeast elbows.
\end{proof}

\begin{corollary}
\label{coro:disjointPipesInsertionAlgorithm}
The pipes constructed by the insertion algorithm are disjoint, except at crossings and contacts.
\end{corollary}

\begin{lemma}
\label{lem:rectangleInsertionAlgorithm}
For any~$t, r, c \in [n]$, the free southeast elbows in the rectangle~$[r] \times [c]$ just before step~$t$ of the insertion algorithm form a staircase of length
\[
\max \big( 0, \# \set{s < t}{\pi(s) \le r \text{ and } \omega^{-1}(\pi(s)) \le c} - \# \set{s < t}{\pi(s) > r \text{ and } \omega^{-1}(\pi(s)) > c} \! \big).
\]
\end{lemma}

\begin{proof}
Denote by~$R$ the rectangle~$[r] \times [c]$.
The proof works by induction on~$t$.
Before step~$1$, there is no southeast elbow in~$R$.
Assume now that just before step~$t$, the free southeast elbows in~$R$ form a staircase~$e_1, \dots, e_k$ with~$k$ given by the formula of the statement.
At step~$t$, we insert a new pipe~$\pi(t)$ which enters at~$\pi(t)$ and exits at~$\omega^{-1}(\pi(t))$.

Assume first that the new pipe~$\pi(t)$ intersects the rectangle~$R$.
Let~${0 \le i < j \le k+1}$ be such that~$e_{i+1}, \dots, e_{j-1}$ are the northwest elbows of pipe~$\pi(t)$ in~$R$, and let~$e'_1, \dots, e'_\ell$ be the southeast elbows of pipe~$\pi(t)$ in~$R$.
Hence, the free southeast elbows in~$R$ after step~$t$ form the sequence~$e_1, \dots, e_i, e'_1, \dots, e'_\ell, e_j, \dots, e_k$.
We thus just need to see that $e_i$ is southwest of~$e'_1$ and~$e'_\ell$ is southwest of~$e_j$, and that
\[
\ell = \begin{cases} j-i  & \text{if } \pi(s) \le r \text{ and } \omega^{-1}(\pi(s)) \le c, \\  j-i-2  & \text{if } \pi(s) > r \text{ and } \omega^{-1}(\pi(s)) > c, \\  j-i-1 & \text{otherwise.} \end{cases}
\]
Since the northwest and southeast elbows alternate along pipe~$\pi(t)$, this follows from the fact that the pipe~$\pi(t)$ enters~$R$ with an horizontal step if~$\pi(t) \le r$ and with a vertical step if~$\pi(s) > r$, and exits~$R$ with a vertical step if~$\omega^{-1}(\pi(t)) \le c$ and with an horizontal step if~$\omega^{-1}(\pi(s)) > c$.

Finally, assume that the new pipe~$\pi(t)$ does not intersect the rectangle~$R$.
This implies that~$\pi(t) > r$ and~$\omega^{-1}(\pi(t)) > c$, and there are no free southeast elbows in~$R$ just before step~$t$.
Moreover, there is no~$s > t$ such that~$\pi(s) \le r$ and~$\omega^{-1}(\pi(s)) \le c$ (since any inversion of~$\pi$ is an inversion of~$\omega$), and there are no free southeast elbows in~$R$ at any step after~$t$.
We conclude that the formula still holds in this case.
\end{proof}

We now derive two properties of the insertion algorithm from \cref{lem:rectangleInsertionAlgorithm}.
Recall that we denote by~$\noninversions{\omega}{j} \eqdef \#\set{i \in [n]}{i < j \text{ and } \omega^{-1}(i) < \omega^{-1}(j)}$ the number of non-inversions of~$j$ in a permutation~$\omega$.

\begin{corollary}
\label{coro:rectangleInsertionAlgorithm1}
For any~$t \in [n]$, the free southeast elbows in the rectangle~$[\pi(t)] \times [\omega^{-1}(\pi(t))]$ just before step~$t$ of the insertion algorithm form a staircase of length~$\noninversions{\omega}{\pi(t)}$.
\end{corollary}

\begin{proof}
Observe first that if there is~$r < s$ such that~$\pi(r) > \pi(s)$ and~$\omega^{-1}(\pi(r)) > \omega^{-1}(\pi(s))$, then setting~$i \eqdef \pi(s)$ and~$j \eqdef \pi(r)$, we have~${i < j}$ and~$\pi^{-1}(i) > \pi^{-1}(j)$ while~$\omega^{-1}(i) < \omega^{-1}(j)$, which contradicts our assumption that~$\pi \le \omega$.
This implies that
\begin{align*}
\# \set{s < t}{\pi(s) < \pi(t) \text{ and } \omega^{-1}(\pi(s)) < \omega^{-1}(\pi(t))} & = \noninversions{\omega}{\pi(t)} \\
\text{and}\qquad \# \set{s < t}{\pi(s) > \pi(t) \text{ and } \omega^{-1}(\pi(s)) > \omega^{-1}(\pi(t))} & = 0.
\end{align*}
Hence, the result follows from~\cref{lem:rectangleInsertionAlgorithm} applied to the parameters~$t, \pi(t), \omega^{-1}(\pi(t))$.
\end{proof}

\begin{corollary}
\label{coro:rectangleInsertionAlgorithm2}
All pipes constructed by the insertion algorithm remain in the triangular shape.
\end{corollary}

\begin{proof}
Let~$1 \le i \le j \le n$.
Note that
\[
j - \noninversions{\omega}{j} = \#\set{i \in [n]}{i < j \text{ and } \omega^{-1}(i) > \omega^{-1}(j)} + 1 \le n - \omega^{-1}(j) + 1,
\]
so that~$\omega^{-1}(j) - \noninversions{\omega}{j} + j \ge n + 1$.
By~\cref{lem:rectangleInsertionAlgorithm}, the pipe~$j$ has~$\noninversions{\omega}{j}$ southeast elbows in total, at most~$j-i-1$ of which are strictly south of row~$i$, so that at least~$\noninversions{\omega}{j} - j + i$ of which are strictly north of row~$i$.
Hence, the eastmost point of pipe~$j$ in row~$i$ is at most in column~$\omega^{-1}(j) - \noninversions{\omega}{j} + j - i \le n - i + 1$.
We conclude that pipe~$j$ indeed remains in the triangular shape.
\end{proof}

\begin{proof}[Proof of \cref{prop:insertionAlgorithm}]
The insertion algorithm constructs a collection~$P$ of $n$ pipes in the triangular shape (by \cref{coro:rectangleInsertionAlgorithm2}), which are pairwise disjoint except at crossings and contacts (by \cref{coro:disjointPipesInsertionAlgorithm}).
For each~$t \in [n]$, the pipe~$\pi(t)$ enters at row~$\pi(t)$, exits at column~$\omega^{-1}(\pi(t))$, and has~$\noninversions{\omega}{\pi(t)}$ many southeast contacts (by \cref{coro:rectangleInsertionAlgorithm1}).
We thus conclude that~$P$ is a pipe dream of~$\pipeDreams(\omega)$ by a direct application of \cref{lem:characterizationPipeDreams}.

By construction, all southeast contacts of the pipe~$\pi(t)$ inserted at step~$t$ are in contact with northwest contacts of pipes~$\pi(s)$ inserted at steps~$s < t$.
In other words, all edges in the contact graph of~$\insertion{\pi}{\omega}$ are of the form~$\pi(s) \to \pi(t)$ for some~$s < t$.
It follows that the permutation~$\pi$ is a linear extension of the pipe dream~$P$.
\end{proof}


\subsection{Rewriting rule}
\label{subsec:rewritingRule}

Recall from \cref{exm:Tamari3} that the sylvester congruence can be defined as the transitive closure of the classical rewriting rule~$UjVikW \equiv UjVkiW$ where~$i < j < k$ are elements of~$[n]$ while~$U, V, W$ are (possibly empty) words on~$[n]$ (as usual, we write the permutations of~$\fS_n$ as words in one-line notation).
We now describe a similar rewriting rule for the pipe dream congruence~$\equiv_\omega$.

\begin{proposition}
\label{prop:rewritingRule}
On the interval~$[e,\omega]$ of the weak order, the pipe dream congruence~$\equiv_\omega$ coincides with the transitive closure of the rewriting rule~$U ij V \equiv_\omega U ji V$ where~$1 \le i < j \le n$ are elements of~$[n]$ while~$U, V$ are (possibly empty) words on~$[n]$ such that
\[
\# \set{k \in U}{k > i} \ge \# \set{k \in U}{\omega^{-1}(k) < \omega^{-1}(j)}.
\]
\end{proposition}

\begin{proof}
As they are linear extensions of posets, the congruence classes of $\equiv_\omega$ are connected by simple transpositions.
We thus just need to show that any two permutations~$\pi \eqdef UijV$ and~$\pi' \eqdef UjiV$ of~$[e, \omega]$ which differ by the inversion of two consecutive values are equivalent for~$\equiv_\omega$ if and only if~$\# \set{k \in U}{k > i} \ge \# \set{k \in U}{\omega^{-1}(k) < \omega^{-1}(j)}$.
Moreover, by \cref{prop:insertionAlgorithm}, $\pi \equiv_\omega \pi'$ if and only if they are sent to the same pipe dream by the insertion algorithm.
Let~$t \eqdef \pi^{-1}(i) = \pi'^{-1}(j)$.
Before step~$t$ of the insertion algorithm, we insert the word~$U$ both for~$\pi$ and for~$\pi'$.
The insertion of~$i$ and~$j$ then commute if and only if there is no currently free elbow in the rectangle~$[i] \times [\omega^{-1}(j)]$.
By \cref{lem:rectangleInsertionAlgorithm} applied to the parameters~$t, i, \omega^{-1}(j)$, this is equivalent to
\[
\# \set{s < t}{\pi(s) \le i \text{ and } \omega^{-1}(\pi(s)) \le \omega^{-1}(j)} \le \# \set{s < t}{\pi(s) > i \text{ and } \omega^{-1}(\pi(s)) > \omega^{-1}(j)}.
\]
or written differently,
\[
\# \set{k \in U}{k < i \text{ and } \omega^{-1}(k) < \omega^{-1}(j)} \le \# \set{k \in U}{k > i \text{ and } \omega^{-1}(k) > \omega^{-1}(j)}.
\]
(note that  to replace large by strict inequalities in the first set, we used that~$k \ne i$ and ${\omega^{-1}(k) \ne \omega^{-1}(j)}$ since~$i,j \notin U$ while~$k \in U$).
Finally, observe that adding~$\# \set{k \in U}{k > i \text{ and } \omega^{-1}(k) < \omega^{-1}(j)}$ from both side, we obtain the equivalent condition
\[
\# \set{k \in U}{\omega^{-1}(k) < \omega^{-1}(j)} \le \# \set{k \in U}{k > i}.
\qedhere
\]
\end{proof}

\begin{remark}
\label{rem:rewritingRule}
Observe from the proof that the condition of \cref{prop:rewritingRule} is equivalent to
\[
\# \set{k \in U}{k < i \text{ and } \omega^{-1}(k) < \omega^{-1}(j)} \le \# \set{k \in U}{k > i \text{ and } \omega^{-1}(k) > \omega^{-1}(j)}.
\]
Observe moreover that since~$\pi \le \omega$ and~$i < j$, we have
\begin{align*}
\set{k \in U}{k < i \text{ and } \omega^{-1}(k) < \omega^{-1}(j)} & = \set{k \in [i]}{\omega^{-1}(k) < \omega^{-1}(j)} \qquad \text{and}\\
\set{k \in U}{k > i \text{ and } \omega^{-1}(k) > \omega^{-1}(j)} & = \set{k \in [n]}{i < k < j \text{ and } \omega^{-1}(i) > \omega^{-1}(k) > \omega^{-1}(j)}.
\end{align*}
\end{remark}

We close this section by two immediate consequences of \cref{prop:rewritingRule,rem:rewritingRule}.

\begin{corollary}
\label{coro:patternAvoiding}
A permutation is minimal (resp.~maximal) in its pipe dream congruence class if and only if it avoids the patterns~$k_1 - \cdots - k_p - ji$ (resp.~$k_1 - \cdots - k_p - ij$) where~$k_q > i$ and~$\omega^{-1}(k_q) > \omega^{-1}(j)$ for~$q \in [p]$ and~$p = \# \set{k \in [i]}{\omega^{-1}(k) < \omega^{-1}(j)}$.
\end{corollary}

\begin{proof}
This is an immediate consequence of \cref{prop:rewritingRule}: a permutation~$\pi$ is minimal (resp.~maximal) in its $\equiv_\omega$-congruence class if and only if it contains no consecutive exchangeable entries~$ji$ (resp.~$ij$) with~$i < j$.
\end{proof}

\begin{corollary}
\label{coro:canopy}
If~$1 \le i < j \le n$ and~$j-i \le \#\set{k \in [i]}{\omega^{-1}(k) < \omega^{-1}(j)}$, then the pipes~$i$ and~$j$ are comparable for $\contactLess{P}$ in any acyclic pipe dream~$P \in \acyclicPipeDreams(\omega)$.
\end{corollary}

\begin{proof}
If pipes~$i$ and~$j$ were incomparable in~$\contactLess{P}$, there would be two permutations~$UijV \equiv_\omega UjiV$.
However, $j-i \le \#\set{k \in [i]}{\omega^{-1}(k) < \omega^{-1}(j)}$ implies that the condition of \cref{prop:rewritingRule} cannot hold, by \cref{rem:rewritingRule}
\end{proof}


\subsection{Recoils and canopy}
\label{subsec:canopy}

We now generalize the notion of recoils of permutations and of canopy of a binary trees to show a natural commutative diagram of lattice homomorphisms.
We start with the easy generalization of recoils of a permutation.
See \cref{exm:canopy} for an illustration.

\begin{definition}
\label{def:recoils}
Consider the graph~$G(\omega)$ with vertex set~$[n]$ and edge set
\[
\set{ij}{i < j \text{ and } j-i, \; \omega^{-1}(i) > \omega^{-1}(j) \le \#\set{k < i}{\omega^{-1}(k) < \omega^{-1}(j)}}.
\]
Let~$\acyclicOrientations(\omega)$ denote the set of acyclic orientations of~$G(\omega)$.
The \defn{$\omega$-recoils} of a permutation~$\pi \in [e,\omega]$ is the acyclic orientation~${\recoils{\pi}{\omega} \in \acyclicOrientations(\omega)}$ such that~$\pi$ is a linear extension of~$\recoils{\pi}{\omega}$.
\end{definition}

We now generalize the canopy of a binary tree.
Recall that the \defn{canopy} of a binary tree~$T$ with~$n$ internal nodes is the sign sequence~${\canopy{T} \in \{{-},{+}\}^{n-1}}$ defined by~$\canopy{T}_i = {-}$ if the $(i+1)$-th leaf of~$T$ is a left leaf and~$\canopy{T}_i = {+}$ if the $(i+1)$-th leaf of~$T$ is a right leaf.
Equivalently, $\canopy{T}_i = -$ if the node~$i$ of~$T$ is above the node~$i+1$ of~$T$ and~$\canopy{T}_i = +$ otherwise.
This map was already used by J.-L.~Loday in~\cite{LodayRonco, Loday}, but the name ``canopy'' was coined by X.~Viennot~\cite{Viennot}.
We now define a generalization of the canopy map for pipe dreams in~$\acyclicPipeDreams(\omega)$, using \cref{coro:canopy}.
See \cref{exm:canopy} for an illustration.

\begin{definition}
\label{def:canopy}
The \defn{canopy} of a pipe dream~$P \in \acyclicPipeDreams(\omega)$ is the orientation~$\canopy{P} \in \acyclicOrientations(\omega)$ where each edge~$ij$ is oriented~$i \to j$ if~$i \contactLess{P} j$ and~$j \to i$ if~$j \contactLess{P} i$.
\end{definition}

\begin{proposition}
\label{prop:latticeHomomorphisms}
The maps~$\recoils{\cdot}{\omega}$, $\insertion{\cdot}{\omega}$, and~$\canopy{\cdot}$ define the following commutative diagram of surjective lattice homomorphisms:
\[
\begin{tikzpicture}
  \matrix (m) [matrix of math nodes,row sep=1.2em,column sep=5em,minimum width=2em]
  {
     [e,\omega]  	&								& \acyclicOrientations(\omega)	\\
					& \acyclicPipeDreams(\omega) 	&								\\
  };
  \path[->] (m-1-1) edge node [above] {$\recoils{\cdot}{\omega}$} (m-1-3);
  \path[->>] (m-1-1) edge node [below] {$\insertion{\cdot}{\omega}\qquad$} (m-2-2.west);
  \path[->] (m-2-2.east) edge node [below] {$\quad\canopy{\cdot}$} (m-1-3);
\end{tikzpicture}
\]
\end{proposition}

\begin{proof}
Consider a permutation~$\pi$ and let~$i < j \in [n]$ be such that~$j-i \le \#\set{k < i}{\omega^{-1}(k) < \omega^{-1}(j)}$.
Assume that the edge~$ij$ is oriented from~$i$ to~$j$.
Then~$\pi^{-1}(i) < \pi^{-1}(j)$, thus the pipe~$i$ is inserted before the pipe~$j$ in~$\insertion{\pi}{\omega}$, so that~$i \contactLess{\insertion{\pi}{\omega}} j$ and there is also an arc from~$i$ to~$j$ in~$\canopy{\insertion{\pi}{\omega}}$.
\end{proof}

\begin{example}
For the permutation~$\omega = 1 n \cdots 2$, the graph~$G(\omega)$ has an isolated vertex~$1$ and a path~$2-3-\dots -n$. The orientations of this path can be seen as sign sequences, corresponding to the recoils of permutations or the canopies of binary trees.
More generally, for the permutation~${\omega = 1 \cdots k n \dots (k+1)}$, the graph~$G(\omega)$ consist of all edges~$ij$ for~$k+1 \le i < j \le n$ such that~$|i-j| \le k$, and the $\omega$-recoil map on permutations and canopy on acyclic pipe dreams were already considered in~\cite{Pilaud-brickAlgebra}.
\end{example}

\begin{example}
\label{exm:canopy}
For~$\omega = 31542$, the graph~$G(\omega)$ has a single edge~$4-5$.
We have grouped the permutations of~$[e,\omega]$ according to their $\omega$-recoils, and the pipe dreams of~$\acyclicPipeDreams(\omega)$ according to their canopy in \cref{fig:canopy}.
\begin{figure}
	\centerline{\includegraphics[scale=1]{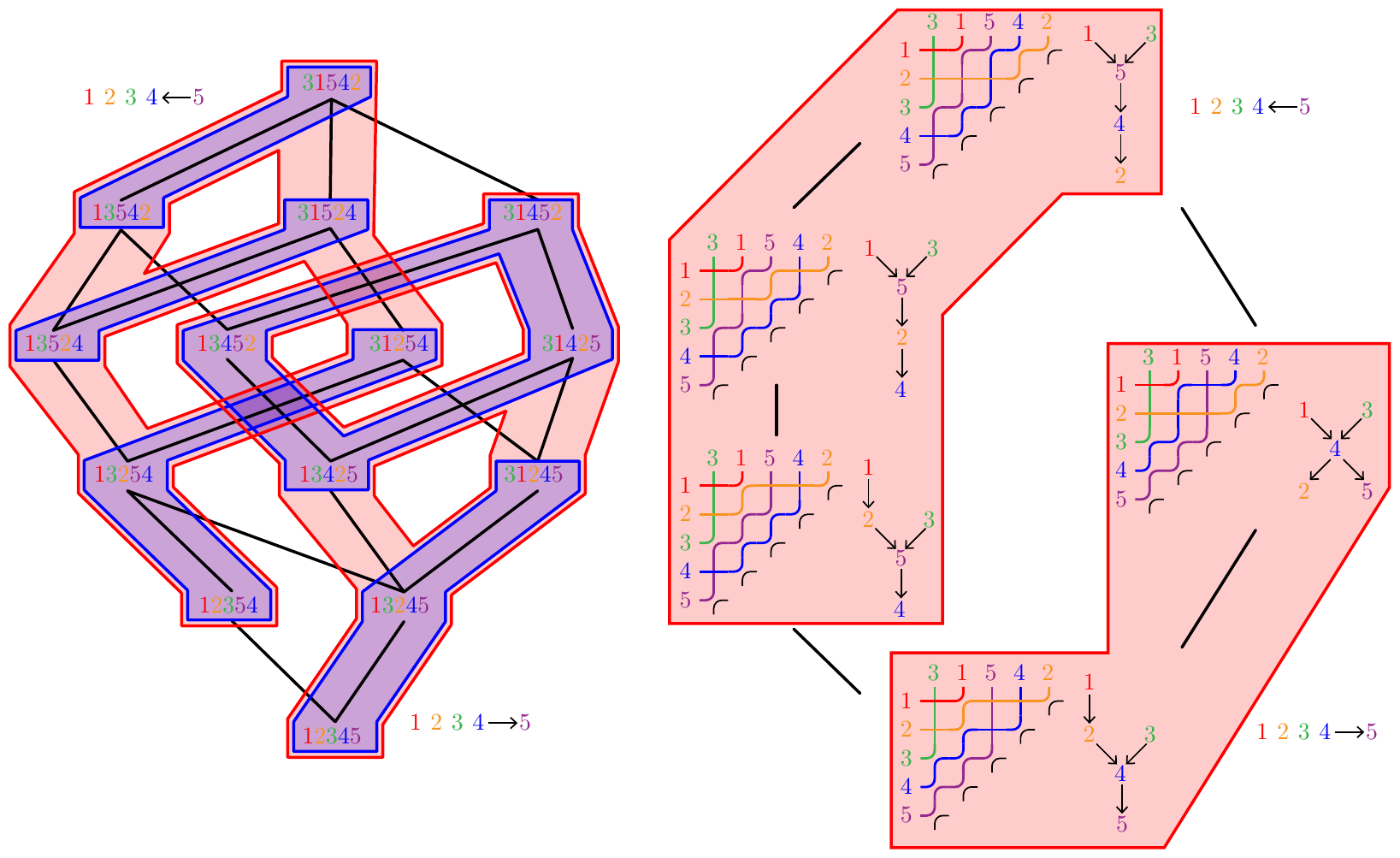}}
	\caption{The pipe dream congruence (blue bubbles) and the $31542$-recoil congruence (red bubbles) on the weak order interval~$[12345, 31542]$ (left) and the canopy congruence (red bubbles) on the increasing flip graph on acyclic pipe dreams (right). Each blue bubble is contained in a red bubble.}
	\label{fig:canopy}
\end{figure}
\end{example}


\subsection{The acyclic property for $\nu$-Tamari lattices} 
\label{subsec:nuTamari} 

We close this section with a theorem that relates our results to the $\nu$-Tamari lattices introduced by L.-F.~Pr\'eville-Ratelle and X.~Viennot in~\cite{PrevilleRatelleViennot}. These posets are indexed by a lattice path $\nu$ consisting of a finite number of north~(N) and east~(E) steps, and coincide with the classical Tamari lattices when $\nu=(NE)^n$.

In~\cite{CeballosPadrolSarmiento}, it was shown that the $\nu$-Tamari lattice can be obtained as the increasing flip poset on pipe dreams~$\pipeDreams(0\omega_\nu)$ for an explicit permutation $\omega_\nu$ associated with $\nu$. The permutations of the form~$\omega_\nu$ can be easily characterized as follows. We refer to~\cite{CeballosPadrolSarmiento} for details.

The \defn{Rothe diagram} of a permutation~$\omega$ is the set $\set{(\omega(j),i)}{i<j \text{ and } \omega(i)>\omega(j)}$ in matrix notation.
A permutation $\omega$ is called \defn{dominant} if its Rothe diagram is the Ferrers diagram of a partition located at the top left corner $(1,1)$. Equivalently, a permutation~$\omega$ is dominant if and only if it is 132-avoiding.
For a permutation~$\omega = \omega_1 \dots \omega_n \in \fS_n$, we denote by~$0\omega$ the permutation~$0 \omega_1 \dots \omega_n$ of~$\{0, 1, \dots, n\}$, and we consider here pipe dreams with pipes indexed by~${0, 1, \dots, n}$.
The $\nu$-Tamari lattice is the increasing flip poset on pipe dreams~$\pipeDreams(0\omega_\nu)$ for some dominant permutation $\omega_\nu$ determined by $\nu$, see~\cite{CeballosPadrolSarmiento}.

The main example is when $\omega \eqdef n \dots 2 1$ is the reverse permutation, see \cref{exm:Tamari1,fig:bijection}.
In this case, the increasing flip poset on~$\pipeDreams(0 \omega)$ is the classical Tamari lattice.
Removing vertex~$0$ from the contact graph of a pipe dream~$ P\in \pipeDreams(0 \omega)$ returns the binary tree corresponding to $P$ (with the leaves removed and the internal nodes labeled in inorder from $1$ to $n$).
The edges of the binary tree are oriented going away from the root of the tree.
In particular, all pipe dreams of~$\pipeDreams(0 \omega)$ are acyclic.
This extends to all dominant permutations as follows.

\begin{theorem}
\label{prob:nuAcyclicProperty}
All pipe dreams of $\pipeDreams(0\omega)$ are acyclic if and only if $\omega$ is a dominant permutation.
\end{theorem}

To prove this result, we will need the following observation.
Note that throughout the proof, all pipes, rows, and columns are indexed by indices ranging from $0$ to~$n$.

\begin{lemma}
\label{lem:notDominant}
For a pipe dream $P \in \pipeDreams(0\omega)$ with a crossing~$x$ between two pipes~$i < j$, if $i$ has an elbow southwest of~$x$ and $j$ has an elbow northeast of~$x$, then $\omega$ is not dominant.
\end{lemma}

\begin{proof}
Let~$r$ be the row and $c$ be the column of the crossing~$x$. 
Since $i$ has an elbow southwest of~$x$, we have $r \le i-1$.
Since $j$ has an elbow northeast of~$x$, there are at most $r-1 \le i-2$ pipes that cross north of $x$ in column~$c$.
Hence, at least two pipes smaller than~$i$ reach the northern border weakly before column~$c$.
Excluding the pipe~$0$, there is at least one pipe $h > 0$ such that $h < i < j$ and $\omega^{-1}(h) < \omega^{-1}(j) < \omega^{-1}(i)$.
This is a $132$ pattern in the permutation~$\omega$, therefore $\omega$ is not dominant.
\end{proof}

\begin{proof}[Proof of \cref{prob:nuAcyclicProperty}]
For the forward implication, assume that $\omega$ is not dominant, and consider~${i,j,k \in [n]}$ such that~$i < j < k$ while~$\omega^{-1}(i) < \omega^{-1}(k) < \omega^{-1}(j)$, and such that~$k$ is maximal for this property.
Consider any pipe dream~$P \in \pipeDreams(0\omega)$ obtained by capping any pipe dream of~$\pipeDreams(\omega)$ by a top row of contacts.
Because of this top row of contacts, $\omega^{-1}(u) < \omega^{-1}(v)$ implies~$u \contactLess{P} v$.
We now consider the eastmost point~$x$ of pipe~$k$ along row~$i$, and let~$\ell$ denote the other pipe at this point.
We distinguish two cases:
\begin{enumerate}[(i)]
\item Assume first that~$k$ and~$\ell$ cross at~$x$. Note that $k$ is vertical while $\ell$ is horizontal at~$x$. Hence, $\omega^{-1}(k) < \omega^{-1}(\ell)$,  so that~$k \contactLess{P} \ell$. Moreover, $\ell \ne i$ (since~$i$ and~$k$ do not cross), so that~$\ell$ must have an elbow west of~$x$ along row~$i$. As $k$ must have an elbow south of~$x$, we obtain that~$\ell \contactLess{P} k$ by~\cref{lem:rectangle}. Hence, we have both~$k \contactLess{P} \ell$ and~$\ell \contactLess{P} k$, so that~$P$ is cyclic.
\item Assume now that~$k$ and~$\ell$ touch at~$x$. Hence~$k \contactLess{P} \ell$. If~$\ell < k$, then~$k$ and~$\ell$ must cross at~$y$ before~$x$. Flipping~$x$ to~$y$, we obtain a pipe dream~$P'$ with~$\ell \contactLess{P'} k$ (because of the contact~$y$) and~$k \contactLess{P'} \ell$ (because of the last row of elbows which is still in~$P'$), so that~$P'$ would be cyclic. If~$\omega^{-1}(\ell) < \omega^{-1}(k)$, then~$\ell \contactLess{P} k$, so that~$P$ is again cyclic. Hence, we can assume that~$k < \ell$ and~$\omega^{-1}(k) < \omega^{-1}(\ell)$. By maximality of~$k$, we obtain that~$j$ and~$\ell$ do not cross, and thus that~$j$ must have an elbow northwest of~$x$. Since $x$ contain an elbow for~$k$, \cref{lem:rectangle} yields~$j \contactLess{P} k$. Since~$\omega^{-1}(k) < \omega^{-1}(j)$, we also have~$k \contactLess{P} j$, so that~$P$ is cyclic.
\end{enumerate}

For the backward implication, we introduce a local notation.
We write~$i \to j$ if there is an elbow of pipe~$i$ weakly northwest of an elbow of pipe~$j$.
Note that~$i \to j$ implies~$i \contactLess{P} j$ by \cref{lem:rectangle}.
The reverse direction does not necessarily holds, but clearly holds for cover relations of~$\contactLess{P}$ (as in this case, there is an elbow of~$i$ in contact with an elbow of~$j$).
Assume now that there is a cyclic pipe dream~$P \in \pipeDreams(0\omega)$.
We can thus find a cycle~$\b{C}=(i_1, \dots, i_m)$ with~$i_1 \to \dots \to i_m \to i_1$.
We can moreover assume that~$\b{C}$ is reduced in the sense that no subsequence of~$\b{C}$ forms a cycle for~$\to$.

If~$m = 2$, then the pipes~$i_1$ and~$i_2$ must cross at some point~$x$.
Without loss of generality, we can assume $i_1<i_2$. 
The fact $i_1 \to i_2$ implies that $i_1$ has an elbow southwest of~$x$, while the fact that $i_2 \to i_1$ implies that $i_2$ has an elbow northeast of~$x$, so that $\omega$ is not dominant by \cref{lem:notDominant}.

We now show that if~$\b{C}$ is reduced then the case $m\geq 3$ is impossible.
Assume that~$m \ge 3$.
Assume moreover that~$i_1 < i_\ell$ for all~$\ell \in [m]$ (if not, rotate the indices).
Since~$i_m \to i_1$, there is an elbow~$e$ of~$i_m$ weakly northwest of an elbow~$e'$ of~$i_1$.
Since~$i_1 \to i_2$, there is an elbow~$f$ of~$i_1$ weakly northwest of an elbow~$f'$ of~$i_2$.
We now distinguish four cases:
\begin{enumerate}[(i)]
\item If~$e$ is weakly northwest of~$f$, thus of~$f'$, then~$i_m \to i_2$, contradicting the minimality of~$\b{C}$.
\item If~$e$ is northeast of~$f$, since~$i_1 < i_m$, the pipe~$i_m$ starts southwest of~$f$ and ends northeast of~$f$, so that it has an elbow either southeast of~$f$ so that~$i_1 \to i_m$, or northwest of~$f$ thus of~$f'$ so that~$i_m \to i_2$. In both cases, this contradicts the minimality of~$\b{C}$.
\item If~$e$ is southwest of~$f$, as~$i_1$ passes trough~$e'$ and~$f$, and $i_2$ starts south of~$i_1$ (since~$i_1 < i_2$) and passes south of~$f$, we obtain that~$i_2$ must pass south of~$e$ (otherwise $i_2$ would cross $i_1$ twice), so that $i_m \to i_2$, which contradicts the minimality of~$\b{C}$.
\item If~$e$ is weakly southeast of~$f$, then~$e'$ is southeast of~$f$, which is impossible as $i_1$ has to pass through both~$e'$ and~$f$.
\qedhere
\end{enumerate}
\end{proof}

Applying our \cref{thm:pipeDreamQuotient}, and using the isomorphism $[0e,0\omega_\nu]\cong [e,\omega_\nu]$, we get the following consequence.

\begin{corollary}
The $\nu$-Tamari lattice is a lattice quotient of the interval~$[e,\omega_\nu]$.
\end{corollary}  


\section{Subword complexes}
\label{sec:subwordComplexes}

The objective of this section is to partially extend our results about the lattice structure of acyclic pipe dreams (\cref{sec:latticeAcyclicPipeDreams}) to the wider context of subword complexes in finite Coxeter groups.
We start with some basic preliminaries on finite Coxeter groups (\cref{subsec:finiteCoxeterGroups}) and subword complexes (\cref{subsec:subwordComplexes}).
We then define linear extensions of facets of subword complexes (\cref{subsec:linearExtensionsFacets}) and present two theorems (\cref{subsec:twoTheorems}) and five conjectures (\cref{subsec:fiveConjectures}) about them.
To conclude, we present a sweeping algorithm to construct the facet with a given linear extension (\cref{subsec:sweepingAlgorithmSubwordComplexes})


\subsection{Finite coxeter groups} 
\label{subsec:finiteCoxeterGroups}

We refer to~\cite{BjornerBrenti, Humphreys} for detailed references on Coxeter groups. 
We consider a \defn{finite root system} $\Phi$ with \defn{positive roots}~$\Phi^+$, \defn{negative roots}~$\Phi^-$, and \defn{simple roots} $\Delta \subseteq \Phi^+$.
The reflections along the hyperplanes orthogonal to the roots in $\Phi$ generate a \defn{finite Coxeter group} $W$.
We denote by~$s_\alpha \in W$ the reflection orthogonal to a root~$\alpha \in \Phi$, and by~$\alpha_s \in \Phi^+$ the positive root orthogonal to a reflection~$s \in W$.
The group~$W$ is actually generated by the simple reflections~$S = \set{s_\alpha}{\alpha \in \Delta}$, and the pair~$(W,S)$ is a \defn{finite Coxeter system}.
Finite root systems and finite Coxeter systems are classified in terms of Dynkin diagrams, see~\cite{Humphreys}. 

The \defn{inversion set}~$\Inv(\omega)$ and the \defn{non-inversion set}~$\Ninv(\omega)$ of an element~$\omega \in W$ are
\[
\Inv(\omega) \eqdef \Phi^+ \cap \omega(\Phi^-)
\qquad\text{and}\qquad
\Ninv(\omega) \eqdef \Phi^+ \cap \omega(\Phi^+).
\]
Note that
\[
\Phi^+ = \Inv(\omega) \sqcup \Ninv(\omega)
\qquad\text{and}\qquad
\omega(\Phi^+) = -\Inv(\omega) \sqcup \Ninv(\omega).
\]

A \defn{reduced expression} of~$\omega$ is a product~$s_1 s_2 \dots s_\ell = \omega$ of simple reflections~$s_i \in S$ such that~$\ell$ is minimal.
This minimal~$\ell$ is the \defn{length} $\ell(\omega)$ of~$\omega$.
The length of~$\omega$ coincides with the size of its inversion set~$\Inv(\omega)$ since $\Inv(\omega) \eqdef \{\alpha_{s_1}, s_1(\alpha_{s_2}), s_1s_2(\alpha_{s_3}), \dots, s_1s_2 \dots s_{\ell-1}(\alpha_{s_\ell})\}$.

The \defn{weak order} on $W$ is the partial order $\le$ defined by $\sigma \leq \omega$ if there exists $\tau \in W$  such that $\sigma\tau = \omega$ and $\ell(\sigma) + \ell(\tau) = \ell(\omega)$.
In other words, the element $\omega$ has a reduced expression with a prefix which is a reduced expression of $\sigma$.
Equivalently, the weak order corresponds to the inclusion order on inversion sets, that is $\sigma \leq \omega$ if and only if $\Inv(\sigma) \subseteq \Inv(\omega)$.
This order defines a lattice structure on the elements of~$W$.
The minimal element is the identity $e \in W$ and the maximal element is the \defn{unique longest element} $\wo$ of~$W$.
Note that~$\Inv(e) = \varnothing = \Ninv(\wo)$ and~$\Ninv(e) = \Phi^+ = \Inv(\wo)$, so that~$\ell(e) = 0$ while~$\ell(\wo) = |\Phi^+|$.

\begin{example}
\label{exm:typeACoxeterSystem}
The Coxeter system of type~$A_{n-1}$ is the symmetric group~$W = \fS_n$ with generators~$S = \set{\tau_i}{i \in [n-1]}$, where~$\tau_i$ is the simple transposition~$\tau_i = (i \; i+1)$.
It naturally acts on~$\R^n$ by permutations of coordinates.
Denoting by~$(\b{e}_i)_{i \in [n]}$ the canonical basis of~$\R^n$, the type~$A_{n-1}$ roots are all~$\b{r}_{i,j} \eqdef \b{e}_i - \b{e}_j$ for distinct~$i,j \in [n]$, with positive roots~$\b{r}_{i,j}$ for~$1 \le i < j \le n$ and simple roots~$\b{r}_{i,i+1}$ for~$i \in [n-1]$.
The inversion set of~$\pi \in \fS_n$ is the set of roots~$\b{e}_i - \b{e}_j$ for the inversions~$(i,j)$ of the permutation~$\pi$.
The longest element is the reversed permutation~$[n, n-1, \dots, 2, 1]$ (written in one line notation), and it admits the reduced expression~$\tau_1 \cdots \tau_{n-1} \tau_1 \cdots \tau_{n-2} \cdots \tau_1 \tau_2 \tau_1$.
\end{example}


\subsection{Subword complexes}
\label{subsec:subwordComplexes}

Motivated by their study of Gröbner geometry of Schubert varieties~\cite{KnutsonMiller-GroebnerGeometry}, A.~Knutson and E.~Miller introduced in~\cite{KnutsonMiller-subwordComplex} the following remarkable family of simplicial complexes in the context of Coxeter groups.

Let $(W,S)$ be a finite Coxeter system, $Q=(q_1,\dots,q_m)$ be a word in the simple reflections $S$, and $\omega \in W$ be an element of the group.
For $J\subseteq [m]$, we denote by $Q_J$ the subword of~$Q$ consisting of the letters with positions in $J$.
The \defn{subword complex} $\subwordComplex(Q,\omega)$ is the simplicial complex whose facets are subsets $I\subseteq [m]$ such that $Q_{[m]\setminus I}$ is a reduced expression for $\omega$.
We denote by $\subwordFacets(Q,\omega)$ the set of facets of $\subwordComplex(Q,\omega)$.

It is known that~$\subwordComplex(Q,\omega)$ is either a ball or a sphere, in particular it is a pseudomanifold (with or without boundary).
The flip graph of~$\subwordComplex(Q,\omega)$ is the graph whose vertices are the facets of~$\subwordComplex(Q,\omega)$ and whose edges are the ridges of~$\subwordComplex(Q,\omega)$.
In other words, two facets $I,J \in \subwordFacets(Q,\omega)$ are connected by a \defn{flip} if they differ one element: $I \ssm \{i\} = J \ssm \{j\}$ for some $i\in I$ and $j\in J$ with~$i \ne j$.
The flip from $I$ to $J$ is called \defn{increasing} if~$i < j$, and \defn{decreasing} otherwise.
The \defn{increasing flip graph} on the facets of the subword complex is an acyclic graph which has a unique source and a unique sink~\cite{Pilaud-greedyFlipTree, PilaudStump-ELlabelings}.
These two special facets are called the \defn{greedy facet} $\greedyFacet$ and the \defn{antigreedy facet} $\antiGreedyFacet$.
They are the lexicographically smallest and largest facets of $\subwordComplex(Q,\pi)$.
The \defn{increasing flip poset} is the transitive closure of the increasing flip graph.

An important tool to study subword complexes are the root functions introduced in~\cite{CeballosLabbeStump}.
For a facet $I$ of~$\subwordComplex(Q,\omega)$, the \defn{root function} $\rootFunction{I}{\cdot}:[m]\rightarrow \Phi$ sends a position~$k$ in the word $Q$ to the root~$ \rootFunction{I}{k}$ defined by 
\[
\rootFunction{I}{k} \eqdef \prod Q_{[k-1]\ssm I}(\alpha_{q_k}),
\]
where~$\prod Q_{[k-1]\ssm I}$ is the product of the letters~$q_i$ for~$i \in [k-1] \ssm I$ computed in the natural order.
The \defn{root configuration} is the set $\Roots(I)=\set{\rootFunction{I}{i}}{i\in I}$.
The facet $I$ is called \defn{acyclic} if $\cone \Roots(I)$ is a pointed cone.
We denote by $\subwordAcyclicFacets(Q,\omega)$ the set of acyclic facets of the subword complex~$\subwordComplex(Q,\omega)$.
We will use the following statement that enables to perform flips using the root function.

\begin{lemma}[{\cite{CeballosLabbeStump, KnutsonMiller-subwordComplex}}]
\label{lem:rootFunctionFlips}
Let $I$ be a facet of the subword complex~$\subwordComplex(Q,\omega)$.
Then
\begin{enumerate}
\item $i \in I$ is flippable if and only if $\pm \rootFunction{I}{i} \in \Inv(\omega)$.
\item If $i \in I$ is flippable, it can be flipped to the unique $j\in [m]\ssm I$ such that $\rootFunction{I}{j} = \pm\rootFunction{I}{i}$.
The flip is increasing ($i<j$) when $\rootFunction{I}{i}\in \Phi^+$ and decreasing ($i>j$) when $\rootFunction{I}{i}\in \Phi^-$.
\label{lem:rootFunctionFlips2}
\item If $I$ and $J$ are two facets related by a flip, with $I \ssm \{i\} = J \ssm  \{j\}$ and $i < j$, then
\[
\rootFunction{J}{k} = 
\begin{cases}
s_\beta(\rootFunction{I}{k}), & \text{for } i < k \leq j \\
\rootFunction{I}{k}, & \text{otherwise}
\end{cases}
\]
where $\beta \eqdef \rootFunction{I}{i}$ and $s_\beta\in W$ is the reflection orthogonal to the root $\beta$.
\label{lem:rootFunctionFlips3}
\item $i \in I$ is not flippable if and only if $\rootFunction{I}{i}\in \Ninv(\omega)$.
\end{enumerate}
\end{lemma}

\begin{example}
\label{exm:typeARootConfiguration}
Continuing \cref{exm:typeACoxeterSystem}, we consider the type~$A_{n-1}$ Coxeter system, the word $Q \eqdef \tau_{n-1} \cdots \tau_1 \tau_{n-1} \cdots \tau_2 \cdots \tau_{n-1} \tau_{n-2} \tau_{n-1}$, and a permutation~$\omega \in \fS_n = A_{n-1}$.
The word~$Q$ naturally fits on an $n \times n$ triangular grid (place $\tau_k$ in all boxes with row~$i$ and column~$j$ such that~$i+j = k+1$). 
Moreover, each facet~$I$ of the subword complex~$\subwordComplex(Q, \omega)$ corresponds to a pipe dream~$P_I$ of~$\pipeDreams(\omega)$ (replace each position in~$I$ by a contact in~$P_I$, and the other positions by crossings in~$P_I$).
The root function is given~$\rootFunction{I}{i} = \b{r}_{p,q} \eqdef \b{e}_q - \b{e}_p$ where~$p$ is the pipe arriving from the west and $q$ is the pipe arriving from the south at the box of~$P_I$ corresponding to position~$i$ of~$Q$.
Hence, the root configuration is the incidence configuration~$\Roots(I) = \bigset{\b{r}_{p,q}}{(p,q) \in P_I\contact}$ of the contact graph of~$P_I$.
In particular, the acyclic facets of~$\subwordComplex(Q, \omega)$ correspond to the acyclic pipe dreams of~$\acyclicPipeDreams(\omega)$.

For a specific illustration, consider the word~$Q \eqdef \tau_6 \tau_5 \tau_4 \tau_3 \tau_2 \tau_1 \tau_6 \tau_5 \tau_4 \tau_3 \tau_2 \tau_6 \tau_5 \tau_4 \tau_3 \tau_6 \tau_5 \tau_4 \tau_6 \tau_5 \tau_6$ of~${A_6 = \fS_7}$ and the permutation~$\pi \eqdef 1365724$.
Then~$Q$ fits in the triangle of \cref{fig:pipeDreams}.
Moreover, the facets~$I \eqdef \{1,2,3,4,5,6,7,8,9,10,12,16,21\}$ and~$I' \eqdef \{1,2,4,5,6,7,8,9,10,12,16,17,21\}$ \linebreak of~$\subwordComplex(Q, \pi)$ give the pipe dreams of \cref{fig:pipeDreams}.
The roots~$\rootFunction{I}{3} = \alpha_{q_3} = \alpha_4 = \b{e}_5 - \b{e}_4$ and~$\rootFunction{I'}{17} = q_3 q_{11} q_{13} q_{14} q_{15} (\alpha_{q_{17}}) = \tau_4 \tau_2 \tau_5 \tau_4 \tau_3 (\alpha_5) = - \alpha_4 = \b{e}_4 - \b{e}_5$ correspond to the pipes~$4$ and~$5$ at the corresponding crossings.
\end{example}

\begin{example}
\label{exm:sortingNetworks}
More generally, for the type~$A_n$ Coxeter system, the word~$Q$ can be represented by a sorting network~$\cal{N}$, a facet~$I$ of the subword complex~$\subwordComplex(Q,w)$ can be represented by a pseudoline arrangement~$P_I$ on the network~$\cal{N}$, and the root configuration~$\Roots(I)$ is again the incidence configuration of the contact graph of~$P_I$.
We refer to~\cite{PilaudPocchiola, PilaudSantos-brickPolytope, PilaudStump-brickPolytope} for details of this representation and use it in \cref{fig:subwordComplex1,fig:subwordComplex2,fig:subwordComplex3,fig:subwordComplex4,fig:subwordComplex5}.
\end{example}

\begin{example}
For any finite Coxeter group~$W$ and any Coxeter element~$c \in W$ (a product of all generators of~$S$ in a given arbitrary order), let~$\wo(c)$ denote the $c$-sorting word of~$\wo$ (the lexicographically minimal reduced expression for~$\wo$ in~$c^\infty$, see \cite{Reading-CambrianLattices} for details), and consider the concatenation~$c\wo(c)$.
Extending the observation of \cref{exm:Tamari1}, it was shown in~\cite{CeballosLabbeStump} that the subword complex~$\subwordComplex(c\wo(c), \wo)$ is isomorphic to the cluster complex of type~$W$.
In particular, the increasing flip graph is the Hasse diagram of the $c$-Cambrian lattice of~\cite{Reading-CambrianLattices}.
See \cref{fig:increasingFlipPosets}\,(left) and \cref{fig:subwordComplex1} for an example with~$c = \tau_2\tau_1\tau_3$ in type~$A_3$.
\end{example}

\begin{example}
\label{exm:counterExampleLattice}
In contrast to the previous example, we have already seen in \cref{rem:increasingFlipPosetNotLattice,rem:acyclicIncreasingFlipPosetNotLattice} that neither the increasing flip poset on all facets, nor its subposet induced by the acyclic facets, are lattices in general, even in type~$A$ and even for pipe dreams.
Here is another example in type~$A$, but in a case where all facets are actually acyclic.
Consider the word~$Q = \tau_1 \tau_2 \tau_3 \tau_2 \tau_1 \tau_2 \tau_3 \tau_2 \tau_1$ on the simple generators of the symmetric group~$\fS_4$ and the subword complex~$\subwordComplex(Q,\wo)$.
The increasing flip lattice of this subword complex is represented in \cref{fig:increasingFlipPosets}\,(right) and \cref{fig:subwordComplex3}.
In \cref{fig:increasingFlipPosets}\,(right), we have highlighted two blue facets which have no join and two red facets which have no meet, proving that it is not a lattice (this was observed in~\cite[Rem.~5.12]{PilaudStump-brickPolytope}).
\begin{figure}[t]
	\centerline{
	\begin{tikzpicture}[xscale=1.7,yscale=1.25,->]
	\node (123) at (0,-3) {$\{1,2,3\}$};
	\node (139) at (-2,-2) {$\{1,3,9\}$};
	\node (234) at (0,-2) {$\{2,3,4\}$};
	\node (128) at (1.5,-2) {$\{1,2,8\}$};
	\node (345) at (-1,-1) {$\{3,4,5\}$};
	\node (246) at (1,-1) {$\{2,4,6\}$};
	\node (359) at (-2.4,0) {$\{3,5,9\}$};
	\node (189) at (-1.2,0) {$\{1,8,9\}$};
	\node (456) at (0,0) {$\{4,5,6\}$};
	\node (268) at (2,0) {$\{2,6,8\}$};
	\node (567) at (0,1) {$\{5,6,7\}$};
	\node (579) at (-1,2) {$\{5,7,9\}$};
	\node (678) at (1,2) {$\{6,7,8\}$};
	\node (789) at (0,3) {$\{7,8,9\}$};
	\draw (123)--(139);
	\draw (123)--(234);
	\draw (123)--(128);
	\draw (139)--(359);
	\draw (139)--(189);
	\draw (234)--(345);
	\draw (234)--(246);
	\draw (128)--(189);
	\draw (128)--(268);
	\draw (345)--(359);
	\draw (345)--(456);
	\draw (246)--(456);
	\draw (246)--(268);
	\draw (359)--(579);
	\draw (189)--(789);
	\draw (456)--(567);
	\draw (268)--(678);
	\draw (567)--(579);
	\draw (567)--(678);
	\draw (579)--(789);
	\draw (678)--(789);
	\end{tikzpicture}
	\quad
	\begin{tikzpicture}[xscale=2,yscale=1.5,->]
	\node (123) at (0,-2.5) {$\{1,2,3\}$};
	\node (236) at (-1,-1.5) {$\{2,3,6\}$};
	\node (134) at (0,-1.5) {$\blue{\{1,3,4\}}$};
	\node (129) at (1.2,-1.5) {$\blue{\{1,2,9\}}$};
	\node (346) at (-1.8,-.5) {$\{3,4,6\}$};
	\node (148) at (-.3,-.5) {$\{1,4,8\}$};
	\node (269) at (.7,-.5) {$\{2,6,9\}$};
	\node (467) at (-1.8,.5) {$\{4,6,7\}$};
	\node (189) at (1.5,.5) {$\red{\{1,8,9\}}$};
	\node (478) at (-1,1.5) {$\{4,7,8\}$};
	\node (679) at (0,1.5) {$\red{\{6,7,9\}}$};
	\node (789) at (0,2.5) {$\{7,8,9\}$};
	\draw (123)--(236);
	\draw (123)--(134);
	\draw (123)--(129);
	\draw (236)--(346);
	\draw (134)--(346);
	\draw (134)--(148);
	\draw (236)--(269);
	\draw (129)--(269);
	\draw (346)--(467);
	\draw (148)--(189);
	\draw (129)--(189);
	\draw (148)--(478);
	\draw (467)--(478);
	\draw (269)--(679);
	\draw (467)--(679);
	\draw (189)--(789);
	\draw (478)--(789);
	\draw (679)--(789);
	\end{tikzpicture}
	}
	\caption{The increasing flip poset on~$\subwordComplex(Q, \wo)$ where~$Q = \tau_2 \tau_1 \tau_3 \tau_2 \tau_1 \tau_3 \tau_2 \tau_1 \tau_3$ (left) and~$Q = \tau_1 \tau_2 \tau_3 \tau_2 \tau_1 \tau_2 \tau_3 \tau_2 \tau_1$ (right) in type~$A_3$. The left one is the $\tau_2 \tau_1 \tau_3$-Cambrian lattice, while the right one is not a lattice (the two facets highlighted in blue have no join while the two facets highlighted in red have no meet). In both examples, all facets are acyclic.}
	\label{fig:increasingFlipPosets}
\end{figure}
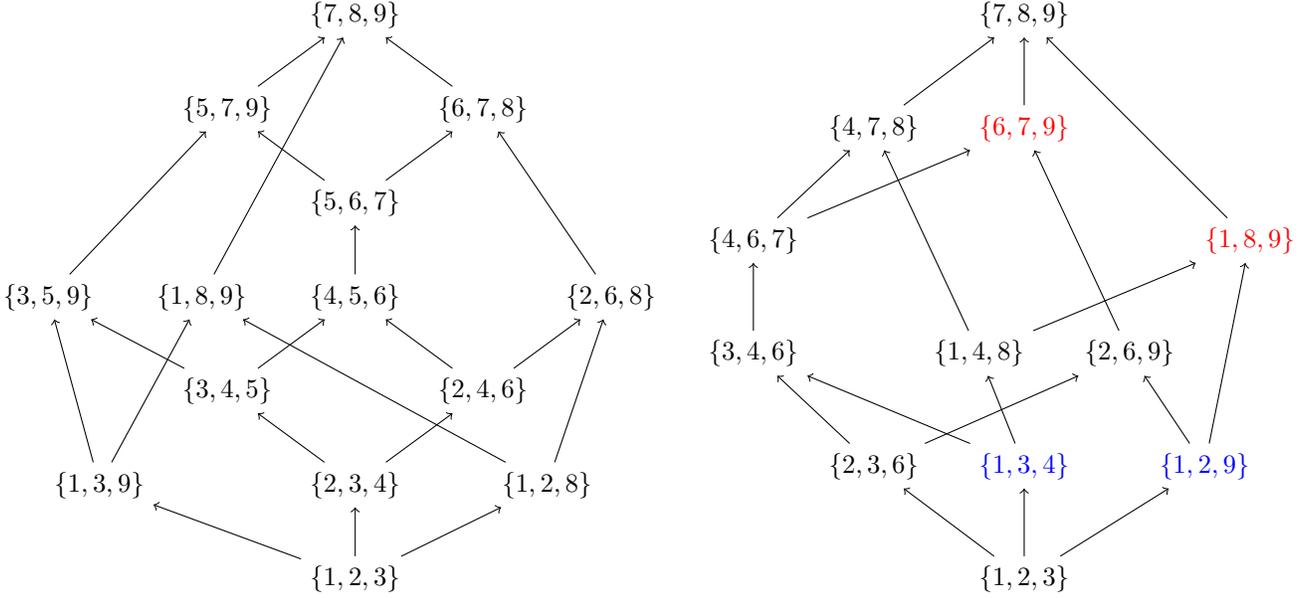
\end{example}


\pagebreak
\subsection{Linear extensions of facets}
\label{subsec:linearExtensionsFacets}

We now introduce the analogue of \cref{def:linearExtensions} for subword complexes.

\begin{definition}
\label{def:linearExtensionsSubwordComplexes}
Let $\subwordComplex(Q,\omega)$ be a non-empty subword complex and $I\in \subwordComplex(Q,\omega)$ be a facet.
A \defn{linear extension} of $I$ is an element $\pi \in W$ such that~$\Roots(I) \subseteq \pi(\Phi^+)$.
We denote by~$\linearExtensions(I)$ the set of linear extensions of $I$, and by
\[
\linearExtensions(Q,\omega) \eqdef \bigcup_{I\in\subwordFacets(Q,\omega)} \linearExtensions(I)
\]
the set of linear extensions of all facets of~$\subwordComplex(Q,\omega)$.
\end{definition}

\begin{example}
In the situation of \cref{exm:typeACoxeterSystem,exm:typeARootConfiguration}, the linear extensions of a facet~$I$ of~$\subwordComplex(Q, \omega)$ are precisely the linear extensions of the pipe dream~$P_I$ of~$\pipeDreams(\omega)$.
\end{example}

\begin{lemma}
\label{lem:linearExtensionsAcyclic}
A facet~$I$ is acyclic if and only if~$\linearExtensions(I)\neq \varnothing$.
\end{lemma}

\begin{proof}
A classical result for finite root systems states that for any generic linear halfspace~$H^+$ the intersection~$\Phi \cap H^+$ is of the form $\pi(\Phi^+)$ for some $\pi \in W$.
Since any pointed cone is contained in some generic linear halfspace, we obtain
\[
\begin{array}[b]{rcl}
I \text{ is acyclic}  & \longleftrightarrow  &  \cone \Roots(I) \text{ is pointed} \\
& \longleftrightarrow  & \cone \Roots(I)\subseteq \pi(\Phi^+) \text{ for some } \pi\in W \\
& \longleftrightarrow  &  \Roots(I)\subseteq \pi(\Phi^+) \text{ for some } \pi\in W \\
& \longleftrightarrow  &  \linearExtensions(I) \neq \varnothing 
\end{array}
\qedhere
\]
\end{proof}

The following lemma regards linear extensions of the greedy and antigreedy facets.

\begin{lemma}
\label{lem:linearExtensionsGreedyFacets}
Let $\greedyFacet$ and $\antiGreedyFacet$ be the greedy and antigreedy facets of $\subwordComplex(Q,\omega)$, respectively.
Then 
\begin{enumerate}
\item $e\in \linearExtensions(\greedyFacet)$.
\item $\omega \in \linearExtensions(\antiGreedyFacet)$.
\item If $e\in \linearExtensions(I)$ then $I=\greedyFacet$.
\end{enumerate}
\end{lemma}

\begin{proof}
The greedy facet $\greedyFacet$ is the unique facet for which every flip (if any) is increasing.
By \cref{lem:rootFunctionFlips}~\eqref{lem:rootFunctionFlips2}, this implies that $\greedyFacet$ is the unique facet $I$ such that $\Roots(I)\subseteq \Phi^+ = e(\Phi^+)$.
This proves parts (1) and (3) of the Lemma.   
    
For part (2) we need to analyze the possibilities for the set $\Roots(\antiGreedyFacet)$.
By \cref{lem:rootFunctionFlips}~\eqref{lem:rootFunctionFlips2}, if $i\in \antiGreedyFacet$ is flippable then $\rootFunction{\antiGreedyFacet}{i}\in -\Inv(\omega)$.
Furthermore, if $i\in \antiGreedyFacet$ is not flippable then $\rootFunction{\antiGreedyFacet}{i}\in \Ninv(\omega)$.
Therefore, $\Roots(\antiGreedyFacet)\subseteq \omega(\Phi^+)=-\Inv(\omega)\sqcup \Ninv(\omega)$.
\end{proof}


\subsection{Two theorems on linear extensions of facets}
\label{subsec:twoTheorems}

In this section, we present our two main results about linear extensions for subword complexes (\cref{thm:linearExtensionsPartitionSubwordComplexA,thm:linearExtensionsPartitionSubwordComplexB}), extending the results of \cref{subsec:linearExtensions}.

A subset $\fR$ of the weak order is 
\begin{itemize}
\item a \defn{lower set} if $\sigma < \tau$ and $\tau \in \fR$ implies~$\sigma \in \fR$,
\item \defn{order convex} if $\sigma < \tau < \rho$ and~$\sigma, \rho \in \fR$ implies~$\tau \in \fR$.
\end{itemize}

\begin{theorem}
\label{thm:linearExtensionsPartitionSubwordComplexA}
Let~$\subwordComplex(Q,\omega)$ be a non-empty subword complex.
Then
\begin{enumerate}
\item For any facet~$I$ of~$\subwordComplex(Q,\omega)$ set $\linearExtensions(I)$ is order convex.
\label{item:convex}
\hfill (convex) 
\item $\linearExtensions(Q,\omega)$ is a lower set of the weak order.
\label{item:lowerSet}
\hfill (lower set)
\item $[e,\omega] \subseteq \linearExtensions(Q,\omega).$
\label{item:cover}
\hfill (cover)
\item If $I_1\neq I_2$ then $\linearExtensions(I_1)\cap \linearExtensions(I_2)=\varnothing$.
\label{item:partition}
\hfill (partition)
\end{enumerate}
\end{theorem}

\begin{example}
We have represented five examples in \cref{fig:subwordComplex1,fig:subwordComplex2,fig:subwordComplex3,fig:subwordComplex4,fig:subwordComplex5}.
The last three use~$\omega = \wo$ and are spherical subword complexes (their brick polytopes are represented in \cref{fig:brickPolytope1,fig:brickPolytope2,fig:brickPolytope3}), while the first two do not.
We use the representation of the facets as pseudoline arrangements on sorting networks, see \cref{exm:sortingNetworks}.
\begin{figure}[p]
	\centerline{\includegraphics[scale=.7]{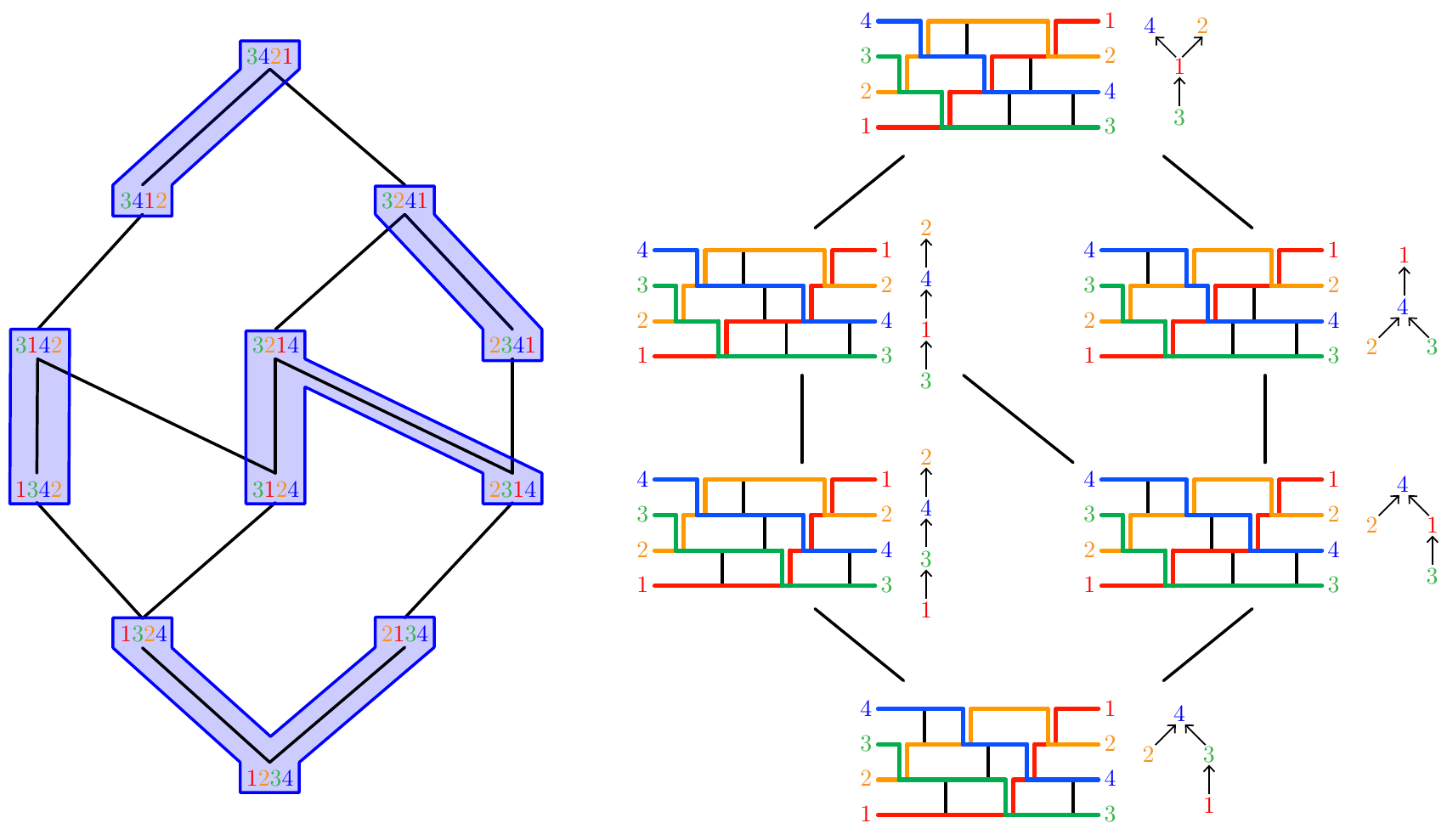}}
	\caption{The subword complex equivalence (left) and the acyclic increasing flip graph (right) for~$Q = \tau_2 \tau_3 \tau_1 \tau_3 \tau_2 \tau_1 \tau_2 \tau_3 \tau_1$ and~$w = \tau_2 \tau_3 \tau_1 \tau_2 \tau_3 = 3421$ in type~$A_3$. Note that the sets~$\linearExtensions(I)$ (blue bulles) are not always intervals.}
	\label{fig:subwordComplex4}
\end{figure}
\begin{figure}[p]
	\centerline{\includegraphics[scale=.7]{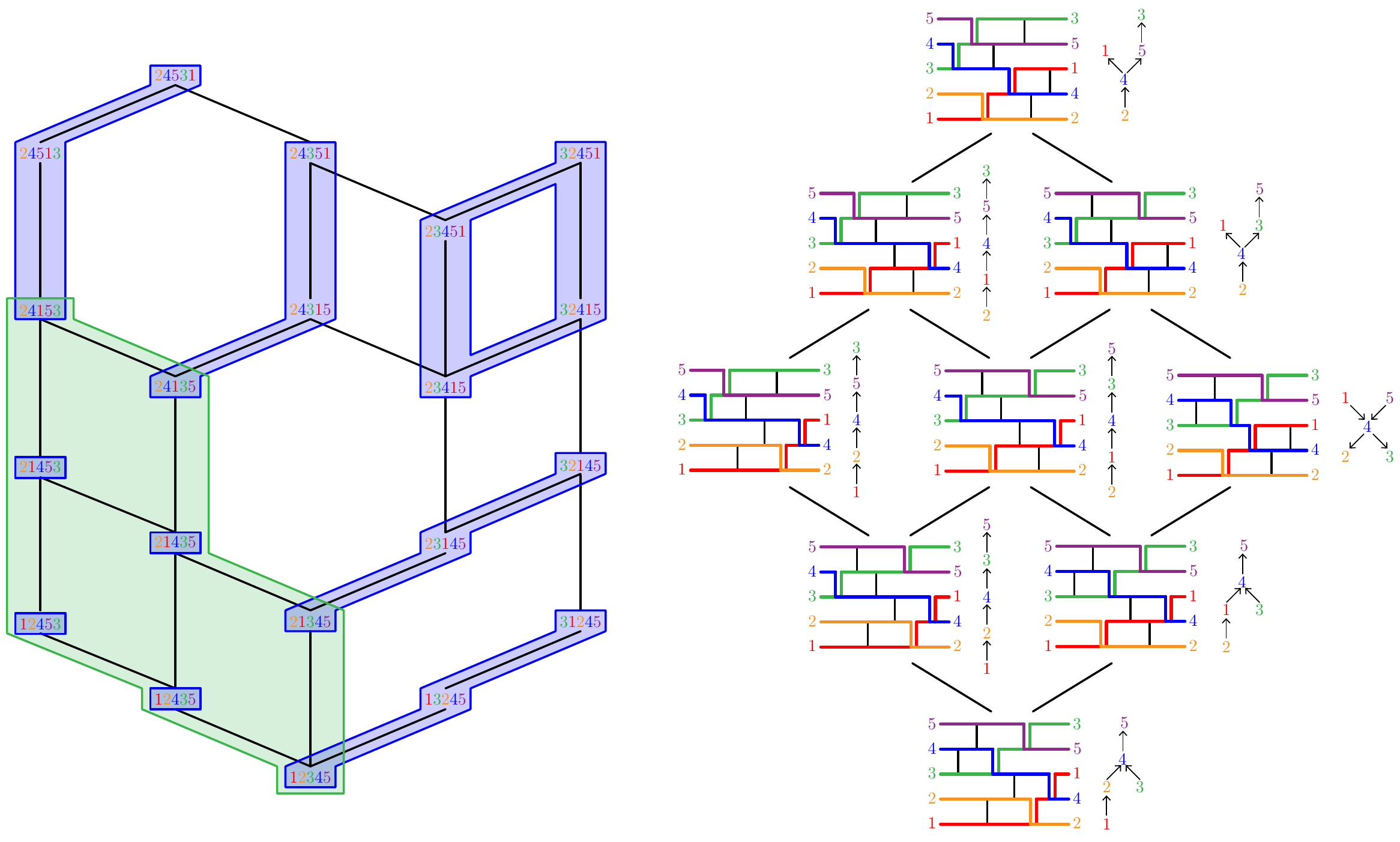}}
	\caption{The subword complex equivalence (left) and the acyclic increasing flip graph (right) for~$Q = \tau_3 \tau_4 \tau_1 \tau_3 \tau_2 \tau_4 \tau_1 \tau_2$ and~$w = \tau_3 \tau_4 \tau_1 \tau_2 = 24153$ in type~$A_4$. Note that the sets~$\linearExtensions(I)$ (blue bubbles) are not always subsets of, and may even not meet~$[e,\omega]$ (green bubble).}
	\label{fig:subwordComplex5}
\end{figure}
\hvFloat[floatPos=p, capWidth=h, capPos=right, capAngle=90, objectAngle=90, capVPos=center, objectPos=center]{figure}
{\includegraphics[scale=.55]{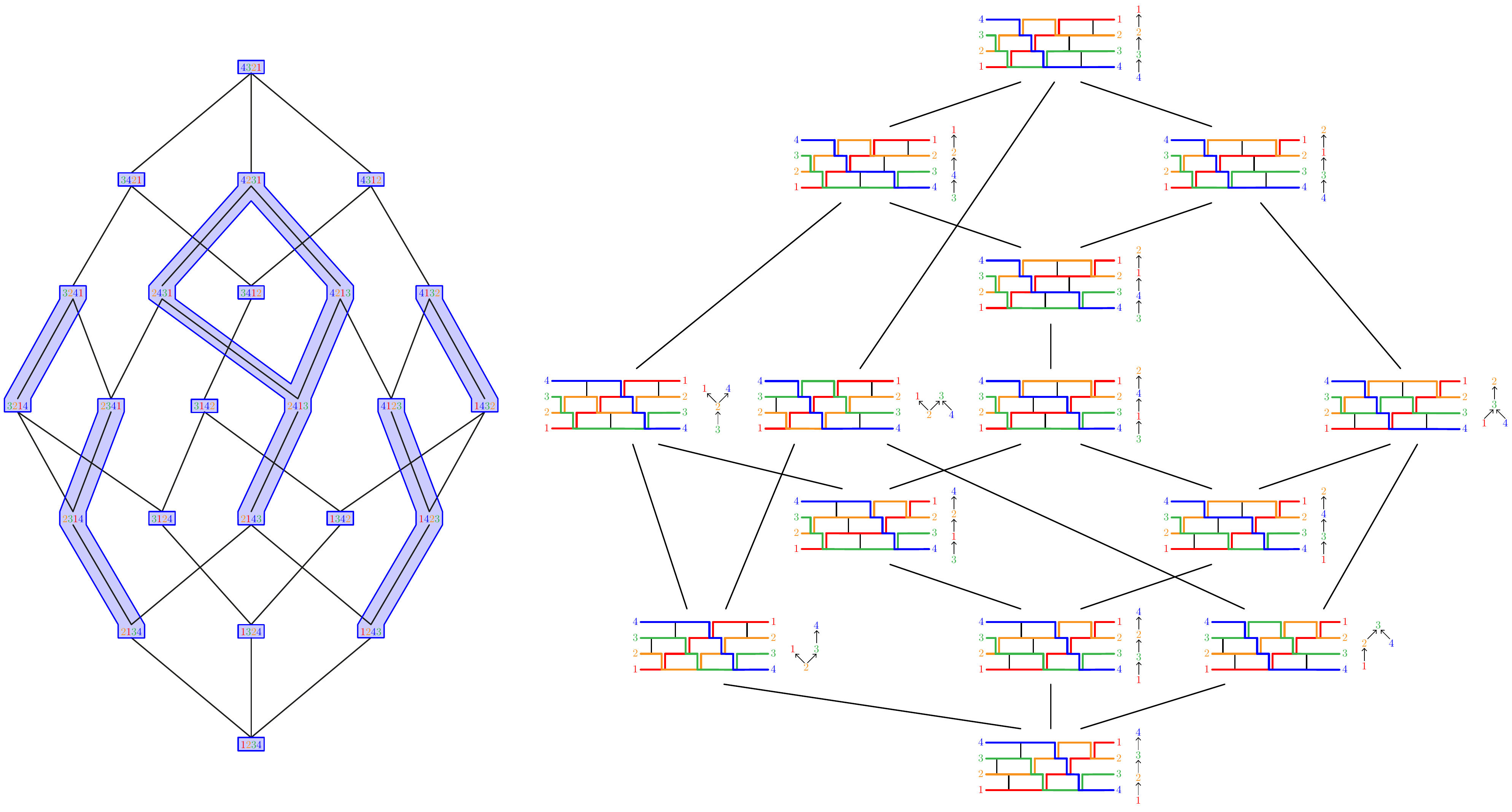}}
{For~$Q = \tau_2 \tau_1 \tau_3 \tau_2 \tau_1 \tau_3 \tau_2 \tau_1 \tau_3$ and~$\omega = \wo = 4321$ in type~$A_3$, the subword complex equivalence is a lattice congruence (left) so that the increasing flip poset is a lattice (right). See also~\cref{fig:increasingFlipPosets}\,(left) It is actually the $\tau_2\tau_1\tau_3$-Cambrian lattice in type~$A_3$. The blue bubbles represent the classes of the pipe dream congruence. For each facet of the subword complex, we have represented the corresponding pseudoline arrangement on the sorting network, and the corresponding contact graph. Note that all facets are acyclic in this case.}
{fig:subwordComplex1}
\hvFloat[floatPos=p, capWidth=h, capPos=right, capAngle=90, objectAngle=90, capVPos=center, objectPos=center]{figure}
{\includegraphics[scale=.65]{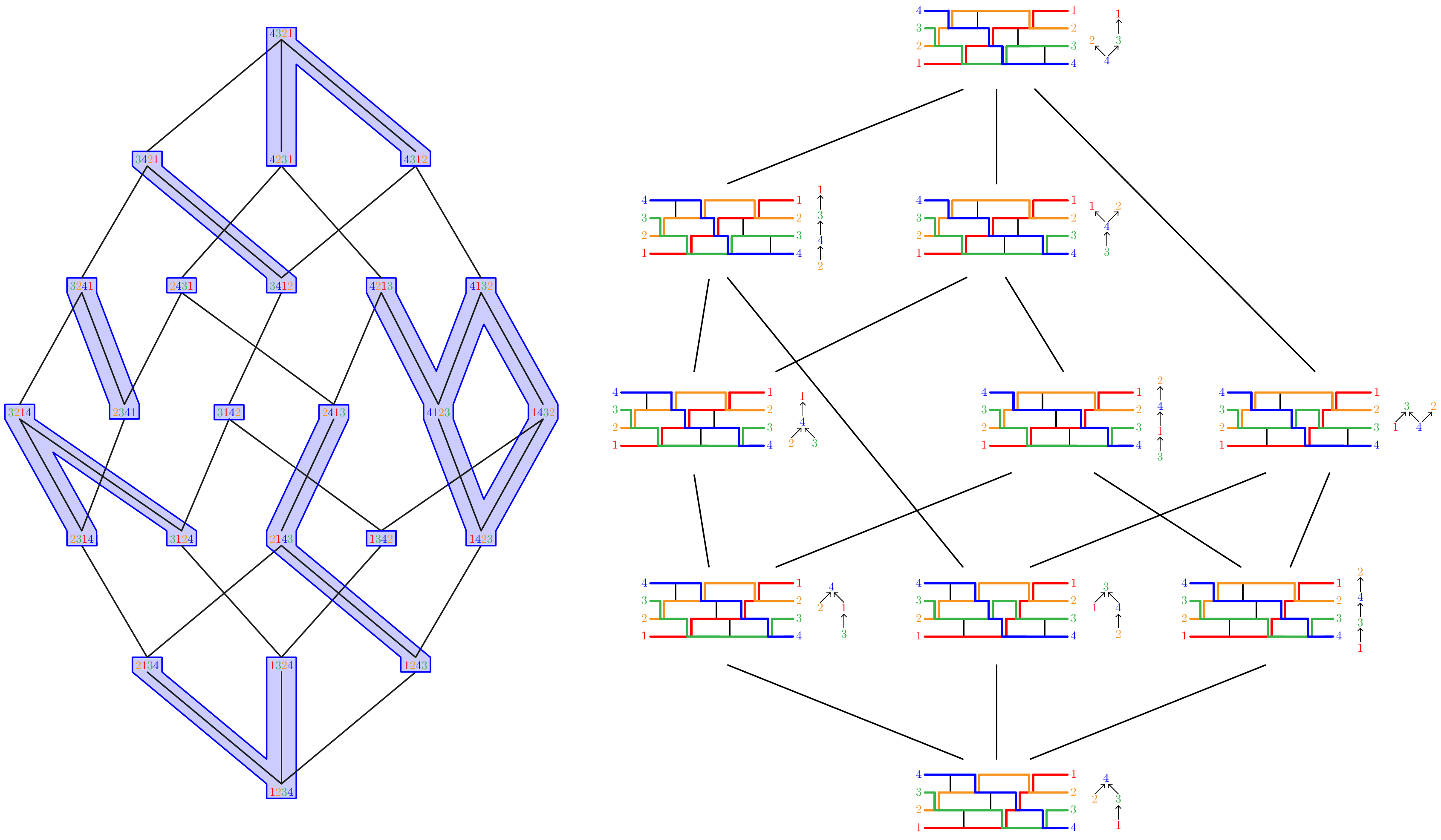}}
{For~$Q = \tau_2 \tau_3 \tau_1 \tau_3 \tau_2 \tau_1 \tau_2 \tau_3 \tau_1$ and~$\omega = \wo = 4321$ in type~$A_3$, the subword complex equivalence is not a lattice congruence (left) but the increasing flip poset is a lattice (right). The blue bubbles represent the classes of the pipe dream congruence. For each facet of the subword complex, we have represented the corresponding pseudoline arrangement on the sorting network, and the corresponding contact graph. Note that all facets are acyclic in this case.}
{fig:subwordComplex2}
\hvFloat[floatPos=p, capWidth=h, capPos=right, capAngle=90, objectAngle=90, capVPos=center, objectPos=center]{figure}
{\includegraphics[scale=.65]{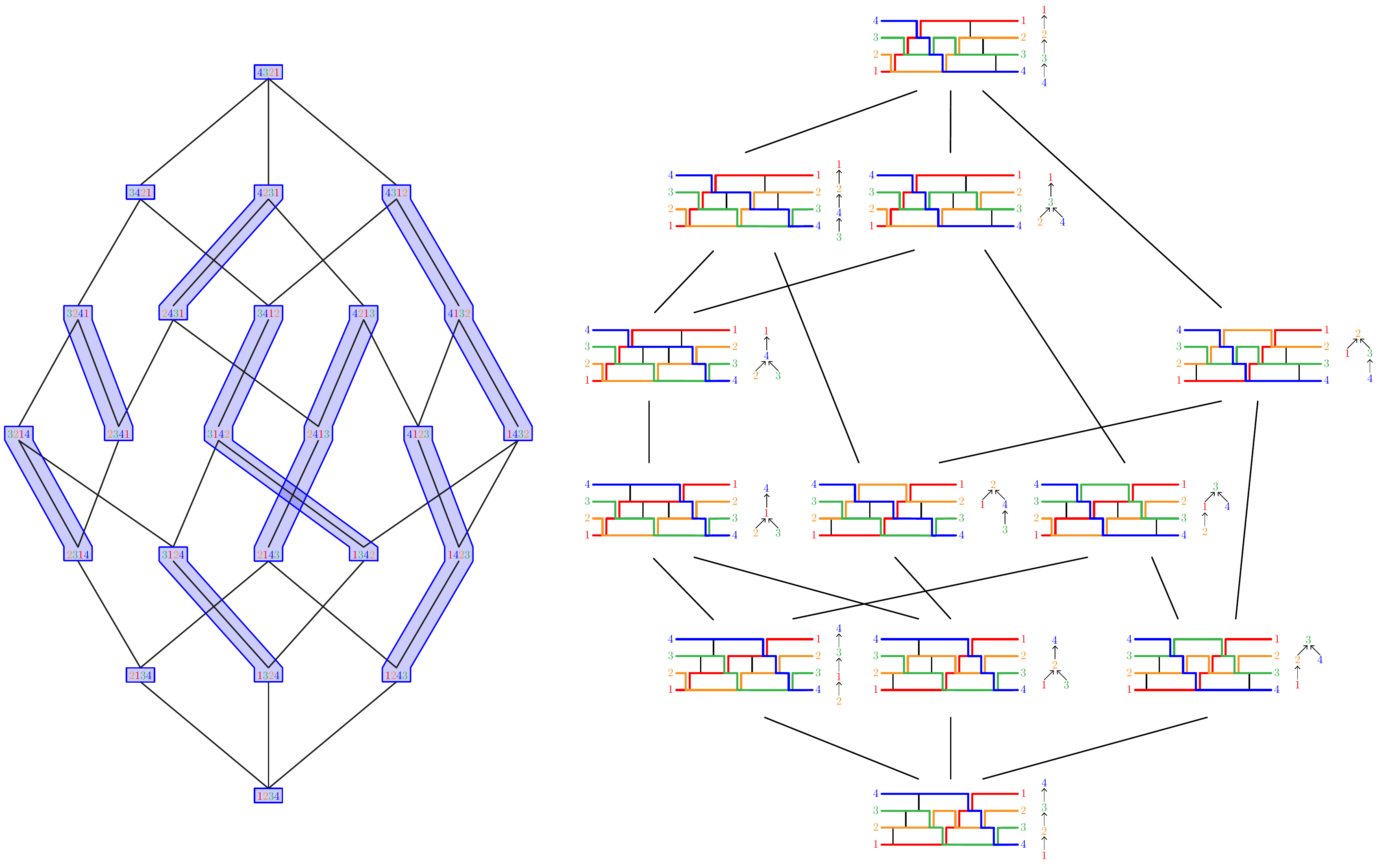}}
{For~$Q = \tau_1 \tau_2 \tau_3 \tau_2 \tau_1 \tau_2 \tau_3 \tau_2 \tau_1$ and~$\omega = \wo = 4321$ in type~$A_3$, the subword complex equivalence is not a lattice congruence (left) and the increasing flip poset is not a lattice (right). See also~\cref{fig:increasingFlipPosets}\,(right). The blue bubbles represent the classes of the pipe dream congruence. For each facet of the subword complex, we have represented the corresponding pseudoline arrangement on the sorting network, and the corresponding contact graph. Note that all facets are acyclic in this case.}
{fig:subwordComplex3}
\end{example}

\begin{remark}
\label{rem:linearExtensionsPartitionSubwordComplexA}
Before proving it, we would like to make the following remarks about theorem.
\begin{enumerate}
\item While the sets~$\linearExtensions(I)$ are order convex, they may not be intervals. See \cref{fig:subwordComplex2,fig:subwordComplex4} (the example of \cref{fig:subwordComplex2} is borrowed from~\cite[Figure~9]{PilaudStump-brickPolytope}).
\item The lower set~$\linearExtensions(Q,\omega)$ may have more than one maximal element, and~$\omega$ is not necessarily maximal in~$\linearExtensions(Q,\omega)$. See \cref{fig:subwordComplex5}.
\item Some of the sets~$\linearExtensions(I)$ may not be included in, nor even meet~$[e,\omega]$. See \cref{fig:subwordComplex5}.
\end{enumerate}
\end{remark}

\begin{remark}
We note that the set~$\linearExtensions(Q,\omega)$ of linear extensions has been considered independently in the work on brick polyhedra by D.~Jahn and C.~Stump in~\cite{JahnStump}. In particular,~$\linearExtensions(Q,\omega)$ is completely characterized using Bruhat cones in~\cite[Prop.~4.12]{JahnStump}, and the convex property~\eqref{item:convex} is equivalent to~\cite[Lem.~4.13]{JahnStump}. The other three properties also follow from their characterization, see~\cite[Sect.~4.2]{JahnStump}.  
Since we found these results independently, and the techniques we use to prove them are rather different, we believe that it is relevant to keep our contributions here. Moreover, this leads to further perspectives and conjectures in connection to brick polyhedra which we present in~\cref{subsec:fiveConjectures}. 
\end{remark}

We will now prove the four points of \cref{thm:linearExtensionsPartitionSubwordComplexA} one by one.

\begin{proof}[Proof of \cref{thm:linearExtensionsPartitionSubwordComplexA}~\eqref{item:convex}]
Let $\sigma < \tau < \rho$ be such that~$\sigma, \rho \in \linearExtensions(I)$.
By definition of linear extensions, we have
\[
\Roots(I) \subseteq \sigma(\Phi^+)=-\Inv(\sigma) \sqcup \Ninv(\sigma)
\quad\text{and}\quad
\Roots(I) \subseteq \rho(\Phi^+)=-\Inv(\rho) \sqcup \Ninv(\rho).
\]
Restricting to the set of negative and positive roots, respectively, we deduce
\[
\Roots(I) \cap \Phi^- \subseteq -\Inv(\sigma) \subseteq -\Inv(\tau) \qquad 
\Roots(I) \cap \Phi^+ \subseteq \Ninv(\rho)  \subseteq \Ninv(\tau)
\]
since~$\sigma < \tau < \rho$. Therefore
\[
\Roots(I) \subseteq -\Inv(\tau) \sqcup \Ninv(\tau) = \tau(\Phi^+)
\]
and so $\tau \in \linearExtensions(I)$.
\end{proof}

\begin{proof}[Proof of \cref{thm:linearExtensionsPartitionSubwordComplexA} \eqref{item:lowerSet}]
Let $\pi \in \linearExtensions(I)$ for some facet $I$.
We need to show that if $\pi' < \pi$ then there exist another facet $I'$ such that $\pi' \in \linearExtensions(I')$.
It is enough to show this when $\pi = \pi's$ for some descent~$s$ of~$\pi$ (\ie some $s\in S$ such that $\ell(\pi') < \ell(\pi)$).
Define~$\beta=\pi'(\alpha_s)$ and observe that~$s_\beta$ preserves the set~$\pi'(\Phi^+) \ssm \{\beta\} = \pi(\Phi^+) \ssm \{-\beta\} = \pi(\Phi^+) \cap \pi'(\Phi^+)$.

By definition, $\pi \in \linearExtensions(I)$ if and only if $\Roots(I) \subseteq \pi(\Phi^+)$.
We now define a new facet $I'$ such that $\pi' \in \linearExtensions(I')$.
We distinguish two cases, depending on whether or not~$-\beta \in \Roots(I)$.

\medskip
\paragraph{\bf Case 1:} $-\beta \notin \Roots(I)$.

In this case, 
$\Roots(I)\subseteq \pi(\Phi^+) \ssm \{-\beta\} \subseteq \pi'(\Phi^+)$.
Taking $I'=I$, we have 
$\pi'\in \linearExtensions(I')$ 
as wanted.

\medskip
\paragraph{\bf Case 2:} $-\beta \in \Roots(I)$.

In this case, we need to remove $-\beta$ from the root configuration.
We will achieve this by flipping the position of the last $-\beta$ in $I$ to create a new facet $I'$.
This position is indeed flippable as we will argue now. 

Given a facet $I$ of a subword complex~$\subwordComplex(Q,\omega)$ and a positive root $\beta\in \Phi^+$, the restriction of the list of roots 
\[
\rootFunction{I}{1}, \rootFunction{I}{2}, \dots , \rootFunction{I}{m}.
\]
to the set $\{\beta,-\beta\}$ is of the form
\[
\begin{array}{cccc}
\beta, \dots , & \beta &, -\beta, \dots, &-\beta.\\
& i && j
\end{array}
\]
The sequence of $-\beta$'s could in principle be empty and so does the sequence of $\beta$'s.
But if there is a $-\beta$ in this list, then there should be at least one $\beta$ preceding it.
The position $i$ of the last $\beta$ is used in the reduced expression of $\omega$ in the complement of $I$, that is $i\notin I$.
The positions of the other $\beta$'s and $-\beta$'s all belong to~$I$, and can all be flipped to $i$ (see \cref{lem:rootFunctionFlips}~\eqref{lem:rootFunctionFlips2}).   
In particular, the position $j$ of the last $-\beta$ belongs to $I$, and it can be flipped to $i$ creating a new facet $I'=I\ssm \{j\} \cap \{i\}$.

Now, since $\beta\notin \pi(\Phi^+)$.
There must be only one $\beta$ in the list.
Otherwise, there would be at least one $\beta$ whose position belongs to the facet $I$.
This would imply that $\beta\in \Roots(I)$ and $\pi$ would not be a linear extension of $I$, which is a contradiction.
So, our restricted list corresponding to $I$ looks like
\[
\begin{array}{ccc}
  \beta &, -\beta, \dots, &-\beta.\\
      i && j
\end{array}
\]
By  \cref{lem:rootFunctionFlips}~\eqref{lem:rootFunctionFlips3}, we obtain that flipping $j$ to $i$ creates the new facet $I' = I \ssm \{j\} \cap \{i\}$ whose corresponding restricted list looks like
\[
\begin{array}{ccc}
  \beta &, \beta, \dots, &\beta.\\
      i && j
\end{array}    
\]

Moreover, since the reflection $s_\beta$ preserves the set 
\[
\pi'(\Phi^+) \ssm \{\beta\} = \pi(\Phi^+) \ssm \{-\beta\} = \pi(\Phi^+) \cap \pi'(\Phi^+).
\]
then $\Roots(I') \subseteq \pi'(\Phi^+)$.
Thus, $\pi'\in \linearExtensions(I')$ as desired.
\end{proof}

\begin{proof}[Proof of \cref{thm:linearExtensionsPartitionSubwordComplexA}~\eqref{item:cover}]
By part \eqref{item:lowerSet}, the union of all linear extensions of facets is a lower set.
So, we just need to show that $\omega$ belongs to this set.
This follows from $\omega\in \linearExtensions(\antiGreedyFacet)$, which was proven in \cref{lem:linearExtensionsGreedyFacets}.
\end{proof}

\begin{proof}[Proof of \cref{thm:linearExtensionsPartitionSubwordComplexA}~\eqref{item:partition}]
We show that if there is two facets~$I_1, I_2$ of~$\subwordComplex(Q, \omega)$ and an element~$\pi \in W$ such that $\pi \in \linearExtensions(I_1) \cap \linearExtensions(I_2)$ then~$I_1= I_2$.
The proof works by induction on the length~$\ell(\pi)$ of~$\pi$.
We already showed this for $\pi = e$ in \cref{lem:linearExtensionsGreedyFacets}.

Let $I_1, I_2$ be two facets such that $e \neq \pi \in \linearExtensions(I_1)\cap \linearExtensions(I_2)$.
As in the proof of part \eqref{item:lowerSet}, let $\pi'=\pi s$ for some $s\in S$ such that $\ell(\pi')<\ell(\pi)$, and let $I_1',I_2'$ be the corresponding facets obtained using the same steps of the proof.
These new facets satisfy $\pi' \in \linearExtensions(I_1')\cap \linearExtensions(I_2')$, so that~$I_1' = I_2'$ by induction.

We now claim that it implies that~$I_1 = I_2$.
We analyze the two cases if the proof of part \eqref{item:lowerSet}.
Note that in Case~1, the resulting facet~$I'$ obtained from $I$ satisfies $\beta\notin \Roots(I')$, while in Case 2 we have $\beta \in \Roots(I')$. 
As $I_1' = I_2'$, this shows that $I_1$ and $I_2$ fall either both into Case 1 or both into Case 2.
If both fall into Case 1, then $I_1'=I_1$ and $I_2'=I_2$ and so $I_1=I_2$ as desired.
If both fall into Case 2, then we just need to flip back the performed flip to obtain $I_1=I_2$.
\end{proof}

Since $\linearExtensions(I)\neq \varnothing$ if and only if $I\in \subwordAcyclicFacets(Q,\omega)$ (\ie $I$ is an acyclic facet) by \cref{lem:linearExtensionsAcyclic}, and also~$\linearExtensions(I) \ne \linearExtensions(J)$ for~$I \ne J$ by \cref{thm:linearExtensionsPartitionSubwordComplexA}~\eqref{item:partition}, we have the following straightforward corollary.

\begin{corollary}
\label{coro:linearExtensionsPartitionSubwordComplexes}
For any word~$Q$ and element~$\omega$,
\[
\linearExtensions(Q,\omega) = \bigsqcup_{I\in\subwordAcyclicFacets(Q,\omega)} \linearExtensions(I).
\]
\end{corollary}

As pointed out in \cref{rem:linearExtensionsPartitionSubwordComplexA}, there are subword complexes for which $[e,\omega] \neq \linearExtensions(Q,\omega)$.
Our second fundamental theorem describes a large family of cases where equality holds.
We say that a word~$Q$ is \defn{sorting} if it contains a reduced expression of $\wo$.
Equivalently~$Q$ contains a reduced expression for any element~$w \in W$.
Still equivalently, $\DemazureProduct(Q) = \wo$ where~$\DemazureProduct(Q)$ denotes the \defn{Demazure product} of~$Q$, defined by~$\DemazureProduct(\varepsilon) = e$ and~$\DemazureProduct(Qs) = \max(\DemazureProduct(Q), \DemazureProduct(Q)s)$ (where the $\max$ is in weak order).

\begin{theorem}
\label{thm:linearExtensionsPartitionSubwordComplexB}
If the word $Q$ is sorting, then the linear extensions of acyclic facets of $\subwordComplex(Q,\omega)$ form a partition of the interval $[e,\omega]$, that is
\[
[e,\omega] = \bigsqcup_{I\in\subwordAcyclicFacets(Q,\omega)} \linearExtensions(I).
\]
\end{theorem}

The proof is based on the following statement, which follows from~\cite[Thm.~3.1 \& Coro.~3.3]{JahnStump}.

\begin{proposition}[{\cite{JahnStump}}]
\label{prop:intersectionRootConfigurations}
If the word $Q$ is sorting, then 
\[
\Big( \bigcap_{I\in \subwordFacets(Q,\omega)} \cone \Roots(I) \Big) \cap \Phi^+ = \Ninv(\omega).
\]
\end{proposition}

\begin{proof}
We include a short proof here for self containment.
We refer to~\cite{JahnStump} for the description of the notation $C^+(\omega,\cdot)$.
By~\cite[Theorem~3.1]{JahnStump} we have
\[
\bigcap_{I\in \subwordFacets(Q,\pi)} \cone \Roots(I) = C^+(\omega,\DemazureProduct(Q)).
\]
The word $Q$ contains a reduced expression of $\wo$ if and only if $\DemazureProduct(Q)=\wo$.
Furthermore, by~\cite[Corollary~3.3]{JahnStump} we have
\[
C^+(\omega,\wo)\cap \Phi^+ = \Inv(\omega \wo) = \Ninv(\omega).
\qedhere
\]
\end{proof}

\begin{proof}[Proof of \cref{thm:linearExtensionsPartitionSubwordComplexB}]
If $\pi \in \linearExtensions(I)$ for some facet~$I$ of~$\subwordComplex(Q, \omega)$, then by \cref{prop:intersectionRootConfigurations} we have 
\[
\Ninv(\omega) \subseteq \cone \Roots(I) \cap \Phi^+ \subseteq \cone \pi(\Phi^+) \cap \Phi^+  = \Ninv(\pi).
\]
Thus $\pi \le \omega$ as desired.
\end{proof}


\subsection{Five conjectures on linear extensions of facets}
\label{subsec:fiveConjectures}

In this section, we present conjectural generalizations of the results of \cref{subsec:pipeDreamCongruence,subsec:pipeDreamQuotient}.
By \cref{coro:linearExtensionsPartitionSubwordComplexes}, we have
\[
\linearExtensions(Q,\omega) = \bigsqcup_{I \in \subwordAcyclicFacets(Q, \omega)} \linearExtensions(I)
\]
which naturally defines an equivalence relation on~$\linearExtensions(Q,\omega)$.
However, this equivalence relation is in general not a lattice congruence for two obvious reasons:
\begin{enumerate}[(i)]
\item while~$\linearExtensions(Q,\omega)$ is a lower set of the weak order containing~$[e, \omega]$ by \cref{thm:linearExtensionsPartitionSubwordComplexA}, it does not always coincides with~$[e, \omega]$ by \cref{rem:linearExtensionsPartitionSubwordComplexA}, and in fact it does not necessarily have a maximal element, hence it is not necessarily a lattice,
\item while the sets~$\linearExtensions(I)$ are always order convex in the weak order by \cref{thm:linearExtensionsPartitionSubwordComplexA}~\eqref{item:convex}, they are not necessarily intervals in the weak order by \cref{rem:linearExtensionsPartitionSubwordComplexA}.
\end{enumerate}
To bypass Issue~(i), we could restrict our attention to the situation when the word~$Q$ is sorting by~\cref{thm:linearExtensionsPartitionSubwordComplexB} (we will do that in \cref{conj:latticeQuotientSortingAlternating}).
Here, we want to be slightly more general, so we will instead consider general words~$Q$, but restrict our attention to the interval~$[e, \omega]$ as follows.
For a facet~$I$ of~$\subwordComplex(Q, \omega)$, we define~$\strongLinearExtensions(I) \eqdef \linearExtensions(I) \cap [e, \omega]$.
We say that~$I$ is \defn{strongly acyclic} if~$\strongLinearExtensions(I) \ne \varnothing$, and we denote by~$\subwordStronglyAcyclicFacets(Q, \omega)$ the set of strongly acyclic facets of~$\subwordComplex(Q, \omega)$.
Note that by \cref{thm:linearExtensionsPartitionSubwordComplexA}~\eqref{item:partition}, we have
\[
[e, \omega] = \bigsqcup_{I\in\subwordStronglyAcyclicFacets(Q,\omega)} \strongLinearExtensions(I).
\]
We can thus now define the analogue of the pipe dream congruence of \cref{def:pipeDreamCongruence} for subword complexes as follows.

\begin{definition}
\label{def:pipeDreamCongruenceSubwordComplexes}
For a non-empty subword complex $\subwordComplex(Q,\omega)$, the \defn{subword complex equivalence} is the equivalence relation~$\equiv_{Q,\omega}$ on the interval~$[e, \omega]$  whose equivalence classes are the sets~$\strongLinearExtensions(I)$ for all strongly acyclic facets~$I$ of~$\subwordStronglyAcyclicFacets(Q, \omega)$.
In other words, $\pi \equiv_{Q,\omega} \pi'$ if and only if $\pi$ and $\pi'$ are linear extensions of the same facet.  
\end{definition}

\begin{example}
Observe for instance that:
\begin{itemize}
\item the subword complex equivalence is a lattice congruence of the weak order in \cref{fig:subwordComplex1} but not in \cref{fig:subwordComplex2,fig:subwordComplex3},
\item the increasing flip poset is a lattice in \cref{fig:subwordComplex1,fig:subwordComplex2} but not in \cref{fig:subwordComplex3}.
\end{itemize}
\end{example}

The subword complex equivalence is not always a congruence because of Issue~(ii) above.
To fix it, we now assume that the word~$Q$ is \defn{alternating}, meaning that all non-commuting pairs $s, t\in S$ alternate within $Q$ (this notion was already considered in \cite{PilaudSantos-brickPolytope, CeballosLabbeStump}).
This enables us to state our first conjecture, intended to extend~\cref{thm:pipeDreamCongruence}.

\begin{conjecture}
\label{conj:latticeCongruenceAlternating}
For a non-empty subword complex $\subwordComplex(Q,\omega)$ where~$Q$ is alternating, the subword complex equivalence~$\equiv_{Q,\omega}$ is a lattice congruence of the interval~$[e,\omega]$ of the weak order.
\end{conjecture}

We now intend to understand the quotient~$[e, \omega]/\equiv_{Q, \omega}$.
First, its elements correspond to the congruence classes of~$\equiv_{Q, \omega}$, hence to the strongly acyclic facets in~$\subwordStronglyAcyclicFacets(Q, \omega)$.
The cover relations are certain increasing flips between the facets in~$\subwordStronglyAcyclicFacets(Q, \omega)$.
However, in contrast to \cref{thm:pipeDreamQuotient}, not all increasing flips between two facets in~$\subwordStronglyAcyclicFacets(Q, \omega)$ yields a cover relation of~$\subwordStronglyAcyclicFacets(Q, \omega)$, as illustrated by the following example.

\begin{example}
\label{exm:acyclicFlipvsWeakOrder}
Consider the type~$A_2$ Coxeter system, the word~$Q = \tau_1 \tau_2 \tau_1 \tau_2 \tau_1 \tau_2$ and the longest element~$\wo = \tau_1 \tau_2 \tau_1 = \tau_2 \tau_1 \tau_2$.
The subword complex~$\subwordComplex(Q, \wo)$ has eight facets, six of which are acyclic.
All the fibers of $\equiv_{Q, \wo}$ are singletons and the lattice quotient~$[e, \wo]/\equiv_{Q, \wo}$ coincides with the weak order of type~$A_2$.
Note that the two acyclic facets~$\{1,3,4\}$ and $\{3,4,6\}$ are connected by a flip but the corresponding classes do not form a cover relation in~$[e, \wo]/\equiv_{Q, \wo}$.
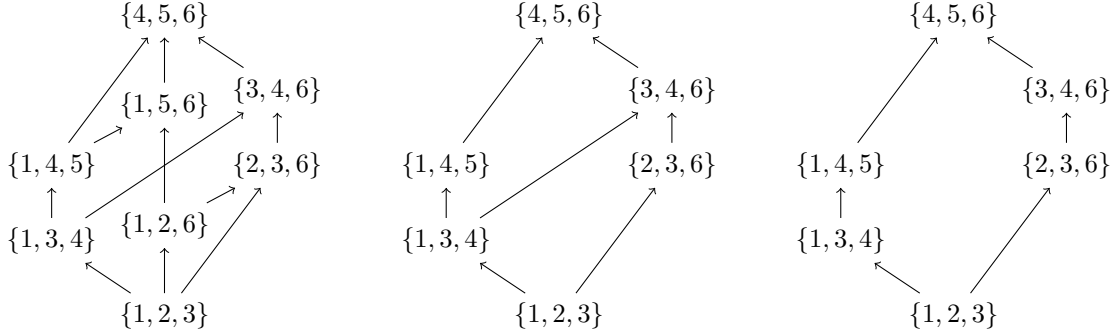
\begin{figure}[h]
	\begin{tikzpicture}[xscale=1.5,->]
	\node (123) at (0,-2) {$\{1,2,3\}$};
	\node (134) at (-1,-1) {$\{1,3,4\}$};
	\node (126) at (0,-.8) {$\{1,2,6\}$};
	\node (236) at (1,0) {$\{2,3,6\}$};
	\node (145) at (-1,0) {$\{1,4,5\}$};
	\node (156) at (0,.8) {$\{1,5,6\}$};
	\node (346) at (1,1) {$\{3,4,6\}$};
	\node (456) at (0,2) {$\{4,5,6\}$};
	\draw (123)--(134);
	\draw (123)--(126);
	\draw (123)--(236);
	\draw (134)--(145);
	\draw (134)--(346);
	\draw (126)--(156);
	\draw (126)--(236);
	\draw (236)--(346);
	\draw (145)--(456);
	\draw (145)--(156);
	\draw (156)--(456);
	\draw (346)--(456);
	\end{tikzpicture}
	\qquad
	\begin{tikzpicture}[xscale=1.5,->]
	\node (123) at (0,-2) {$\{1,2,3\}$};
	\node (134) at (-1,-1) {$\{1,3,4\}$};
	\node (236) at (1,0) {$\{2,3,6\}$};
	\node (145) at (-1,0) {$\{1,4,5\}$};
	\node (346) at (1,1) {$\{3,4,6\}$};
	\node (456) at (0,2) {$\{4,5,6\}$};
	\draw (123)--(134);
	\draw (123)--(236);
	\draw (134)--(145);
	\draw (134)--(346);
	\draw (236)--(346);
	\draw (145)--(456);
	\draw (346)--(456);
	\end{tikzpicture}
	\qquad
	\begin{tikzpicture}[xscale=1.5,->]
	\node (123) at (0,-2) {$\{1,2,3\}$};
	\node (134) at (-1,-1) {$\{1,3,4\}$};
	\node (236) at (1,0) {$\{2,3,6\}$};
	\node (145) at (-1,0) {$\{1,4,5\}$};
	\node (346) at (1,1) {$\{3,4,6\}$};
	\node (456) at (0,2) {$\{4,5,6\}$};
	\draw (123)--(134);
	\draw (123)--(236);
	\draw (134)--(145);
	\draw (236)--(346);
	\draw (145)--(456);
	\draw (346)--(456);
	\end{tikzpicture}
	\caption{The increasing flip graph on~$\subwordComplex(Q, \wo)$ (left), its restriction to acyclic facets (middle) and the Hasse diagram of the quotient~$[e, \wo]/\equiv_{Q, \wo}$ where each class is labeled by its corresponding facet of~$\subwordComplex(Q, \wo)$ (right), for the subword complex~$\subwordComplex(Q, \wo)$ where~$Q = \tau_1 \tau_2 \tau_1 \tau_2 \tau_1 \tau_2$ and~$\wo = \tau_1 \tau_2 \tau_1 = \tau_2 \tau_1 \tau_2$. Note that the increasing flip~$\{1,3,4\} \to \{3,4,6\}$ is not a cover relation of~$[e, \wo]/\equiv_{Q, \wo}$.}
	\label{fig:acyclicFlipvsWeakOrder}
\end{figure}
\end{example}

We say that a flip between two facets~$I, J \in \subwordStronglyAcyclicFacets(Q, \omega)$ with~$I \ssm \{i\} = J \ssm \{j\}$ is \defn{extremal} if the root~$\rootFunction{I}{i}$ is a ray of the root configuration~$\Roots(I)$ (or equivalently, $\rootFunction{J}{j}$ is a ray~$\Roots(J)$).
This enables us to state our second conjecture, intended to extend~\cref{thm:pipeDreamQuotient}.

\begin{conjecture}
\label{conj:latticeQuotientAlternating}
For a non-empty subword complex $\subwordComplex(Q,\omega)$ where~$Q$ is alternating, the Hasse diagram of the lattice quotient~$[e, \omega]/\equiv_{Q, \omega}$ is isomorphic to the graph of extremal increasing flips between strongly acyclic facets of~$\subwordStronglyAcyclicFacets(Q, \omega)$.
\end{conjecture}

We now specialize \cref{conj:latticeCongruenceAlternating,conj:latticeQuotientAlternating} to the case of sorting and alternating words.
In this case, all acyclic facets are strongly acyclic by \cref{thm:linearExtensionsPartitionSubwordComplexB}.
Note that this is precisely the situation we had in~\cref{sec:latticeAcyclicPipeDreams}.

\begin{conjecture}
\label{conj:latticeQuotientSortingAlternating}
If~$Q$ is sorting and alternating, then the Hasse diagram of the lattice quotient~$[e, \omega]/\equiv_{Q, \omega}$ is isomorphic to the graph of extremal increasing flips between acyclic facets of~$\subwordAcyclicFacets(Q, \omega)$.
\end{conjecture}

Moreover, there is a close connection with the brick polyhedra introduced in~\cite{JahnStump} as generalizations of the brick polytopes of~\cite{PilaudSantos-brickPolytope, PilaudStump-brickPolytope}.
We refer to the original papers~\cite{PilaudSantos-brickPolytope, PilaudStump-brickPolytope, JahnStump} for a definition of these polyhedra.
We just need to know here that the brick polyhedron~$\brickPolyhedron(Q, \omega)$~has
\begin{itemize}
\item a vertex for each acyclic facet of~$\subwordAcyclicFacets(Q, \omega)$, and
\item an edge for each extremal flip between two acyclic facets,
\end{itemize}
and that the graph of extremal increasing flips on acyclic facets is isomorphic to the bounded graph (meaning forgetting the unbounded rays) of the brick polyhedron~$\brickPolyhedron(Q, \omega)$ oriented in a suitable direction~$\delta$.
This can be derived from \cite[Thm.~4.4]{JahnStump}.
\cref{conj:latticeQuotientSortingAlternating} can thus be translated geometrically as follows.

\begin{conjecture}
\label{conj:latticeQuotientSortingAlternatingBrickPolyhedra}
If~$Q$ is sorting and alternating, then the bounded oriented graph of the brick polyhedron~$\brickPolyhedron(Q, \omega)$ is isomorphic to the Hasse diagram of the lattice quotient~$[e, \omega]/\equiv_{Q, \omega}$.
\end{conjecture}

In particular, specializing this conjecture to the brick polytopes~\cite{PilaudSantos-brickPolytope, PilaudStump-brickPolytope} for which~$\omega = \wo$, we obtain our last conjecture, intended to extend the results of~\cite{Pilaud-brickAlgebra}.

\begin{conjecture}
\label{conj:latticeQuotientSortingAlternatingBrickPolytopes}
If~$Q$ is sorting and alternating, then the oriented graph of the brick polytope~$\brickPolyhedron(Q, \wo)$ is isomorphic to the Hasse diagram of a lattice quotient of the weak order.
\end{conjecture}

\begin{remark}
\cref{conj:latticeCongruenceAlternating,conj:latticeQuotientAlternating,conj:latticeQuotientSortingAlternating,conj:latticeQuotientSortingAlternatingBrickPolyhedra,conj:latticeQuotientSortingAlternatingBrickPolytopes} holds in type~$A_n$: our specific proof of~\cref{thm:pipeDreamCongruence,thm:pipeDreamQuotient} can be extended to arbitrary alternating words in type~$A_n$ as shown in~\cite{Cartier}.
They are also supported by computer experiments: we verified \cref{conj:latticeCongruenceAlternating,conj:latticeQuotientAlternating} for all alternating words of length at most~$\ell(\wo)$ (hence \cref{conj:latticeQuotientSortingAlternating,conj:latticeQuotientSortingAlternatingBrickPolyhedra,conj:latticeQuotientSortingAlternatingBrickPolytopes} for all alternating reduced expressions of~$\wo$) in types~$B_2, B_3, D_4$ and~$H_3$.
\end{remark}

\begin{example}
We have illustrated in \cref{fig:brickPolytope1,fig:brickPolytope2,fig:brickPolytope3} the brick polytopes of the subword complexes represented in \cref{fig:subwordComplex1,fig:subwordComplex2,fig:subwordComplex3}.
Note that the oriented graph (from bottom to top) defines a lattice in \cref{fig:brickPolytope1,fig:brickPolytope2} but not in \cref{fig:brickPolytope3}.
\begin{figure}[p]
	\centerline{\begin{tikzpicture}%
	[x={(-0.366215cm, -0.789554cm)},
	y={(0.235950cm, -0.590693cm)},
	z={(0.900119cm, -0.166391cm)},
	scale=.5300000,
	back/.style={very thin, opacity=0.5},
	edge/.style={color=blue, thick},
	facet/.style={fill=blue,fill opacity=0},
	vertex/.style={}]
%
%

\coordinate (1.33333, -3.77124, 3.26599) at (1.33333, -3.77124, 3.26599);
\coordinate (1.33333, -3.77124, 6.53197) at (1.33333, -3.77124, 6.53197);
\coordinate (1.33333, 1.88562, 9.79796) at (1.33333, 1.88562, 9.79796);
\coordinate (1.33333, 7.54247, -3.26599) at (1.33333, 7.54247, -3.26599);
\coordinate (1.33333, 7.54247, 6.53197) at (1.33333, 7.54247, 6.53197);
\coordinate (4.00000, -5.65685, 3.26599) at (4.00000, -5.65685, 3.26599);
\coordinate (4.00000, -5.65685, 6.53197) at (4.00000, -5.65685, 6.53197);
\coordinate (6.66667, -4.71405, 8.16497) at (6.66667, -4.71405, 8.16497);
\coordinate (6.66667, -1.88562, 9.79796) at (6.66667, -1.88562, 9.79796);
\coordinate (9.33333, -3.77124, 0.00000) at (9.33333, -3.77124, 0.00000);
\coordinate (9.33333, -3.77124, 6.53197) at (9.33333, -3.77124, 6.53197);
\coordinate (9.33333, -0.94281, 8.16497) at (9.33333, -0.94281, 8.16497);
\coordinate (9.33333, 1.88562, -3.26599) at (9.33333, 1.88562, -3.26599);
\coordinate (9.33333, 1.88562, 6.53197) at (9.33333, 1.88562, 6.53197);
\draw[edge,back] (1.33333, -3.77124, 3.26599) -- (1.33333, 7.54247, -3.26599);
\draw[edge,back] (1.33333, 7.54247, -3.26599) -- (1.33333, 7.54247, 6.53197);
\draw[edge,back] (1.33333, 7.54247, -3.26599) -- (9.33333, 1.88562, -3.26599);
\node[vertex] at (1.33333, 7.54247, -3.26599)     {};
\fill[facet] (9.33333, -3.77124, 6.53197) -- (6.66667, -4.71405, 8.16497) -- (4.00000, -5.65685, 6.53197) -- (4.00000, -5.65685, 3.26599) -- (9.33333, -3.77124, 0.00000) -- cycle {};
\fill[facet] (6.66667, -1.88562, 9.79796) -- (1.33333, 1.88562, 9.79796) -- (1.33333, -3.77124, 6.53197) -- (4.00000, -5.65685, 6.53197) -- (6.66667, -4.71405, 8.16497) -- cycle {};
\fill[facet] (4.00000, -5.65685, 6.53197) -- (1.33333, -3.77124, 6.53197) -- (1.33333, -3.77124, 3.26599) -- (4.00000, -5.65685, 3.26599) -- cycle {};
\fill[facet] (9.33333, -0.94281, 8.16497) -- (6.66667, -1.88562, 9.79796) -- (6.66667, -4.71405, 8.16497) -- (9.33333, -3.77124, 6.53197) -- cycle {};
\fill[facet] (9.33333, 1.88562, 6.53197) -- (1.33333, 7.54247, 6.53197) -- (1.33333, 1.88562, 9.79796) -- (6.66667, -1.88562, 9.79796) -- (9.33333, -0.94281, 8.16497) -- cycle {};
\fill[facet] (9.33333, 1.88562, 6.53197) -- (9.33333, -0.94281, 8.16497) -- (9.33333, -3.77124, 6.53197) -- (9.33333, -3.77124, 0.00000) -- (9.33333, 1.88562, -3.26599) -- cycle {};
\draw[edge] (1.33333, -3.77124, 3.26599) -- (1.33333, -3.77124, 6.53197);
\draw[edge] (1.33333, -3.77124, 3.26599) -- (4.00000, -5.65685, 3.26599);
\draw[edge] (1.33333, -3.77124, 6.53197) -- (1.33333, 1.88562, 9.79796);
\draw[edge] (1.33333, -3.77124, 6.53197) -- (4.00000, -5.65685, 6.53197);
\draw[edge] (1.33333, 1.88562, 9.79796) -- (1.33333, 7.54247, 6.53197);
\draw[edge] (1.33333, 1.88562, 9.79796) -- (6.66667, -1.88562, 9.79796);
\draw[edge] (1.33333, 7.54247, 6.53197) -- (9.33333, 1.88562, 6.53197);
\draw[edge] (4.00000, -5.65685, 3.26599) -- (4.00000, -5.65685, 6.53197);
\draw[edge] (4.00000, -5.65685, 3.26599) -- (9.33333, -3.77124, 0.00000);
\draw[edge] (4.00000, -5.65685, 6.53197) -- (6.66667, -4.71405, 8.16497);
\draw[edge] (6.66667, -4.71405, 8.16497) -- (6.66667, -1.88562, 9.79796);
\draw[edge] (6.66667, -4.71405, 8.16497) -- (9.33333, -3.77124, 6.53197);
\draw[edge] (6.66667, -1.88562, 9.79796) -- (9.33333, -0.94281, 8.16497);
\draw[edge] (9.33333, -3.77124, 0.00000) -- (9.33333, -3.77124, 6.53197);
\draw[edge] (9.33333, -3.77124, 0.00000) -- (9.33333, 1.88562, -3.26599);
\draw[edge] (9.33333, -3.77124, 6.53197) -- (9.33333, -0.94281, 8.16497);
\draw[edge] (9.33333, -0.94281, 8.16497) -- (9.33333, 1.88562, 6.53197);
\draw[edge] (9.33333, 1.88562, -3.26599) -- (9.33333, 1.88562, 6.53197);
\node[vertex] at (1.33333, -3.77124, 3.26599)     {};
\node[vertex] at (1.33333, -3.77124, 6.53197)     {};
\node[vertex] at (1.33333, 1.88562, 9.79796)     {};
\node[vertex] at (1.33333, 7.54247, 6.53197)     {};
\node[vertex] at (4.00000, -5.65685, 3.26599)     {};
\node[vertex] at (4.00000, -5.65685, 6.53197)     {};
\node[vertex] at (6.66667, -4.71405, 8.16497)     {};
\node[vertex] at (6.66667, -1.88562, 9.79796)     {};
\node[vertex] at (9.33333, -3.77124, 0.00000)     {};
\node[vertex] at (9.33333, -3.77124, 6.53197)     {};
\node[vertex] at (9.33333, -0.94281, 8.16497)     {};
\node[vertex] at (9.33333, 1.88562, -3.26599)     {};
\node[vertex] at (9.33333, 1.88562, 6.53197)     {};
\end{tikzpicture}}
	\caption{The brick polytope of~$\subwordComplex(Q,\wo)$ where~$Q = \tau_2 \tau_1 \tau_3 \tau_2 \tau_1 \tau_3 \tau_2 \tau_1 \tau_3$ in type~$A_3$. Its graph, oriented from bottom to top is the Hasse diagram of the lattice of \cref{fig:increasingFlipPosets}\,(left) and \cref{fig:subwordComplex1}. This is actually the $c$-associahedron and the $c$-Cambrian lattice for the Coxeter element~$c = \tau_2\tau_1\tau_3$ of type~$A_3$.}
	\label{fig:brickPolytope1}
\end{figure}
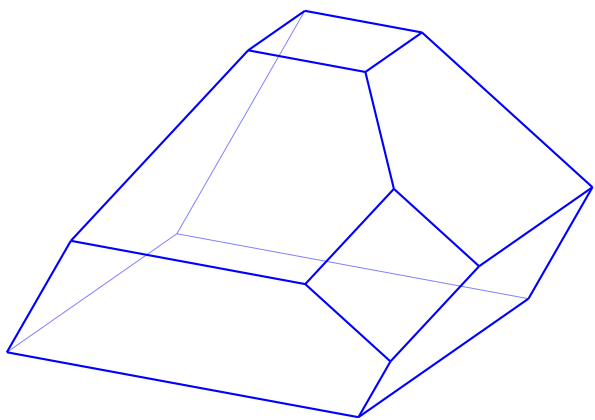
\begin{figure}[p]
	\centerline{\begin{tikzpicture}%
	[x={(-0.366215cm, -0.789554cm)},
	y={(0.235950cm, -0.590693cm)},
	z={(0.900119cm, -0.166391cm)},
	scale=.5300000,
	back/.style={very thin, opacity=0.5},
	edge/.style={color=blue, thick},
	facet/.style={fill=blue,fill opacity=0},
	vertex/.style={}]
%
%

\coordinate (4.00000, -5.65685, 3.26599) at (4.00000, -5.65685, 3.26599);
\coordinate (4.00000, -5.65685, 6.53197) at (4.00000, -5.65685, 6.53197);
\coordinate (4.00000, -2.82843, 4.89898) at (4.00000, -2.82843, 4.89898);
\coordinate (4.00000, -2.82843, 8.16497) at (4.00000, -2.82843, 8.16497);
\coordinate (6.66667, -4.71405, 8.16497) at (6.66667, -4.71405, 8.16497);
\coordinate (6.66667, -1.88562, 9.79796) at (6.66667, -1.88562, 9.79796);
\coordinate (9.33333, -3.77124, 0.00000) at (9.33333, -3.77124, 0.00000);
\coordinate (9.33333, -3.77124, 6.53197) at (9.33333, -3.77124, 6.53197);
\coordinate (9.33333, -0.94281, 1.63299) at (9.33333, -0.94281, 1.63299);
\coordinate (9.33333, -0.94281, 8.16497) at (9.33333, -0.94281, 8.16497);
\draw[edge,back] (4.00000, -5.65685, 3.26599) -- (4.00000, -2.82843, 4.89898);
\draw[edge,back] (4.00000, -2.82843, 4.89898) -- (4.00000, -2.82843, 8.16497);
\draw[edge,back] (4.00000, -2.82843, 4.89898) -- (9.33333, -0.94281, 1.63299);
\node[vertex] at (4.00000, -2.82843, 4.89898)     {};
\fill[facet] (6.66667, -1.88562, 9.79796) -- (4.00000, -2.82843, 8.16497) -- (4.00000, -5.65685, 6.53197) -- (6.66667, -4.71405, 8.16497) -- cycle {};
\fill[facet] (9.33333, -3.77124, 6.53197) -- (6.66667, -4.71405, 8.16497) -- (4.00000, -5.65685, 6.53197) -- (4.00000, -5.65685, 3.26599) -- (9.33333, -3.77124, 0.00000) -- cycle {};
\fill[facet] (9.33333, -0.94281, 8.16497) -- (6.66667, -1.88562, 9.79796) -- (6.66667, -4.71405, 8.16497) -- (9.33333, -3.77124, 6.53197) -- cycle {};
\fill[facet] (9.33333, -0.94281, 8.16497) -- (9.33333, -3.77124, 6.53197) -- (9.33333, -3.77124, 0.00000) -- (9.33333, -0.94281, 1.63299) -- cycle {};
\draw[edge] (4.00000, -5.65685, 3.26599) -- (4.00000, -5.65685, 6.53197);
\draw[edge] (4.00000, -5.65685, 3.26599) -- (9.33333, -3.77124, 0.00000);
\draw[edge] (4.00000, -5.65685, 6.53197) -- (4.00000, -2.82843, 8.16497);
\draw[edge] (4.00000, -5.65685, 6.53197) -- (6.66667, -4.71405, 8.16497);
\draw[edge] (4.00000, -2.82843, 8.16497) -- (6.66667, -1.88562, 9.79796);
\draw[edge] (6.66667, -4.71405, 8.16497) -- (6.66667, -1.88562, 9.79796);
\draw[edge] (6.66667, -4.71405, 8.16497) -- (9.33333, -3.77124, 6.53197);
\draw[edge] (6.66667, -1.88562, 9.79796) -- (9.33333, -0.94281, 8.16497);
\draw[edge] (9.33333, -3.77124, 0.00000) -- (9.33333, -3.77124, 6.53197);
\draw[edge] (9.33333, -3.77124, 0.00000) -- (9.33333, -0.94281, 1.63299);
\draw[edge] (9.33333, -3.77124, 6.53197) -- (9.33333, -0.94281, 8.16497);
\draw[edge] (9.33333, -0.94281, 1.63299) -- (9.33333, -0.94281, 8.16497);
\node[vertex] at (4.00000, -5.65685, 3.26599)     {};
\node[vertex] at (4.00000, -5.65685, 6.53197)     {};
\node[vertex] at (4.00000, -2.82843, 8.16497)     {};
\node[vertex] at (6.66667, -4.71405, 8.16497)     {};
\node[vertex] at (6.66667, -1.88562, 9.79796)     {};
\node[vertex] at (9.33333, -3.77124, 0.00000)     {};
\node[vertex] at (9.33333, -3.77124, 6.53197)     {};
\node[vertex] at (9.33333, -0.94281, 1.63299)     {};
\node[vertex] at (9.33333, -0.94281, 8.16497)     {};
\end{tikzpicture}}
	\caption{The brick polytope of~$\subwordComplex(Q,\wo)$ where~$Q = \tau_2 \tau_3 \tau_1 \tau_3 \tau_2 \tau_1 \tau_2 \tau_3 \tau_1$ in type~$A_3$. Its graph, oriented from bottom to top is the Hasse diagram of the lattice of \cref{fig:subwordComplex2}, but it is not obtained as a lattice quotient of the weak order.}
	\label{fig:brickPolytope2}
\end{figure}
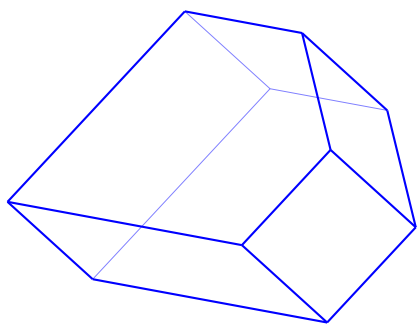
\begin{figure}[p]
	\centerline{\begin{tikzpicture}%
	[x={(-0.366215cm, -0.789554cm)},
	y={(0.235950cm, -0.590693cm)},
	z={(0.900119cm, -0.166391cm)},
	scale=.530000,
	back/.style={very thin, opacity=0.5},
	edge/.style={color=blue, thick},
	facet/.style={fill=blue,fill opacity=0},
	vertex/.style={}]
%
%

\coordinate (-3.33333, 3.77124, 0.00000) at (-3.33333, 3.77124, 0.00000);
\coordinate (-3.33333, 3.77124, 6.53197) at (-3.33333, 3.77124, 6.53197);
\coordinate (-3.33333, 6.59966, -1.63299) at (-3.33333, 6.59966, -1.63299);
\coordinate (-3.33333, 6.59966, 8.16497) at (-3.33333, 6.59966, 8.16497);
\coordinate (-0.66667, 7.54247, 9.79796) at (-0.66667, 7.54247, 9.79796);
\coordinate (2.00000, 8.48528, -4.89898) at (2.00000, 8.48528, -4.89898);
\coordinate (2.00000, 8.48528, 8.16497) at (2.00000, 8.48528, 8.16497);
\coordinate (7.33333, -3.77124, 0.00000) at (7.33333, -3.77124, 0.00000);
\coordinate (7.33333, -3.77124, 6.53197) at (7.33333, -3.77124, 6.53197);
\coordinate (7.33333, 1.88562, 9.79796) at (7.33333, 1.88562, 9.79796);
\coordinate (7.33333, 4.71405, -4.89898) at (7.33333, 4.71405, -4.89898);
\coordinate (7.33333, 4.71405, 8.16497) at (7.33333, 4.71405, 8.16497);
\draw[edge,back] (-3.33333, 3.77124, 0.00000) -- (-3.33333, 6.59966, -1.63299);
\draw[edge,back] (-3.33333, 6.59966, -1.63299) -- (-3.33333, 6.59966, 8.16497);
\draw[edge,back] (-3.33333, 6.59966, -1.63299) -- (2.00000, 8.48528, -4.89898);
\draw[edge,back] (2.00000, 8.48528, -4.89898) -- (2.00000, 8.48528, 8.16497);
\draw[edge,back] (2.00000, 8.48528, -4.89898) -- (7.33333, 4.71405, -4.89898);
\node[vertex] at (-3.33333, 6.59966, -1.63299)     {\color{red} $\bullet$};
\node[vertex] at (2.00000, 8.48528, -4.89898)     {};
\fill[facet] (7.33333, -3.77124, 6.53197) -- (-3.33333, 3.77124, 6.53197) -- (-3.33333, 3.77124, 0.00000) -- (7.33333, -3.77124, 0.00000) -- cycle {};
\fill[facet] (7.33333, 1.88562, 9.79796) -- (-0.66667, 7.54247, 9.79796) -- (-3.33333, 6.59966, 8.16497) -- (-3.33333, 3.77124, 6.53197) -- (7.33333, -3.77124, 6.53197) -- cycle {};
\fill[facet] (7.33333, 4.71405, 8.16497) -- (2.00000, 8.48528, 8.16497) -- (-0.66667, 7.54247, 9.79796) -- (7.33333, 1.88562, 9.79796) -- cycle {};
\fill[facet] (7.33333, 4.71405, 8.16497) -- (7.33333, 1.88562, 9.79796) -- (7.33333, -3.77124, 6.53197) -- (7.33333, -3.77124, 0.00000) -- (7.33333, 4.71405, -4.89898) -- cycle {};
\draw[edge] (-3.33333, 3.77124, 0.00000) -- (-3.33333, 3.77124, 6.53197);
\draw[edge] (-3.33333, 3.77124, 0.00000) -- (7.33333, -3.77124, 0.00000);
\draw[edge] (-3.33333, 3.77124, 6.53197) -- (-3.33333, 6.59966, 8.16497);
\draw[edge] (-3.33333, 3.77124, 6.53197) -- (7.33333, -3.77124, 6.53197);
\draw[edge] (-3.33333, 6.59966, 8.16497) -- (-0.66667, 7.54247, 9.79796);
\draw[edge] (-0.66667, 7.54247, 9.79796) -- (2.00000, 8.48528, 8.16497);
\draw[edge] (-0.66667, 7.54247, 9.79796) -- (7.33333, 1.88562, 9.79796);
\draw[edge] (2.00000, 8.48528, 8.16497) -- (7.33333, 4.71405, 8.16497);
\draw[edge] (7.33333, -3.77124, 0.00000) -- (7.33333, -3.77124, 6.53197);
\draw[edge] (7.33333, -3.77124, 0.00000) -- (7.33333, 4.71405, -4.89898);
\draw[edge] (7.33333, -3.77124, 6.53197) -- (7.33333, 1.88562, 9.79796);
\draw[edge] (7.33333, 1.88562, 9.79796) -- (7.33333, 4.71405, 8.16497);
\draw[edge] (7.33333, 4.71405, -4.89898) -- (7.33333, 4.71405, 8.16497);
\node[vertex] at (-3.33333, 3.77124, 0.00000)     {};
\node[vertex] at (-3.33333, 3.77124, 6.53197)     {};
\node[vertex] at (-3.33333, 6.59966, 8.16497)     {};
\node[vertex] at (-0.66667, 7.54247, 9.79796)     {};
\node[vertex] at (2.00000, 8.48528, 8.16497)     {};
\node[vertex] at (7.33333, -3.77124, 0.00000)     {\color{red} $\bullet$};
\node[vertex] at (7.33333, -3.77124, 6.53197)     {};
\node[vertex] at (7.33333, 1.88562, 9.79796)     {\color{blue} $\bullet$};
\node[vertex] at (7.33333, 4.71405, -4.89898)     {\color{blue} $\bullet$};
\node[vertex] at (7.33333, 4.71405, 8.16497)     {};
\end{tikzpicture}}
	\caption{The brick polytope of~$\subwordComplex(Q,\wo)$ where~$Q = \tau_1 \tau_2 \tau_3 \tau_2 \tau_1 \tau_2 \tau_3 \tau_2 \tau_1$ in type~$A_3$. Its graph, oriented from bottom to top is the Hasse diagram of the poset of \cref{fig:increasingFlipPosets}\,(right) and \cref{fig:subwordComplex3}, which is not a lattice (the two blue vertices have no join while the two red vertices have no meet).}
	\label{fig:brickPolytope3}
\end{figure}
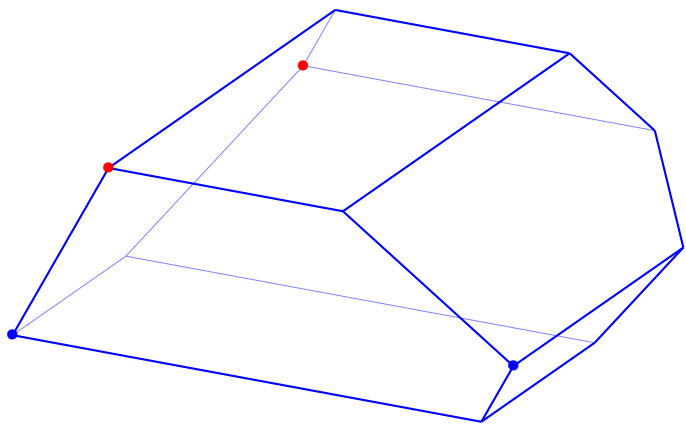
\end{example}


\subsection{Sweeping algorithm}
\label{subsec:sweepingAlgorithmSubwordComplexes}

We now extend the sweeping algorithm of \cref{subsec:sweepingAlgorithm} to construct, from a linear extension $\pi \in \linearExtensions(Q,\omega)$, the unique acyclic facet~$I$ such that $\pi \in \linearExtensions(I)$.
We note that this algorithm is related to an algorithm that was independently described in~\cite{JahnStump}, see~\cref{rem:sweepJahnStump}.

We start from a subword complex~$\subwordComplex(Q, \omega)$ and any element~$\pi \in W$.
The sweeping algorithm outcomes a facet for every element~$\pi \in W$, and this will be the desired facet for each $\pi \in \linearExtensions(Q,\omega)$.

It starts by setting $I^0=\varnothing$ to be the empty set.
It then scans the word $Q=(q_1,\dots ,q_m)$ from left to right.
At position $j\in [m]$ it produces a new set $I^j$ obtained from the previous one~$I^{j-1}$ by either adding $j$ or not, according to certain rules.
Roughly speaking, the goal of the algorithm is to insert a reduced expression of $\omega$ in $Q$ by sweeping the word from left to right, while deciding whether we use a letter or not by looking at the root $\beta=\rootFunction{I^j}{j}$ that it produces.
If~$\beta \in \Ninv(\omega)$ then it can not be used (Case~\ref{sweeping_case1}), otherwise we would not have a reduced expression for $\omega$.
If~$\beta \in \Inv(\omega)$ then it takes it when $\beta \in \Inv(\pi)$ (Case~\ref{sweeping_case2}), or delays it as much as possible when $\beta \in \Ninv(\pi)$ (Case~\ref{sweeping_case3}).
If $\beta \in -\Inv(\omega)$ it can simply not take it (Case~\ref{sweeping_case4}).

More precisely, let $Q^j=Q_{[j]\ssm I^j}$ and denote by
\[
\rootFunction{I^j}{k} = \prod Q_{[k-1]\ssm I^j}(\alpha_{q_k}).
\]
the partial root function for $k\leq j$.  
The rules are the following:

\begin{enumerate}[{Case} 1:]
    \item If $\rootFunction{I^j}{j}\in \Ninv(\omega)$ then 
    \(
    I^j=I^{j-1}\cup\{j\}.
    \) 
    \label{sweeping_case1}
    \item If $\rootFunction{I^j}{j}\in \Inv(\omega) \cap \Inv(\pi)$ then 
    \(
    I^j=I^{j-1}.
    \)
    \label{sweeping_case2}
    \item If $\rootFunction{I^j}{j}\in \Inv(\omega) \cap \Ninv(\pi)$ then we consider two cases: 
    \label{sweeping_case3} 
    \begin{enumerate}[\hspace{-1cm} (a)]
        \item If $Q_{[m]\ssm I^{j-1}\ssm \{j\}}$ contains a reduced expression of $\omega$ with prefix $Q^{j-1}$ then 
        \(
        {I^j=I^{j-1}\cup \{j\}}.
        \)~
        \label{sweeping_case3a}
        \item Otherwise
        \(
        I^j=I^{j-1}.
        \)
        \label{sweeping_case3b}
    \end{enumerate}
    \item If $\rootFunction{I^j}{j}\in -\Inv(\omega)$ then
    \(
    I^j=I^{j-1}\cup \{j\}.
    \)
    \label{sweeping_case4}
\end{enumerate}

We denote by $\sweepingAlgorithm(Q,\omega,\pi)$ the resulting set $I^m$ obtained at the last step of the algorithm.
The objective of this section is the following statement.

\begin{proposition}
\label{prop:sweeping}
Let $\subwordComplex(Q,\omega)$ be a non-empty subword complex. Then, for every $\pi \in W$, 
\begin{enumerate}
\item the set $\sweepingAlgorithm(Q,\omega,\pi)$ is a facet of $\subwordComplex(Q,\omega)$, \label{item:sweeping1}
\item if $\pi\in \linearExtensions(I)$ for some facet $I\in \subwordComplex(Q,\omega)$, then $\sweepingAlgorithm(Q,\omega,\pi) = I$. \label{item:sweeping2}
\end{enumerate}
\end{proposition}

\begin{remark}
\label{rem:sweepJahnStump}
In~\cite[Sect.~3.3]{JahnStump}, D.~Jahn and C.~Stump describe an algorithm to compute the $f$-antigreedy facet of a subword complex $\subwordComplex(Q,\omega)$ associated to a linear functional $f$.  
The output of our algorithm coincides with the output of their algorithm whenever $f$ is positive on~$\pi(\Phi^+)$ and negative on $\pi(\Phi^-)$, see~\cite[Prop.~4.12]{JahnStump}.
Also, compare \cref{prop:sweeping}~\eqref{item:sweeping1} with~\cite[Thm.~3.17 (a)]{JahnStump}, and \cref{prop:sweeping}~\eqref{item:sweeping2} with the second part of~\cite[Prop.~4.12]{JahnStump}.

Although the algorithm presented in~\cite{JahnStump} is more general, we highlight that our sweeping algorithm is conceptually simpler, since we skip the step of translating the conditions on the sign of $f(\rootFunction{I^j}{j})$ in the algorithm of~\cite{JahnStump}.
We remark that these two algorithms were developed independently, while approaching different problems.
This shows that the sweeping algorithm has a significant importance in the combinatorial and geometric understanding of subword complexes.  
\end{remark}

In order to prove \cref{prop:sweeping}, we first need to argue that the algorithm terminates, that is, that any position $j$ falls into one of the four cases above.

\begin{lemma}
\label{lem:sweeping1}
At step $j$ of the sweeping algorithm, the root~$\rootFunction{I^j}{j}$ belongs to exactly one of~the~sets
\[
\Ninv(\omega),
\qquad
\Inv(\omega) \cap \Inv(\pi),
\qquad
\Inv(\omega) \cap \Ninv(\pi)
\qquad\text{or}\qquad
-\Inv(\omega).
\]
\end{lemma}

\begin{proof}
The root $\rootFunction{I^j}{j}$ belongs to~$\Phi^+ \sqcup \Phi^-$ and we know that
\[
\Phi^+ = \Inv(\omega) \sqcup \Ninv(\omega) \qquad \qquad
\Phi^- = -\Inv(\omega) \sqcup -\Ninv(\omega)  
\]
Since we are inserting a reduced expression of $\omega$ in $Q$, the case $-\Ninv(\omega)$ never occurs.
\end{proof}

We now observe the main invariant of the sweeping algorithm.

\begin{lemma}
\label{lem:sweeping2}
At any step~$j$ of the algorithm, $Q^j$ is a reduced expression which is the prefix of a reduced expression of $\omega$ in $Q_{[m]\ssm I^j}$.
\end{lemma}

\begin{proof}
The proof works by induction on $j$.

For $j=0$, $Q^0$ is the empty word which is reduced by definition.
Furthermore, $Q_{[m]\ssm I^0 }=Q$ contains a reduced expression of $\omega$ because the subword complex is non-empty.
The empty word is a prefix of this reduced expression. 

We assume now that the claim holds for $j-1$ and we prove it for $j$.
Note that 
\[
Q^j = 
\begin{cases}
Q^{j-1}, & \text{if } I^j=I^{j-1}\cup \{j\} ,\\
Q^{j-1} q_j, & \text{if } I^j=I^{j-1}.
\end{cases}
\]
We analyze the different cases of the sweeping algorithm.
\begin{enumerate}
    \item If $\rootFunction{I^j}{j}\in \Ninv(\omega)$ then 
    $Q^j=Q^{j-1}$, which is reduced by induction hypothesis.
    Moreover, it is a prefix of a reduced expression of $\omega$ in $Q_{[m]\ssm I^{j-1}}$.
    Since $\rootFunction{I^j}{j}\in \Ninv(\omega)$, this reduced expression can not use the letter $q_j$; so, it is a reduced expression of $\omega$ in $Q_{[m]\ssm I^{j}}$ as desired.
    \item If $\rootFunction{I^j}{j}\in \Inv(\omega) \cap \Inv(\pi)$ then $Q^j= Q^{j-1}\circ (q_j)$, which is a reduced expression because~$Q^{j-1}$ is reduced and $\rootFunction{I^j}{j}\in \Inv(\omega)$.
    
    Now let $\widetilde Q$ be a subword of $Q_{[m]\ssm I^{j-1}}=Q_{[m]\ssm I^{j}}$ which is a reduced expression of $\omega$ with prefix $Q^{j-1}$, and let $\widetilde I$ be the corresponding facet.
    
    If $j\notin \widetilde I$, then $q_j$ is used in $\widetilde Q$ and we are done because $\widetilde Q$ has $Q^j$ as a prefix.

    If $j\in \widetilde I$, then we can flip it to a position $j'>j$ (by \cref{lem:rootFunctionFlips}~\eqref{lem:rootFunctionFlips2}), creating a new reduced expression $\widetilde Q'$ of $\omega$ which uses $q_j$, and thus has $Q^j$ as a prefix.

    \item[(3)(a)] If $\rootFunction{I^j}{j}\in \Inv(\omega) \cap \Ninv(\pi)$ and $Q_{[m]\ssm I^{j-1}\ssm \{j\}}$ contains a reduced expression of $\omega$ with prefix $Q^{j-1}$, then $Q^j=Q^{j-1}$.
    Therefore, $Q^j$ is reduced and it is a prefix of a reduced expression of $\omega$ in 
    $Q_{[m]\ssm I^{j-1}\ssm \{j\}}=Q_{[m]\ssm I^j}$.

    \item[(3)(b)] In this case, $\rootFunction{I^j}{j}\in \Inv(\omega) \cap \Ninv(\pi)$ but $Q^j= Q^{j-1}\circ (q_j)$. The proof is similar to the proof of (2).

    \item[(4)] If $\rootFunction{I^j}{j}\in -\Inv(\omega)$ then $Q^j=Q^{j-1}$. The argument is similar to that of Case~(1).
    \qedhere
    \end{enumerate}
\end{proof}

Finally, we need the following technical statement.

\begin{lemma}
\label{lem:sweeping3}
Let $I_1,I_2\in \subwordComplex(Q,\omega)$ be two different facets, and $j\in [m]$ be the first position where they differ.
Without loss of generality assume 
\[
I_2\cap [j] = I_1\cap [j] \ssm \{j\}
\]
with $j\in I_1$.
Let $\beta=\rootFunction{I_1}{j}=\rootFunction{I_2}{j}$, then 
\[
-\beta \in \cone \Roots(I_2).
\]
\end{lemma}

In the proof of \cref{lem:sweeping3}, let us write~$\leqslant_B$ the strong Bruhat order, defined by~$x \leqslant_B y$ if and only if a reduced expression of~$y$ has a reduced expression of~$x$ as a subword, and~$x \prec_B y$ the covers of this order. We note that for any word~$Q$ and any~$x\in W$, we can find a reduced expression of~$x$ as a subword of~$Q$ if and only if~$x \leqslant_B \DemazureProduct(Q)$ the Demazure product of~$Q$.
We will use \cite[Proposition 3.14]{JahnStump}, reformulated as follows.

\begin{proposition}[\cite{JahnStump}]\label{prop:JSDemazureCone}
Let~$\subwordComplex(Q,w)$ be a non-empty subword complex. Then for any facet~$I \in \subwordComplex(Q,w)$ and for any simple root~$\alpha$ such that~$w \prec_B s_\alpha w \leqslant_B \DemazureProduct(Q)$, we have~$\alpha \in \Roots(I)$.
\end{proposition}

\begin{proof}[Proof of \cref{lem:sweeping3}]
Let us write~$w = uv$ with~$u$ the prefix of~$w$ defined by~$I_1$ and~$I_2$ on~$Q$ until index~$j-1$ and~$v$ the suffix of~$w$ defined by the same facets from index~$j$ and onward. Let us also write~$s$ the~$j^{\text{th}}$ letter of~$Q$ and~$\alpha$ the simple root associated with~$s$. We note that~$\beta = \rootFunction{I_1}{j} = u(\alpha)$, and that since~$j \notin I_2$, we know that~$sv \prec_B v$ (this is also a cover of the left weak order).

Let us call~$Q'$ the suffix of~$Q$ starting at index~$j+1$, and~$I'_1,I'_2$ the restrictions of~$I_1$ and~$I_2$ to~$Q'$. The subwords induced by~$I'_1$ and~$I'_2$ on~$Q'$ are suffixes of reduced subwords of~$Q$, so they are also reduced. Moreover, since~$j \in I_1$, we know that~$I'_1$ is a facet of~$\subwordComplex(Q',v)$; this means that this subword complex is not empty and that~$v \leqslant_B \DemazureProduct(Q')$. Similarly, since~$j \notin I_2$, we know that~$I'_2$ is a facet of~$\subwordComplex(Q',sv)$. By combining the previous statements, we know that~$sv \prec_B v \leqslant_B \DemazureProduct(Q')$ (or equivalently and with~$v' = sv$, that~$v' \prec_B sv' \leqslant_B \DemazureProduct(Q')$). We can thus apply \cref{prop:JSDemazureCone} to obtain that~$\alpha \in \Roots(I'_2)$.

Going back to our facets on~$Q$, we know that the prefix of~$w$ written by~$I_2$ on~$Q \cap [j]$ is~$us$, and thus~$us(\Roots(I'_2)) \subseteq \Roots(I_2)$. Therefore, we have~$us(\alpha) = u(-\alpha) = -u(\alpha) = -\beta \in \Roots(I_2)$, thus concluding the proof.
\end{proof}

\begin{proof}[Proof of \cref{prop:sweeping}]
Since the sweeping algorithm terminates by \cref{lem:sweeping1}, Point~\eqref{item:sweeping1} follows directly from the invariant of the sweeping algorithm of \cref{lem:sweeping2} applied when~$j = m$.

For Point~\eqref{item:sweeping2}, assume $\pi \in \linearExtensions(I)$ and let $I^j=I\cap [j]$.
We will show that the partial root function~$\rootFunction{I^j}{\cdot}$ agrees with the decisions taken in the sweeping algorithm.
Indeed, we will see that those decisions are forced.

Recall that 
\begin{center}
\begin{tabular}{ccl}
    $\pi \in \linearExtensions(I)$ & $\longleftrightarrow$ & $R(I)\subseteq \pi(\Phi^+)$\\
     & $\longleftrightarrow$ & $\cone R(I)\subseteq \pi(\Phi^+)$
\end{tabular}    
\end{center}
and 
\[
\pi(\Phi^+) = -\Inv(\pi) \sqcup \Ninv(\pi).
\]
We analyze the possible cases in the sweeping algorithm.

\begin{enumerate}
    \item If $\rootFunction{I^j}{j}\in \Ninv(\omega)$ then clearly $j\in I$ is forced. Otherwise $Q_{[m]\ssm I}$ would not be a reduced expression of $\omega$.
    \item If $\rootFunction{I^j}{j}\in \Inv(\omega) \cap \Inv(\pi)$ then $j\notin I$ is forced. 
    Otherwise we would have an inversion of $\pi$ in the root configuration, which contradicts $\pi \in \linearExtensions(I)$.
    \item[(3)(a)] If $\rootFunction{I^j}{j}\in \Inv(\omega) \cap \Ninv(\pi)$ and $Q_{[m]\ssm I^{j-1}\ssm \{j\}}$ contains a reduced expression of $\omega$ with prefix $Q^{j-1}$ then $j\in I^j$ is forced. 

    We argue this by contradiction. 
    Assume $j\notin I^j$ ($j\notin I$). Let $I_1=\sweepingAlgorithm(Q,\omega,\pi)$ and $I_2=I$. Applying \cref{lem:sweeping3}, we deduce that $\beta=\rootFunction{I^j}{j}=\rootFunction{I}{j}$ satisfies 
    \[
    -\beta \in \cone \Roots(I).
    \]
    But $\beta \in \Ninv(\pi)$. This contradicts $\pi\in \linearExtensions(I)$.
    \item[(3)(b)]  If $\rootFunction{I^j}{j}\in \Inv(\omega) \cap \Ninv(\pi)$ and $Q_{[m]\ssm I^{j-1}\ssm \{j\}}$ does not contain a reduced expression of $\omega$ with prefix $Q^{j-1}$ then $j\notin I^j$ is forced. Otherwise, the complement of $I$ would not be a reduced expression of $\omega$.
    \item[(4)] If $\rootFunction{I^j}{j}\in -\Inv(\omega)$ then $j\in I^j$ is clearly forced. Otherwise, the complement of $I$ would not be a reduced expression.
    \qedhere
\end{enumerate}
\end{proof}

\begin{remark}
    Although the sweeping algorithm produces a facet $I=\sweepingAlgorithm(Q,\omega,\pi)$ for every~${\omega\in W}$, in some cases we have $\pi \notin \linearExtensions(I)$. 
    This happens because of Case~\eqref{sweeping_case1}, when a non-inversion~$\beta\in \Ninv(\omega)$ of $\omega$ is added to the root configuration $\Roots(I)$, such that~$\beta \notin \Ninv(\pi)$. This is only potentially possible when $\pi \notin [e,\omega]$.
\end{remark}


\section*{Acknowledgments}

We thank Florent Hivert for various discussions, Martin Rubey for his comments which led us to think about~\cref{prob:nuAcyclicProperty}, Lucas Gagnon for sharing his independent proof of the backward direction of \cref{prob:nuAcyclicProperty}, allowing us to greatly simplify our argument, and Nathan Reading for pointing out a serious mistake in the statement and proof of a previous version of \cref{lem:rectangleInsertionAlgorithm}.
We are grateful to an anonymous referee for comments and suggestions which greatly improved the presentation of this paper.


\bibliographystyle{alpha}
\bibliography{latticePipeDreams}
\label{sec:biblio}


\end{document}